\documentclass[11pt,reqno]{amsart}

\usepackage[a4paper,
            left=2cm,
            right=2cm,
            top=2cm,
            bottom=2cm]{geometry}
\usepackage{hyperref}
\usepackage{cleveref}
\usepackage{newtxtext}
\usepackage{newtxmath}
\usepackage{enumitem}
\usepackage{tikz}
\usetikzlibrary{cd}
\usepackage[mathscr]{euscript}
\usepackage{mathtools}
\usepackage{subcaption}

\usepackage{adjustbox}

\newcommand{\graf}{\mathsf{\Gamma}}
\newcommand{\grafd}{\mathsf{\Delta}}
\newcommand{\grafl}{\mathsf{\Lambda}}

\newcommand{\sfC}{\mathsf{C}}
\newcommand{\sfD}{\mathsf{D}}

\newcommand{\sfK}{\mathsf{K}}
\newcommand{\sfP}{\mathsf{P}}

\newcommand{\sfS}{\mathsf{St}}
\newcommand{\sfT}{\mathsf{T}}

\newcommand{\sfW}{\mathsf{W}}

\newcommand{\frs}{\mathfrak{s}}

\newcommand{\sfa}{\mathsf{a}}
\newcommand{\sfb}{\mathsf{b}}
\newcommand{\sfc}{\mathsf{c}}
\newcommand{\sfe}{\mathsf{e}}
\newcommand{\sff}{\mathsf{f}}
\newcommand{\sft}{\mathsf{t}}

\newcommand{\sfv}{\mathsf{v}}
\newcommand{\sfw}{\mathsf{w}}
\newcommand{\sfx}{\mathsf{x}}

\newcommand{\sfz}{\mathsf{z}}

\newcommand{\sfzero}{\mathsf{0}}

\newcommand{\bbA}{\mathbb{A}}
\newcommand{\bbB}{\mathbb{B}}
\newcommand{\bbF}{\mathbb{F}}
\newcommand{\bbG}{\mathbb{G}}
\newcommand{\bbR}{\mathbb{R}}
\newcommand{\bbS}{\mathbb{S}}
\newcommand{\bbZ}{\mathbb{Z}}
\newcommand{\bbH}{\mathbb{H}}

\newcommand{\cC}{\mathcal{C}}

\newcommand{\cF}{\mathcal{F}}
\newcommand{\cG}{\mathcal{G}}
\newcommand{\cI}{\mathcal{I}}
\newcommand{\cM}{\mathcal{M}}
\newcommand{\cP}{\mathcal{P}}
\newcommand{\cQ}{\mathcal{Q}}

\newcommand{\cS}{\mathcal{S}}

\newcommand{\cY}{\mathcal{Y}}

\newcommand{\bfzero}{\mathbf{0}}

\newcommand{\interior}[1]{\mathring{#1}}

\newcommand{\grapegraph}{\mathscr{G}\mathsf{rape}}

\newcommand{\m}[1]{{#1}^\mathsf{max}}

\newcommand{\larg}{\mathsf{large}}
\newcommand{\norm}{\mathsf{normal}}

\newcommand{\Ast}{\mathop{\scalebox{1.5}{\raisebox{-0.2ex}{$\ast$}}}}

\newcommand{\graphs}{\cG}

\newcommand{\grape}[2][0]{
\draw[thick, fill] (#2.center) circle(2pt) -- +({30+#1}:0.5) circle(2pt) +(0,0) -- +({-30+#1}:0.5) circle(2pt);
\draw[thick] (#2.center) +({30+#1}:0.5) arc ({120+#1}:{-120+#1}:{0.5 / sqrt(3)});
}

\DeclareMathOperator{\CAT}{CAT}

\DeclareMathOperator{\Lk}{Link}
\DeclareMathOperator{\lk}{\mathsf{Lk}}
\DeclareMathOperator{\st}{\mathsf{St}}
\DeclareMathOperator{\val}{val}
\DeclareMathOperator{\image}{Im}

\DeclareMathOperator{\diam}{diam}
\DeclareMathOperator{\len}{Length}
\DeclareMathOperator{\tripod}{Tripod}
\DeclareMathOperator{\hocolim}{hocolim}

\newcommand{\Lb}[1]{\operatorname{Label}(#1)}

\newcommand\loops\ell
\newcommand\itimes{\mathbin{\interior{\times}}}
\renewcommand\bar\overline

\newcommand{\isom}{\cong}

\newcommand{\qisom}{\approx}

\numberwithin{equation}{section}

\theoremstyle{plain}
\newtheorem{theorem}[equation]{Theorem}
\newtheorem{lemma}[equation]{Lemma}
\newtheorem{corollary}[equation]{Corollary}
\newtheorem{proposition}[equation]{Proposition}

\newtheorem{conjecture}[equation]{Conjecture}

\newtheorem{introthm}{Theorem}

\newtheorem{introcor}[introthm]{Corollary}

\theoremstyle{definition}
\newtheorem{definition}[equation]{Definition}
\newtheorem{assumption}[equation]{Assumption}
\newtheorem{convention}[equation]{Convention}

\newtheorem{example}[equation]{Example}

\theoremstyle{remark}
\newtheorem{remark}[equation]{Remark}
\newtheorem{question}{Question}

\AddToHook{env/theorem/begin}{\crefalias{equation}{theorem}}
\AddToHook{env/lemma/begin}{\crefalias{equation}{lemma}}
\AddToHook{env/corollary/begin}{\crefalias{equation}{corollary}}
\AddToHook{env/proposition/begin}{\crefalias{equation}{proposition}}
\AddToHook{env/claim/begin}{\crefalias{equation}{claim}}
\AddToHook{env/definition/begin}{\crefalias{equation}{definition}}
\AddToHook{env/assumption/begin}{\crefalias{equation}{assumption}}
\AddToHook{env/convention/begin}{\crefalias{equation}{convention}}
\AddToHook{env/observation/begin}{\crefalias{equation}{observation}}
\AddToHook{env/example/begin}{\crefalias{equation}{example}}
\AddToHook{env/remark/begin}{\crefalias{equation}{remark}}
\AddToHook{env/question/begin}{\crefalias{equation}{question}}
\AddToHook{env/conjecture/begin}{\crefalias{equation}{conjecture}}

\crefname{theorem}{Theorem}{Theorems}
\crefname{lemma}{Lemma}{Lemmas}
\crefname{proposition}{Proposition}{Propositions}
\crefname{definition}{Definition}{Definitions}
\crefname{corollary}{Corollary}{Corollaries}
\crefname{assumption}{Assumption}{Assumptions}
\crefname{remark}{Remark}{Remarks}
\crefname{conjecture}{Conjecture}{Conjectures}

\usepackage{marginnote}

\begin{document}
\title{Cubical structures and the large-scale geometry of graph braid groups}
\author{Byung Hee An}
\address{\parbox{\linewidth}{Department of Mathematics Education, Kyungpook National University, Daegu, Korea\\
Department of STEM Education, North Carolina State University, NC, United States, Visiting Scholar}}
\email{anbyhee@knu.ac.kr, ban2@ncsu.edu}
\author{Sangrok Oh}
\address{\parbox{\linewidth}{Innovation Center for MathScience Research \& Education, Pusan National University, Busan, Korea}}
\email{SangrokOh.math@gmail.com}
\begin{abstract}
We study the large-scale geometry of graph braid groups through the cubical structure of their unordered discrete configuration spaces, focusing on quasi-isometry to right-angled Artin groups (RAAGs). We first give a complete classification of graph braid groups quasi-isometric to free groups.

For the $2$-braid group on a graph $\graf$, we study the union $UP_2(\graf)$ of maximal product subcomplexes of the associated unordered discrete configuration space. We introduce a hierarchy recording geometric and algebraic properties of this inclusion and identify conditions under which the quasi-isometry type of its fundamental group is determined by that of the ambient braid group. For normal bunches of grapes, a class of graphs obtained by attaching cycles to trees, this fundamental group is a one-ended free factor, and the quasi-isometry classification of the braid groups reduces to that of these factors. Using the combinatorics of the underlying trees and the associated intersection complexes, we obtain a graph-theoretic sufficient condition and new obstructions for quasi-isometry to RAAGs, yielding infinite families of non-hyperbolic examples and nonexamples. We also construct infinitely many graph $2$-braid groups hyperbolic relative to a thick proper subgroup not isomorphic to any braid group with at most two particles on a subgraph of the underlying graph.
\end{abstract}


\maketitle
\setcounter{tocdepth}{1}
\tableofcontents

\section{Introduction}
For a finite connected graph \(\graf\), the \emph{graph \(n\)-braid group} \(\bbB_n(\graf)\) is the fundamental group of the \emph{unordered \(n\)-configuration space}
\[UC_n(\graf)=\{(x_1,\ldots,x_n)\in \graf^n: x_i\neq x_j\text{ for }i\neq j\}/\bbS_n,\]
where the symmetric group \(\bbS_n\) acts by permuting the coordinates.
A classical problem in this subject, going back to \cite{Ghrist01}, is to determine which graph braid groups are isomorphic to right-angled Artin groups (RAAGs).
Although graph braid groups embed in RAAGs \cite{CW}, they need not themselves be RAAGs. Work on this problem has established both positive results and obstructions \cite{FS08,KKP12,CD14,KLP16}, but a complete classification remains open.

Graph braid groups are analogues of classical braid groups in which particles move on a graph rather than in the plane. 
The graph setting, however, offers a natural approach through cubical geometry: after suitably subdividing \(\graf\), the group \(\bbB_n(\graf)\) is realized as the fundamental group of the unordered discrete \(n\)-configuration space \(UD_n(\graf)\); see \Cref{Abrams}.
These models have played a central role in the study of graph braid groups through combinatorial techniques, particularly discrete Morse theory, yielding results on presentations and cohomology rings; see \cite{FS05,FS08,FS12,KP12,LSUTeam25}.

Beyond their role as combinatorial models, these configuration spaces are \emph{special} cube complexes in the sense of Haglund--Wise \cite{HW08}; see \Cref{thm:GBGSpecial}.
More recently, this special cubical structure has been exploited in the study of graph braid groups, yielding new results on negative curvature \cite{Gen21GBG}, quasi-isometry invariants \cite{Oh22}, and graph of groups decompositions \cite{Ber23}.
In this paper, we pursue this geometric perspective by studying a large-scale version of the classical problem above, replacing isomorphism by quasi-isometry.
Examples of graph braid groups that are quasi-isometric to RAAGs, as well as examples that are not, were constructed in \cite{Oh22}. This leads to the following classification problem.

\begin{question}\label{Question:QItoRAAG}
Which graph braid groups are quasi-isometric to RAAGs? More generally, when is a graph braid group quasi-isometric to a given RAAG?
\end{question}

Our results address this question from two complementary directions.
First, we give a complete characterization of graph braid groups that are quasi-isometric to free groups, addressing the hyperbolic case.
Second, for non-hyperbolic graph \(2\)-braid groups, we study how the union of maximal product subcomplexes of \(UD_2(\graf)\) reflects the large-scale geometry of the group, using intersection complexes of universal covers.
This yields quasi-isometry invariants and graph-theoretic sufficient conditions for, and obstructions to, quasi-isometry to RAAGs, together with applications to relative hyperbolicity.

\subsection{Main results and strategy}

\subsubsection*{Free graph braid groups}

A natural first step toward \Cref{Question:QItoRAAG} is to understand the hyperbolic case. Genevois gave a complete characterization of when \(\bbB_n(\graf)\) is hyperbolic \cite[Theorem~1.1]{Gen21GBG}.  Since a RAAG is hyperbolic if and only if it is free, this leads to the problem of determining when a graph braid group is quasi-isometric to a free group.
Our first main result gives a complete answer.

\begin{introthm}[\Cref{theorem:freeness classification}]
\label{Introthm:FreeClassification}
Let \(\graf\) be a finite connected graph not homeomorphic to a line segment, and let \(n\ge2\).  Then the graph \(n\)-braid group \(\bbB_n(\graf)\) is quasi-isometric to a free group if and only if one of the following holds:
\begin{enumerate}
\item \(n=2\) and \(\graf\) is planar and contains no pair of disjoint cycles;
\item \(n=3\) and \(\graf\) is either a tree, a graph with a unique essential vertex, a graph with a single cycle passing through all essential vertices, or a graph homeomorphic to one obtained from the complete bipartite graph \(\sfK_{2,3}\) by attaching edges to its trivalent vertices;
\item \(n\ge4\) and \(\graf\) has at most one essential vertex.
\end{enumerate}
\end{introthm}

Here, an \emph{essential} vertex is a vertex of valency at least \(3\), and a \emph{bivalent} vertex is a vertex of valency \(2\).

Partial results in this direction were obtained in \cite{FS05,Gen21GBG,LSUTeam25}.
Our proof combines Genevois' hyperbolicity criterion with undistorted surface-subgroup obstructions, cubical decompositions of discrete configuration spaces, and known presentation results. For details, see \Cref{Subsec:FreeGBG}.

\subsubsection*{Product regions and graph \(2\)-braid groups}

We now turn to non-hyperbolic graph braid groups. One important approach in this direction is due to Genevois, who characterized when graph braid groups are \emph{toral relatively hyperbolic}, that is, hyperbolic relative to free abelian subgroups \cite[Theorem~1.3]{Gen21GBG}. His criterion is formulated in terms of the cubical geometry of \(UD_n(\graf)\), and is governed by the presence of subgroups such as \(\bbZ^2\) and the absence of subgroups of the form \(\bbZ\times\bbF_2\).
The product-region viewpoint developed in this paper is complementary to this relative-hyperbolic perspective. In dimension two, the intersection-complex framework introduced in \cite{Oh22} and developed in \cite{OhProducts} allows us to study how maximal product subcomplexes intersect in weakly special \emph{square} complexes. Drawing on this framework, we focus on graph \(2\)-braid groups, for which the arrangement of product regions can yield quasi-isometry information beyond the toral relatively hyperbolic setting.

Let \(Y\) be a compact connected two-sided weakly special square complex, and let \(Z\) be either \(Y\) or its universal cover \(\bar Y\). Roughly speaking, a \emph{standard product subcomplex} of \(Z\) is represented by a \emph{standard product structure}, a local isometry from a product of two nontrivial connected leafless graphs into the ambient complex, with each coordinate fiber embedded. 
A \emph{maximal product structure} is a standard product structure that is maximal among standard product structures with respect to factorization by product embeddings, and the product subcomplex it represents is called a \emph{maximal product subcomplex}.
The intersection complex \(\cI(Z)\) then encodes this arrangement: its vertices are maximal product subcomplexes, and its simplices record their common standard product subcomplexes. See \Cref{Subsection:Intersection complexes} for precise definitions.
Moreover, \(\cI(\bar Y)\) provides a quasi-isometry invariant: quasi-isometries between universal covers of such square complexes induce isomorphisms of their intersection complexes preserving the relevant data; see \cite[Theorems~4.4 and~4.6]{OhProducts} and \Cref{theorem:IsobetInt}.

In \Cref{Subsection:Union of product regions}, we relate this information to the geometry of \(Y\) by considering the union \(\m Y\) of the images of maximal product structures into \(Y\), which collects the product regions of \(Y\).
However, the inclusion \(\m Y\hookrightarrow Y\) need not be locally convex or even \(\pi_1\)-injective, which leads to the \(\m Y\)-hierarchy introduced in \Cref{definition:Ymax hierarchy}. It records conditions involving connectedness, \(\pi_1\)-injectivity, free product decompositions, local convexity, and the equality \(\m Y=Y\), together with the implications among them.

For graph \(2\)-braid groups, the \(\m Y\)-hierarchy can be studied in terms of the underlying graph \(\graf\). First, the space \(UD_2(\graf)\) has the \emph{embedded product property}: every product structure into it is an embedding; see \Cref{Lem:SPSinGBG}. Thus its product subcomplexes can be identified with their embedded images, and we write
\[UP_2(\graf)=\m{UD_2(\graf)}\]
for the union of maximal product subcomplexes of \(UD_2(\graf)\).
Moreover, standard product subcomplexes of \(UD_2(\graf)\) (or \(UP_2(\graf)\)) arise precisely from pairs of disjoint, connected, nontrivial leafless subgraphs of \(\graf\); see \Cref{corollary:maximal}. 
Together, these facts give a combinatorial description of \(UP_2(\graf)\) and graph-theoretic criteria for several levels of the \(\m Y\)-hierarchy. We use this description to identify conditions under which the quasi-isometry type of \(\pi_1(UP_2(\graf))\) is determined by that of \(\bbB_2(\graf)\).

For a square complex with the embedded product property, the \emph{standard intersection property} requires every connected component of every nonempty finite intersection of maximal product subcomplexes to be a standard product subcomplex. This is an additional hypothesis in our second main result.

\begin{introthm}[cf.~\Cref{Thm:SubclassdefiningQII}]\label{Thm:Intro_SubclassdefiningQII}
For \(i=1,2\), let \(\graf_i\) be a finite connected graph such that \(UP_2(\graf_i)\) is nonempty, locally convex in \(UD_2(\graf_i)\), and has the standard intersection property.  If \(\bbB_2(\graf_1)\) and \(\bbB_2(\graf_2)\) are quasi-isometric, then \(\pi_1(UP_2(\graf_1))\) and \(\pi_1(UP_2(\graf_2))\) are quasi-isometric.
\end{introthm}

\subsubsection*{Bunches of grapes}

The preceding theorem leads us to seek graphs for which \(UP_2(\graf)\) is locally convex in \(UD_2(\graf)\) and satisfies the standard intersection property. Since standard product subcomplexes arise from pairs of disjoint connected leafless subgraphs, we look for graphs in which these subgraphs and their disjointness can be described explicitly.
A natural starting point is the class of graphs obtained from trees by attaching cycles at vertices. For graphs in this class, the underlying tree organizes the relevant leafless subgraphs, providing a way to verify the required geometric properties, while disjoint attached cycles give rise to nontrivial product regions.
Moreover, this class turns out to contain infinitely many graphs whose \(2\)-braid groups are neither hyperbolic nor toral relatively hyperbolic, as we show below. It therefore provides a setting in which to apply our product-region approach beyond the toral relatively hyperbolic case of graph \(2\)-braid groups.

\begin{figure}[ht]
\begin{align*}
\graf&=\begin{tikzpicture}[baseline=-.5ex, scale=0.7]
\draw[thick,fill] (0,0) circle (2pt) node (A) {} -- ++(5:2) circle(2pt) node(B) {} -- +(0,1.5) circle (2pt) node (C) {}+(0,0) -- +(0,-1.5) circle (2pt) node(D) {} +(0,0) -- ++(-10:1.5) circle(2pt) node(E) {} -- ++(10:1.5) circle(2pt) node(F) {} -- +(120:1) circle(2pt) node(G) {} +(0,0) -- +(60:1) circle (2pt) node(H) {} +(0,0) -- ++(-5:1) circle (2pt) node(I) {} -- +(30:1) circle(2pt) node(J) {} +(0,0) -- +(-30:1) circle (2pt) node(K) {};
\grape[90]{A};
\grape[270]{A};
\grape[0]{C}; \grape[90]{C}; \grape[180]{C};
\grape[-90]{E};
\grape[-90]{F};
\grape[120]{G};
\grape[60]{H};
\grape[-30]{K};
\end{tikzpicture}&
\graf'&=\begin{tikzpicture}[baseline=-.5ex, scale=0.7]
\draw[thick,fill] (0,0) circle (2pt) node (A) {} -- ++(5:2) circle(2pt) node(B) {} -- +(0,1.5) circle (2pt) node (C) {}+(0,0) -- +(0,-1.5) circle (2pt) node(D) {} +(0,0) -- ++(-10:1.5) circle(2pt) node(E) {} -- ++(10:1.5) circle(2pt) node(F) {} -- +(120:1) circle(2pt) node(G) {} +(0,0) -- +(60:1) circle (2pt) node(H) {} +(0,0) -- ++(-5:1) circle (2pt) node(I) {} -- +(30:1) circle(2pt) node(J) {} +(0,0) -- +(-30:1) circle (2pt) node(K) {};
\grape[90]{A};
\grape[270]{A};
\grape[0]{C}; \grape[90]{C}; \grape[180]{C};
\grape[-90]{E};
\grape[-90]{F};
\grape[120]{G};
\grape[60]{H};
\grape[-30]{K};
\begin{scope}[red]
\grape[-90]{D};
\grape[30]{J};
\end{scope}
\end{tikzpicture}
\end{align*}
\caption{A bunch of grapes and its normal modification.}
\label{figure:grapes stem and twigs}
\end{figure}

To study their graph \(2\)-braid groups using discrete configuration spaces, it suffices to use minimal simplicial representatives in which each attached cycle is a \(3\)-cycle and every bivalent vertex of the underlying tree carries a cycle.
These representatives are called \emph{bunches of grapes}.
We denote their class by \(\grapegraph\), and write \(\grapegraph_\norm\) for the subclass consisting of graphs with no leaves and containing two disjoint cycles. Their formal definition and the associated terminology are given in \Cref{section:grapes}.

For graphs in \(\grapegraph_{\norm}\), the invariant from \Cref{Thm:Intro_SubclassdefiningQII} becomes complete in the following sense.

\begin{introthm}[\Cref{theorem:QIbetweenGBGs}]\label{theorem:QIbetweenGBGs_Intro}
For \(\graf_1,\graf_2\in\grapegraph_{\norm}\), the graph \(2\)-braid groups \(\bbB_2(\graf_1)\) and \(\bbB_2(\graf_2)\) are quasi-isometric if and only if \(\pi_1(UP_2(\graf_1))\) and \(\pi_1(UP_2(\graf_2))\) are quasi-isometric.
In particular, the converse of \Cref{Thm:Intro_SubclassdefiningQII} holds for \(\grapegraph_{\norm}\).
\end{introthm}
\begin{proof}[Sketch of proof]
For \(\graf\in\grapegraph_{\norm}\), the subcomplex \(UP_2(\graf)\) is locally convex in \(UD_2(\graf)\), and
\[\bbB_2(\graf)\cong \pi_1(UP_2(\graf))\ast \bbF_N\]
for some \(N\ge2\); see \Cref{Prop:free factor}.  Moreover, \(\pi_1(UP_2(\graf))\) is one-ended.  The conclusion then follows from the theorem of Papasoglu--Whyte on quasi-isometries of free products \cite{PW02}; see \Cref{PW}.
\end{proof}

We use this reduction to study \Cref{Question:QItoRAAG} for graph \(2\)-braid groups on bunches of grapes. For \(\graf\in\grapegraph_\norm\), if \(\bbB_2(\graf)\) is quasi-isometric to a RAAG, then the universal cover \(\overline{UP_2(\graf)}\) of \(UP_2(\graf)\) is quasi-isometric to the universal cover of a two-dimensional Salvetti complex; see \Cref{Prop:StartingPoint}.
By \Cref{theorem:IsobetInt}, their intersection complexes must therefore be isomorphic. Our description of \(\cI(\overline{UP_2(\graf)})\) yields graph-theoretic obstructions to such an isomorphism; see \Cref{theorem:NotQItoRAAG,theorem:necessary condition 2}.

For example, if the underlying tree contains a tripod whose three arms have pairwise distinct lengths and whose leaves are leaves of the underlying tree, then \(\bbB_2(\graf)\) is not quasi-isometric to any RAAG; see \Cref{theorem:necessary condition 2}.
Together with our sufficient condition for a graph \(2\)-braid group to be a RAAG, these obstructions yield the following families.

\begin{introcor}\label[corollary]{Introcor:QItoRAAG}
There exist infinitely many graph braid groups that are quasi-isometric to RAAGs but are not toral relatively hyperbolic.  Moreover, there exist infinitely many non-hyperbolic graph braid groups that are not quasi-isometric to any RAAG.
\end{introcor}
\begin{proof}
For the positive examples, attach a \(3\)-cycle at every vertex of a path with at least two edges. The resulting graph \(2\)-braid group is isomorphic to a RAAG by \Cref{theorem:isomorphictoRAAG}, but is not toral relatively hyperbolic by \cite[Theorem~1.3]{Gen21GBG}.

For the negative examples, take a tree formed by joining three paths of lengths \(1,2,r\), where \(r\ge3\), at a common endpoint, and attach a \(3\)-cycle at every vertex.
By \Cref{theorem:necessary condition 2}, its graph \(2\)-braid group is not quasi-isometric to any RAAG. It is non-hyperbolic since the graph contains two disjoint cycles; see \Cref{Lem:UP2empty}. In each family, increasing the number of vertices in the tree gives unbounded first Betti numbers by \Cref{Prop:free factor}, and hence infinitely many isomorphism types.
\end{proof}

We also obtain new examples related to relative hyperbolicity. Genevois asked for a characterization of relatively hyperbolic graph braid groups beyond the toral case \cite[Question~5.7]{Gen21GBG}. Berlyne showed that there exists a graph braid group hyperbolic relative to a thick subgroup which does not arise as a braid group on a proper subgraph \cite[Theorem~F]{Ber23}. Our framework shows that this phenomenon occurs in infinite families.

\begin{introthm}[\Cref{Corollary:Thickness}]\label{Introcor:Thickness}
There exist infinitely many graph braid groups \(\bbB_n(\graf)\) that are hyperbolic relative to a thick proper subgroup which is not isomorphic to any graph braid group of the form \(\bbB_k(\grafl)\), where \(k\le n\) and \(\grafl\subseteq\graf\).
\end{introthm}

For \(\graf\in\grapegraph_{\norm}\), the subgroup appearing here is \(\pi_1(UP_2(\graf))\).  The thickness follows from the connectedness of the intersection complex of \(\overline{UP_2(\graf)}\), which implies that \(\pi_1(UP_2(\graf))\) is one-ended and thick of order at most one.

Finally, our results suggest the following natural next step.

\begin{question}
Can one classify graph \(2\)-braid groups on bunches of grapes up to quasi-isometry?
\end{question}

\subsection{Organization of the paper}
In \Cref{section:prelim}, we introduce the terminology and background used throughout the paper, including basic cubical geometry, (weakly) special cube complexes, and discrete configuration spaces.
In \Cref{sec:graph configuration spaces}, we examine the cubical structure of discrete configuration spaces and prove \Cref{Introthm:FreeClassification}, the complete classification of graph \(n\)-braid groups that are quasi-isometric to free groups.

We then turn to the product-region geometry underlying our study of graph \(2\)-braid groups. In \Cref{Section:Hierarchy}, we recall product structures and intersection complexes, introduce \(\m Y\) and the \(\m Y\)-hierarchy, and study geometric and algebraic properties of the inclusion \(\m Y\hookrightarrow Y\).
In \Cref{section:properties}, we specialize this framework to graph \(2\)-braid groups. We describe product subcomplexes of \(UD_2(\graf)\), define \(UP_2(\graf)\), and prove the quasi-isometry invariance result \Cref{Thm:Intro_SubclassdefiningQII} under the relevant hypotheses.
In \Cref{section:grapes}, we study bunches of grapes, prove \Cref{theorem:QIbetweenGBGs_Intro}, and describe the intersection complexes of \(UP_2(\graf)\) and its universal cover in terms of the underlying tree.
Finally, in \Cref{Section:Application to bunches of grapes}, we establish the sufficient condition and obstructions underlying \Cref{Introcor:QItoRAAG}, and prove the relative-hyperbolicity result \Cref{Introcor:Thickness}.

\subsection*{Acknowledgements}

The first author was supported by Samsung Science and Technology Foundation under Project Number SSTF-BA2022-03. The second author was supported by the Basque Government grant IT1483-22.

\section{Preliminaries}\label{section:prelim}

\subsection{Terminology and conventions}
Throughout this paper, we mostly consider two types of finite-dimensional polyhedral complexes \(X\) equipped with their natural length metrics \(d_X\): \emph{cube complexes} and \emph{triangular complexes}\footnote{A \emph{triangular complex}, as in \cite{DK18}, is a regular cell complex in which every cell is an embedded simplex.}. These are built from closed unit \(n\)-cubes \([0,1]^n\) and standard \(n\)-simplices \(\triangle_n\), respectively. 
A \emph{square complex} is a cube complex of dimension at most \(2\), and a \emph{simplicial complex} is a triangular complex in which each simplex is uniquely determined by its vertex set.

For a polyhedral complex \(X\), each \(0\)-cell (\(1\)-cell, resp.) is called a \emph{vertex} (\emph{edge}, resp.), and the vertex set (edge set, resp.) of \(X\) is denoted by \(V(X)\) (\(E(X)\), resp.).
For a subset \(V_0\subseteq V(X)\), a subcomplex \(X_0\) of \(X\) is said to be \emph{induced by \(V_0\)} if a cell in \(X\) is contained in \(X_0\) if and only if its vertices are contained in \(V_0\); such a subcomplex is called \emph{full}.
For a vertex \(x\in X\), the \emph{link} of \(x\) in \(X\), denoted \(\Lk(X,x)\), is the triangular complex whose vertices correspond to half-edges incident to \(x\), and whose \(k\)-simplices correspond to \((k+1)\)-cells of \(X\) containing the associated half-edges.

Let \(X, X'\) be polyhedral complexes.
A map \(\phi:X\to X'\) is \emph{combinatorial} if its restriction to each cell of \(X\) is an isometry onto a cell of \(X'\).
For each vertex \(x\in V(X)\), such a map induces a link map 
\[\Lk(\phi,x):\Lk(X,x)\longrightarrow \Lk(X',\phi(x)).\]
We call \(\phi\) an \emph{immersion} if \(\Lk(\phi,x)\) is injective for every \(x\in V(X)\); an \emph{embedding} if \(\phi\) is an injective immersion; and an \emph{isometry} if \(\phi\) is a surjective embedding. In this case, we say that \(X\) and \(X'\) are \emph{isometric} and write \(X\isom X'\).


For \(\lambda\ge1\) and \(\epsilon\ge0\), a (not necessarily continuous) map \(f:X\to X'\) is called a \emph{\((\lambda,\epsilon)\)-quasi-isometric embedding} if
\[\forall\, x_1,x_2\in X,\quad \frac 1 \lambda d_{X}(x_1,x_2) - \epsilon \le d_{X'}(f(x_1),f(x_2)) \le \lambda d_{X}(x_1,x_2) + \epsilon.\]
It is called a \emph{quasi-isometry} if, in addition, every point of \(X'\) lies within distance \(\epsilon\) of the image. In this case, we say that \(X\) and \(X'\) are \emph{quasi-isometric} and write \(X\qisom X'\).

Suppose that \(X, X'\) are compact and connected, and set \(\bbG=\pi_1(X), \bbG'=\pi_1(X')\). We say that \(\bbG\) is 
\begin{enumerate}
\item \emph{undistorted} in \(\bbG'\) if there exists an injective homomorphism \(\bbG \hookrightarrow \bbG'\) which is a quasi-isometric embedding with respect to some (equivalently, any) choice of finite generating sets;
\item \emph{quasi-isometric} to \(\bbG'\), written \(\bbG\qisom\bbG'\), if the universal covers \(\bar{X}, \bar{X'}\) of \(X,X'\) are quasi-isometric;
\item \emph{\(e\)-ended} if, for any basepoint \(\bar x\in\bar X\), the number of unbounded connected components of the complement of the ball of radius \(M\) about \(\bar x\) tends to \(e\) as \(M\to\infty\) (it is known that \(e\in\{0,1,2,\infty\}\));
\item \emph{hyperbolic} if \(\bar X\) is a hyperbolic space in the sense of \cite{Grom}.
\end{enumerate}

The following result describes how quasi-isometry types behave under taking free products with free groups.

\begin{theorem}[\cite{PW02}]\label{PW}
Let $\mathbb{G}_1$ and $\mathbb{G}_2$ be two finitely generated infinite groups. If $\mathbb{G}_1$ and $\mathbb{G}_2$ are quasi-isometric, then $\mathbb{G}_1*\mathbb{F}_m$ and $\mathbb{G}_2*\mathbb{F}_n$ are quasi-isometric for $m,n\geq 1$.  
Moreover, the converse also holds if both $\mathbb{G}_1$ and $\mathbb{G}_2$ are one-ended.
\end{theorem}
\begin{proof}
The first statement follows from \cite[Theorem~0.3]{PW02}, and the second from \cite[Theorem~0.4]{PW02}.
\end{proof}

Now, we turn our attention to cube complexes. A cube complex \(Y\) is \emph{non-positively curved (NPC)} if for every vertex \(y\), \(\Lk(Y,y)\) is a flag simplicial complex. If, in addition, \(Y\) is simply connected, then \(Y\) is \(\CAT(0)\).

Let \(Y\) and \(Y'\) be connected NPC cube complexes. A combinatorial map \(\phi:Y\to Y'\) is called a \emph{local isometry} if, for every \(y\in V(Y)\), the induced link map \(\Lk(\phi,y)\) is injective and its image is an induced subcomplex; a \emph{locally isometric embedding} if \(\phi\) is an injective local isometry; an \emph{isometric embedding} if both \(Y\) and \(Y'\) are \(\CAT(0)\) and \(\phi\) is a locally isometric embedding.
For a (locally) isometric embedding \(\phi\), the image \(\phi(Y)\) is said to be (\emph{locally}) \emph{convex} in \(Y'\).

\begin{remark}\label{Rem:NPCvsLocallyCAT(0)}
For a cube complex, being NPC in the sense above is equivalent to being \emph{locally \(\CAT(0)\)} as a length space; see \cite{Grom,Lea} for the finite and infinite dimensional cases, respectively. In particular, a local isometry between NPC cube complexes is locally an isometric embedding in the usual metric sense; see \cite{CW,HW08,Wis12}. Moreover, as a \(\CAT(0)\) cube complex is contractible, the fundamental group of every compact NPC cube complex is torsion-free.
\end{remark}

A local isometry \(\phi:Y\to Y'\) induces
\begin{enumerate}
\item an injective homomorphism on fundamental groups \(\phi_*:\pi_1(Y,y)\rightarrow\pi_1(Y',\phi(y))\), and 
\item an isometric embedding on universal covers \(\bar\phi:\bar Y\to \bar{Y'}\), called an \emph{elevation} of \(\phi\), satisfying \(\phi\circ p_Y=p_{Y'}\circ\bar\phi\). 
\end{enumerate}
The image of \(\bar\phi\) will be called a \emph{copy} of \(\bar{Y}\) or a \emph{lift} of \(Y\) in \(\bar{Y'}\).
It is well known that if \(Y\) and \(Y'\) are compact, then \(\phi_*(\pi_1(Y,y))\) is undistorted in \(\pi_1(Y',\phi(y))\).
For more details, see \cite[Chapter II.4]{BH}.

\begin{example}[Flats in \(\CAT(0)\) cube complexes]\label{Ex:flat}
Viewing \(\bbR^n\) as the standard \(\CAT(0)\) cube complex tiled by unit \(n\)-cubes, the image of an isometric embedding \(\bbR^n\to X\) into a \(\CAT(0)\) cube complex \(X\) is called an \emph{\(n\)-dimensional flat}, or briefly \emph{\(n\)-flat}. If \(\dim X=n\), then such a flat is said to be \emph{top-dimensional}.

Let \(Y\) be an NPC cube complex. If there exists a local isometry \(\iota\colon \sfC_{i_1}\times\cdots\times\sfC_{i_n}\to Y\), where \(\sfC_i\) denotes the cycle graph of length \(i\), then any elevation of \(\iota\) is an isometric embedding \(\bbR^n\to\bar Y\) whose image is an \(n\)-dimensional flat.
\end{example}

As a special case, when mapping a product of graphs into a cube complex of the same dimension, any immersion is automatically a local isometry.

\begin{lemma}\label{Lem:ImmersionImpliesLocalisometry}
Let \(Y\) be an \(n\)-dimensional NPC cube complex, and let \(\graf_1,\dots,\graf_n\) be connected graphs with at least one edge.
Then any immersion \(\iota:\graf_1\times\dots\times\graf_n\to Y\) is a local isometry.
\end{lemma}
\begin{proof}
Each vertex link in \(\graf_1\times\cdots\times\graf_n\) is a join of \(n\) nonempty discrete sets, and the link maps of \(\iota\) are injective.
Since both source and target links are flag, if a link image were not induced, there would be an additional edge joining the images of two vertices in the same factor.
Together with one vertex from each other factor, these vertices would span an \(n\)-simplex in the target link, contradicting \(\dim Y=n\). Thus \(\iota\) is a local isometry.
\end{proof}

\begin{convention}\label{assumption:cube}
In this paper, we adopt the following standing conventions and assumptions unless stated otherwise:
\begin{itemize}[leftmargin=2em]
\item A \(1\)-dimensional cube complex, which is always NPC, is called a \emph{graph}, and is typically denoted by \(\graf, \grafd, \grafl, \sfT\), etc. We also regard the real line $\bbR$ as a graph whose vertex set is the set of integers.
\item All cube complexes (including graphs) are assumed to be connected and NPC.
\item Unless stated otherwise, every (sub)graph is assumed to be \emph{nontrivial}, meaning that it contains at least one edge. A graph is called \emph{simple} if it is a simplicial complex, \emph{leafless} if it contains no \emph{leaves} (i.e., vertices of valency \(1\)), and a \emph{tree} if its \(\pi_1\) is trivial.
\item Groups are assumed to be finitely generated, but their subgroups need not be. We denote groups by \(\bbB,\bbG,\bbS\), etc. A free group is denoted by \(\bbF\), or by \(\bbF_n\) when its rank \(n\) is specified. We write \(\bbZ\) for \(\bbF_1\).
\item The image of a map \(f\) is denoted by \(\image f\).
\end{itemize}
Let \(X\) be a polyhedral complex.
\begin{itemize}[leftmargin=2em]
\item When \(X'\) is a subcomplex of \(X\), it is equipped with its intrinsic metric \(d_{X'}\), rather than the restriction of \(d_X\).
\item The universal cover of \(X\) and the covering map are denoted by \(\bar X\) and \(p_X:\bar X\to X\), respectively.
\end{itemize}
\end{convention}

\subsection{Weakly special cube complexes}\label{section:WSSCandIC}
Let \(Y\) be a cube complex.
Two edges \(e_1\) and \(e_2\) of \(Y\) are said to be \emph{parallel} if there is an immersion \(e\times [0,n]\to Y\) for some integer \(n\ge0\) such that \(e_1\) and \(e_2\) are the images of \(e\times \{0\}\) and \(e\times\{n\}\), respectively. The \emph{hyperplane} \(H\) dual to an edge \(e\) of \(Y\) is the set of all edges parallel to \(e\). 

Haglund--Wise \cite{HW08} introduced special cube complexes by imposing conditions on their hyperplanes. Motivated by this framework, Huang \cite{Hua17QuasiFlat} introduced weakly special cube complexes to study the large-scale geometry of their universal covers.

\begin{definition}[\cite{HW08, Hua17QuasiFlat} (Weakly) special cube complex]\label{Def:SpecialCubeComplex}
An (NPC) cube complex \(Y\) is \emph{weakly special} if its hyperplanes neither self-intersect nor self-osculate. We follow \cite[Definition~5.3]{Hua17QuasiFlat}: distinct unoriented edges incident to a common vertex must be dual to distinct hyperplanes; the two germs of one loop edge do not count as self-osculation. It is \emph{two-sided} if every hyperplane is two-sided.

We use \emph{special} in the two-sided sense: an NPC cube complex is special if its hyperplanes are two-sided, neither self-intersect nor directly self-osculate, and no two hyperplanes inter-osculate.
\end{definition}

We next recall right-angled Artin groups and their Salvetti complexes, which provide basic examples of special and weakly special cube complexes.

\begin{definition}[RAAGs and Salvetti complexes]\label{Def:Salvetti complex}
Let \(\grafl\) be a (possibly disconnected or trivial) finite simple graph. The \emph{right-angled Artin group} (RAAG) \(\bbA_{\grafl}\) associated to \(\grafl\) is defined by the presentation
\[\bbA_{\grafl}=\langle\, V(\grafl)\,:\, \sfv_i\sfv_j=\sfv_j\sfv_i\quad\forall\{\sfv_i,\sfv_j\}\in E(\grafl)\,\rangle.\]
The \emph{Salvetti complex} \(S_{\grafl}\) is obtained by gluing one standard cubical \(n\)-torus for each \(n\)-clique of \(\grafl\), identifying coordinate subtori corresponding to common subcliques. 
\end{definition}

The Salvetti complex \(S_{\grafl}\) is a compact special cube complex with a single vertex and fundamental group \(\bbA_{\grafl}\). It is also weakly special, since distinct unoriented edges are dual to distinct hyperplanes. Its dimension is the largest integer \(n\) such that \(\grafl\) contains an \(n\)-clique. 
A compact cube complex \(Y\) is special if and only if it admits a local isometry into a Salvetti complex \cite{HW08}.

The following is a well-known fact about RAAGs; see \cite{CH} for more about RAAGs and Salvetti complexes.

\begin{proposition}[cf.~{\cite[Section~2.7]{CH}}]\label{Prop:BasiconRAAG}
Let \(\grafl\) be a (possibly disconnected or trivial) finite simple graph. Then the number of ends of \(\bbA_{\grafl}\) is determined by the structure of \(\grafl\) as follows:
\begin{enumerate}
\item \(\bbA_{\grafl}\) is one-ended if and only if \(\grafl\) is connected and has at least two vertices;
\item \(\bbA_{\grafl}\) is two-ended if and only if \(\grafl\) consists of a single vertex;
\item \(\bbA_{\grafl}\) has infinitely many ends if and only if \(\grafl\) has at least two vertices and is disconnected.
In this case, \(\bbA_{\grafl}\cong \bbA_{\grafl_1}\ast\cdots\ast\bbA_{\grafl_k}\),
where \(\grafl_1,\dots,\grafl_k\) are the connected components of \(\grafl\).
\end{enumerate}
\end{proposition}

Another class of special and weakly special cube complexes---central to this paper---is given by the discrete configuration spaces of graphs \(D_n(\graf)\) and \(UD_n(\graf)\), both of which admit natural cube complex structures.

\begin{definition}[Discrete configuration spaces of graphs]\label{definition:discrete configuration spaces}
Let \(\graf\) be a (possibly disconnected) finite graph and let \(n\ge 0\) be an integer. The \emph{ordered} and \emph{unordered discrete \(n\)-configuration spaces} of \(\graf\) are defined by
\[D_n(\graf)=\bigcup_{\substack{\sigma_i\cap\sigma_j=\varnothing\\i\ne j}}\sigma_1\times\cdots\times\sigma_n\subseteq\graf^n,
\qquad UD_n(\graf)=D_n(\graf)/\bbS_n,\]
where the union ranges over all \(n\)-tuples \((\sigma_1,\dots,\sigma_n)\) of pairwise disjoint closed cells (vertices or edges) of \(\graf\).
\end{definition}

Each allowed tuple \((\sigma_1,\dots,\sigma_n)\) determines a cube \(\sigma_1\times\cdots\times\sigma_n\), whose dimension equals the number of edges among the \(\sigma_i\). Since \(\bbS_n\) acts isometrically on \(D_n(\graf)\), the quotient \(UD_n(\graf)\) inherits a cube complex structure.

\begin{theorem}\label{thm:GBGSpecial}
Each component of \(D_n(\graf)\) and of \(UD_n(\graf)\) is a compact cube complex that is both special and weakly special.
\end{theorem}
\begin{proof}
Specialness follows from \cite[Theorem~2]{CW}; see also \cite{Gen21GBG}.
To see weak specialness directly, label each edge of a configuration space by the edge of \(\graf\) along which a particle moves.
Opposite edges of a square have the same label, whereas distinct edges incident to a vertex have distinct labels. Thus distinct incident edges are dual to distinct hyperplanes.
\end{proof}

We also consider weakly special cube complexes because the quasi-isometry methods used later apply in this setting. Huang proved that quasi-isometries between the universal covers of compact weakly special cube complexes of the same dimension preserve top-dimensional flats up to uniformly bounded Hausdorff distance \cite[Theorem~1.3]{Hua17QuasiFlat}.
In dimension two, this result underlies the study of product subcomplexes and intersection complexes developed by the second author in \cite{Oh22,OhProducts}.
We recall this framework in \Cref{Section:Hierarchy}.

\section{Configuration spaces of graphs and freeness}\label{sec:graph configuration spaces}
In this section we recall fundamental facts about configuration spaces of graphs and graph braid groups. We review the structure of these spaces, identify when graph braid groups are hyperbolic, and determine precisely when they are (quasi-isometric to) free groups.

Throughout this section, unless stated otherwise, every (sub)graph \(\graf\) is assumed to be finite and simple and, as in \Cref{assumption:cube}, nontrivial and connected. Possible exceptions are as follows:
\begin{enumerate}
\item A subset of vertices or edges of \(\graf\); in particular, a single vertex \(\sfv\) may be regarded as a trivial subgraph.
\item The \emph{complement} \(\graf \setminus \graf_0\) of a (possibly disconnected or trivial) subgraph \(\graf_0 \subseteq \graf\), which is induced by \(V(\graf) \setminus V(\graf_0)\).
\end{enumerate}
We further fix the following notation and terminology for graphs:
\begin{enumerate}
\item An \emph{\(n\)-path} \(\sfP_n\) (an \emph{\(n\)-cycle} \(\sfC_n\), resp.) is a graph with \(n\) edges homeomorphic to a line segment (a circle, resp.).
\item The complete graph on \(n\) vertices is denoted by \(\sfK_n\), and the complete bipartite graph with partition sets of size \(m\) and \(n\) is denoted by \(\sfK_{m,n}\). In particular, \(\sfK_{1,n}\) (or \(\sfK_{n,1}\)) is called an \emph{\(n\)-star} and denoted by \(\sfS_n\).
\item For each vertex \(\sfv\) of \(\graf\), \(\Lk(\graf,\sfv)\) is realized as the subgraph \(\lk_\graf(\sfv)\) consisting of the vertices adjacent to \(\sfv\). The subgraph formed by the union of all edges incident to \(\sfv\), which is an \(n\)-star, is denoted by \(\st_\graf(\sfv)\).
\end{enumerate}

\subsection{Discrete configuration spaces of graphs}\label{Subsection:Discrete configuration spaces of graphs}
Recall the unordered discrete configuration space \(UD_n(\graf)\) of a graph \(\graf\), defined in \Cref{definition:discrete configuration spaces}; by \Cref{thm:GBGSpecial}, this is a special cube complex.
Notice that its cube structure does depend on the cell structure on \(\graf\). We say that \(\graf\) is \emph{sufficiently subdivided} for \(n\) if
\begin{enumerate}
\item each path between two non-bivalent vertices contains at least \(n-1\) edges;
\item each cycle contains at least \(n+1\) edges.
\end{enumerate}
In particular, a graph \(\graf\) is sufficiently subdivided for \(n=2\) if and only if \(\graf\) is simple.

\begin{proposition}[\cite{Abrams00,KKP12,PS12}]\label{Abrams}
Let \(\graf\) be a graph that is sufficiently subdivided for \(n\ge 1\). 
Then the discrete \(n\)-configuration space \(UD_n(\graf)\) is a strong deformation retract of \(UC_n(\graf)\), and therefore, the graph braid group \(\bbB_n(\graf)\) is the fundamental group of a compact special cube complex.
\end{proposition}

We next describe locally convex subcomplexes of \(UD_n(\graf)\) arising from subgraphs of \(\graf\). These constructions provide undistorted subgroups of graph braid groups and will also be used in \Cref{section:properties} to describe product subcomplexes of \(UD_2(\graf)\).

\begin{lemma}[{\cite[Theorem~3.13]{Abrams00}}]\label{Lem:embedding}
For any subgraph \(\graf_0\) of \(\graf\), the inclusion \(\graf_0\hookrightarrow\graf\) induces a locally isometric embedding \(\iota:UD_n(\graf_0)\hookrightarrow UD_n(\graf)\). In particular, \(\bbB_n(\graf)\) contains an undistorted subgroup isomorphic to \(\bbB_n(\graf_0)\).
\end{lemma}

\begin{remark}\label{Remark:UD_nfordisconnected}
Allowing \(n_i=0\), if \(\graf=\graf_1\sqcup\dots\sqcup\graf_m\) is a disjoint union of graphs, then the unordered discrete \(n\)-configuration space decomposes as
\[UD_n(\graf)= \bigsqcup_{\substack{n_1+\dots+n_m=n \\ n_i\ge 0}} UD_{n_{1}}(\graf_1)\times\dots\times UD_{n_{m}}(\graf_m),\]
where each component is a special cube complex.
\end{remark}

Let \(n_1,\dots, n_m\ge 1\) be integers and set \(n=n_1+\dots+n_m\). Let \(\graf\) be a graph that is sufficiently subdivided for \(n\). For finite pairwise disjoint subgraphs \(\graf_1, \dots, \graf_m\subseteq\graf\), we further assume that each \(\graf_i\) is sufficiently subdivided for \(n_i\). We temporarily allow some of the \(\graf_i\) to have no edges, but only when \(n_i=1\). 
By \Cref{Lem:embedding,Remark:UD_nfordisconnected}, there is a locally isometric embedding, called a \emph{canonical inclusion},
\[UD_{n_1}(\graf_1)\times\dots\times UD_{n_m}(\graf_m)\hookrightarrow UD_n(\graf).\]
In the special case where \(n_1=\dots=n_m=1\), this yields a locally isometric embedding \(\graf_1\times\dots\times\graf_m\hookrightarrow UD_m(\graf)\), whose image is denoted by \(\graf_1\itimes\dots\itimes\graf_m\). 

The following lemma gives a converse in dimension \(n\): the image of any immersion from a product of \(n\) graphs into \(UD_n(\graf)\) arises from \(n\) pairwise disjoint subgraphs of \(\graf\).

\begin{lemma}\label{lemma:local isometric embedding into UD}
Let \(\graf\) be a graph that is sufficiently subdivided for \(n\ge 1\), and let \(\grafl_1,\dots,\grafl_n\) be graphs. If there is an immersion \(\iota:\grafl_1\times\cdots\times\grafl_n\rightarrow UD_n(\graf)\), then there exist pairwise disjoint subgraphs \(\graf_1,\dots,\graf_n\subseteq\graf\) and surjective local isometries \(q_i:\grafl_i\to\graf_i\) such that \(\iota\) factors as
\[\grafl_1\times\cdots\times\grafl_n\xrightarrow{q_1\times\cdots\times q_n}
\graf_1\times\cdots\times\graf_n \hookrightarrow UD_n(\graf),\]
where the last map is the canonical inclusion.
\end{lemma}
\begin{proof}
The existence of \(\iota\) implies that \(UD_n(\graf)\) contains an \(n\)-cube. Since \(\dim UD_n(\graf)\le n\), it follows from \Cref{Lem:ImmersionImpliesLocalisometry} that \(\iota\) is a local isometry.

The case \(n=1\) is immediate, so assume that \(n\ge 2\). Each edge \(E\) of \(UD_n(\graf)\) is represented by one edge of \(\graf\) and \(n-1\) vertices. We call that edge of \(\graf\) the \emph{label} of \(E\). By the definition of \(UD_n(\graf)\), opposite edges of a square have the same label, while adjacent edges of a square have disjoint labels.

Fix \(i\) and an edge \(\sfe_i\in E(\grafl_i)\). Since the other factors are connected, any two copies of \(\sfe_i\) in \(\grafl_1\times\cdots\times\grafl_n\) are connected by a sequence of pairs of opposite edges of squares.
Their images under \(\iota\) therefore have the same label, which we denote by \(\sff_i(\sfe_i)\in E(\graf)\). Let \(\graf_i\) be the subgraph of \(\graf\) given by the union of these edges as \(\sfe_i\) varies over \(E(\grafl_i)\).
For \(i\neq j\), copies of any two edges \(\sfe_i\in E(\grafl_i)\) and \(\sfe_j\in E(\grafl_j)\) occur as adjacent edges of a square in the product.
Consequently, \(\sff_i(\sfe_i)\) and \(\sff_j(\sfe_j)\) are disjoint in \(\graf\). It follows that the subgraphs \(\graf_1,\dots,\graf_n\) are pairwise disjoint.

Since each \(\grafl_i\) is connected and nontrivial, every vertex \((\sfw_1,\dots,\sfw_n)\) of the product lies in an \(n\)-cube. Its image under \(\iota\) therefore consists of exactly one vertex in each \(\graf_i\). Along an edge in the \(j\)-th factor, only the vertex in \(\graf_j\) changes. By connectedness of the other factors, the vertex in \(\graf_i\) depends only on \(\sfw_i\). Thus there are maps \(q_i:V(\grafl_i)\to V(\graf_i)\) satisfying
\[\iota(\sfw_1,\dots,\sfw_n)=\{q_1(\sfw_1),\dots,q_n(\sfw_n)\}.\]
These maps extend to graph maps \(q_i:\grafl_i\to\graf_i\) by sending each edge \(\sfe_i\) to \(\sff_i(\sfe_i)\). They are surjective by construction; in particular, each \(\graf_i\) is connected.

If some \(q_i\) identified two distinct half-edges incident to a vertex, then, after fixing vertices in the other factors, the corresponding half-edges in the product would have the same image under \(\iota\). This contradicts injectivity of the link map of \(\iota\).
Hence each \(q_i\) is an immersion of graphs, and therefore a local isometry. The required factorization now follows from the construction.
\end{proof}

We now use the canonical inclusions to construct free abelian subgroups from disjoint cycles and essential vertices.

\begin{example}
As in \Cref{Fig:tripodFig}, \(UD_2(\sfS_3)\isom\sfC_6\) and \(UD_2(\sfC_3)\isom\sfC_3\), and therefore \(\bbB_2(\sfS_3)\isom\bbB_2(\sfC_3)\isom\bbZ\). 
If \(\graf\) contains either a vertex of valency \(\geq 3\) or a \(3\)-cycle, then it contains a subgraph isomorphic to \(\sfS_3\) or \(\sfC_3\), and so by \Cref{Lem:embedding}, \(UD_2(\graf)\) contains a locally convex subcomplex isomorphic to \(\sfC_6\) or \(\sfC_3\), respectively.
\end{example}

\begin{figure}[ht]
\[
\setlength{\arraycolsep}{0.5pc}
\begin{array}{cccc}
\begin{tikzpicture}[baseline=-.5ex,scale=0.8,transform shape]
\draw[thick] (0,0) node[below] {\(\sfv\)} -- (-30:1) node[right] {\(\sfc\)} (0,0) -- (90:1) node[above] {\(\sfa\)} (0,0) -- (210:1) node[left] {\(\sfb\)};
\draw[fill] (-30:1) circle (2pt) (90:1) circle (2pt) (210:1) circle (2pt);
\end{tikzpicture}\quad &
\begin{tikzpicture}[baseline=-.5ex,scale=0.8,transform shape]
\foreach \i in {1,2,3,4,5,6} {
\draw[fill] ({\i*60+30}:1) circle (2pt);
\draw[thick] ({\i*60+30}:1) -- ({\i*60+90}:1);}
\draw (30:1) node[right] {\(\{\sfa,\sfb\}\)}
(90:1) node[above] {\(\{\sfa,\sfv\}\)}
(150:1) node[left] {\(\{\sfa,\sfc\}\)}
(210:1) node[left] {\(\{\sfc,\sfv\}\)}
(270:1) node[below] {\(\{\sfb,\sfc\}\)}
(330:1) node[right] {\(\{\sfb,\sfv\}\)};
\end{tikzpicture}\quad &
\begin{tikzpicture}[baseline=-.5ex,scale=0.8,transform shape]
\foreach \i in {1,2,3} {
\draw[fill] ({\i*120-30}:1) circle (2pt);
\draw[thick] ({\i*120-30}:1) -- ({\i*120+90}:1);}
\draw (90:1) node[above] {\(\sfb\)}
(-30:1) node[right] {\(\sfa\)}
(210:1) node[left] {\(\sfc\)};
\end{tikzpicture}\quad &
\begin{tikzpicture}[baseline=-.5ex,scale=0.8,transform shape]
\foreach \i in {1,2,3} {
\draw[fill] ({\i*120-30}:1) circle (2pt);
\draw[thick] ({\i*120-30}:1) -- ({\i*120+90}:1);}
\draw (90:1) node[above] {\(\{\sfa,\sfc\}\)}
(-30:1) node[right] {\(\{\sfa,\sfb\}\)}
(210:1) node[left] {\(\{\sfb,\sfc\}\)};
\end{tikzpicture}\\
\sfS_3& UD_2(\sfS_3)& \sfC_3& UD_2(\sfC_3)
\end{array}
\]
\caption{Examples of \(UD_2(\graf)\).}
\label{Fig:tripodFig}
\end{figure}

\begin{proposition}\label{Prop:FreeAbelianSubgroup}
Let \(\graf\) be a graph. If \(\graf\) contains a pairwise disjoint collection of \(p\) cycles and \(q\) essential vertices for some \(p,q\) with \(p+2q\le n\), then \(\bbB_n(\graf)\) contains an undistorted subgroup isomorphic to \(\bbZ^{p+q}\). In particular, it is not quasi-isometric to a hyperbolic group if \(p+q\ge 2\).
\end{proposition}
\begin{proof}
If \(\graf\) is a cycle, then \(\bbB_n(\graf)\cong\bbZ\), and the assertion is immediate. Assume henceforth that \(\graf\) is not a cycle. After subdividing further if necessary, we may assume that \(\graf\) is sufficiently subdivided for \(n+5\).
Let \(\sfC_1,\dots,\sfC_p\) and \(\sfv_1,\dots,\sfv_q\) be pairwise disjoint cycles and essential vertices, respectively.
By the subdivision assumption, the complement \(\graf\setminus\bigcup_{i=1}^p\sfC_i\setminus\bigcup_{i=1}^q\sfv_i\) contains a path \(\sfP\) of length at least \(n+2\). We may therefore choose \(r=n-p-2q\) distinct vertices \(\sfw_1,\dots,\sfw_r\) in \(\sfP\) such that the distance between \(\sfv_i\) and \(\sfw_j\) in \(\graf\) is at least two for all \(i,j\).

Since each \(\sfv_i\) is essential, there exist subgraphs \(\grafl_1,\dots,\grafl_q\), each of which is isomorphic to the \(3\)-star \(\sfS_3\).
Then we have a locally isometric embedding
\[\sfC_1\times\dots\times\sfC_p\times UD_2(\grafl_1)\times\dots\times UD_2(\grafl_q)\times \sfw_1\times\dots\times\sfw_r \hookrightarrow UD_n(\graf).\]
This means that \(\bbB_n(\graf)\) contains an undistorted subgroup isomorphic to 
\[\prod_{i=1}^p\pi_1(\sfC_i)\times \prod_{j=1}^q\bbB_2(\grafl_j)\isom\bbZ^{p+q}.\]
If \(p+q\ge2\), then it follows that \(\bbB_n(\graf)\) is not quasi-isometric to a hyperbolic group by \cite[Corollary~1.2]{Gen21}.
\end{proof}

\begin{corollary}\label{corollary:hyperbolicity}
Let \(\graf\) be a graph that is not a path graph.
If \(\bbB_n(\graf)\) is quasi-isometric to a hyperbolic group, then:
\begin{enumerate}
\item when \(n=2\), \(\graf\) contains no pair of disjoint cycles;
\item when \(n=3\), \(\graf\) contains no disjoint pair consisting of a cycle and an essential vertex;
\item when \(n\ge 4\), \(\graf\) contains at most one essential vertex.
\end{enumerate}
\end{corollary}
\begin{remark}
Indeed, the above necessary conditions are sufficient by \cite[Theorem~1.1]{Gen21GBG}.
\end{remark}

The case \(n\ge4\) in \Cref{corollary:hyperbolicity} singles out graphs with at most one essential vertex. Such graphs are determined up to homeomorphism by their numbers of leaves and cycles and admit the following representatives.

\begin{definition}[Elementary bunches of grapes]\label{definition:FreeGroupCase}
For integers \(k,\ell\ge 0\) with \(k+2\ell\ge 2\), let \(\grafl=\grafl(k,\ell)\) be the graph obtained from the \(k\)-star \(\sfS_k\) with central vertex \(\sfv\), where \(\sfS_0=\{\sfv\}\), by attaching \(\ell\) copies of a \(3\)-cycle at \(\sfv\). That is, 
\[\grafl=\left(\sfS_k\coprod_{i=1}^\ell\sfC_3^i\right)\bigg/ \sfv\sim \sfzero^i,\] 
where \(\sfC_3^i\) is the \(i\)-th copy of a 3-cycle \(\sfC_3\) and \(\sfzero^i\) is its designated vertex.
We call such a graph \(\grafl\) an \emph{elementary bunch of grapes} of type \((k,\ell)\); `general' bunches of grapes will be defined later (see \Cref{Def:BunchesofGrapes}).
\end{definition}

\begin{theorem}[\cite{KP12}]\label{theorem:elementary rank}
Let \(\grafl=\grafl(k,\ell)\) be an elementary bunch of grapes with \(k+2\ell\ge 2\).
Then the graph \(n\)-braid group \(\bbB_n(\grafl)\) is free of rank
\[N=N(n,k,\ell)=\binom{n+k+\ell-2}{k+\ell-1}(k+2\ell-2)-\binom{n+k+\ell-2}{k+\ell-2}+1.\]
Here, a binomial coefficient with a negative lower index is understood to be zero.
\end{theorem}

\subsection{Free graph braid groups}\label{Subsec:FreeGBG}
We now characterize when \(\bbB_n(\graf)\) is free. Since graph braid groups are finitely generated and torsion-free, this is equivalent to being quasi-isometric to a free group.

\begin{theorem}\label{theorem:freeness classification}
Let \(\graf\) be a graph which is not a path graph, and let \(n\ge 2\). Then the graph braid group \(\bbB_n(\graf)\) is (quasi-isometric to) a free group if and only if one of the following holds:
\begin{enumerate}
\item \(n=2\) and \(\graf\) is a planar graph without pairs of disjoint cycles;
\item \(n=3\) and \(\graf\) is either a tree, a graph with a single cycle passing through all essential vertices, or a graph homeomorphic to either an elementary bunch of grapes or a graph obtained from \(\sfK_{2,3}\) by attaching edges to its essential vertices;
\item \(n\ge4\) and \(\graf\) is a subdivision of an elementary bunch of grapes.
\end{enumerate}
\end{theorem}

The proof combines three main ingredients. First, hyperbolicity imposes strong restrictions on the underlying graph. Second, we detect undistorted surface subgroups in \(\bbB_n(\graf)\), which provide geometric obstructions to freeness. Third, we describe how graph operations give rise to group-theoretic decompositions, notably via iterated HNN extensions arising from the cubical structure of configuration spaces. Together, these ingredients show that freeness can be detected through the large-scale geometry of \(\bbB_n(\graf)\).

We treat the cases \(n\ge4\), \(n=2\), and \(n=3\) separately, as each exhibits distinct geometric features.

\subsubsection*{The case \(n\ge4\)}

In this case, freeness is determined entirely by hyperbolicity.

\begin{proof}[Proof of \Cref{theorem:freeness classification} for \(n\ge4\)]
If \(\bbB_n(\graf)\) is free, then it is hyperbolic. By \Cref{corollary:hyperbolicity,definition:FreeGroupCase}, \(\graf\) must be a subdivision of an elementary bunch of grapes. Conversely, freeness follows from \Cref{theorem:elementary rank}.
\end{proof}

\subsubsection*{The graph \(2\)-braid group case}

We now turn to the case \(n=2\), where planar structure plays a central role.

\begin{proposition}\label{Non-planar}
If \(\graf\) is a non-planar graph, then \(\bbB_2(\graf)\) has an undistorted surface subgroup. In particular, \(\bbB_2(\graf)\) is not (quasi-isometric to) a free group.
\end{proposition}
Here, a \emph{surface (sub)group} refers to a (sub)group isomorphic to the fundamental group of a closed hyperbolic surface; in particular, it is quasi-isometric to the hyperbolic plane.
\begin{proof}
By Kuratowski's theorem, a graph \(\graf\) is non-planar if and only if it contains a subdivision of either the complete bipartite graph \(\sfK_{3,3}\) or the complete graph \(\sfK_5\).
Since \(UD_2(\sfK_{3,3})\) and \(UD_2(\sfK_5)\) are homeomorphic to closed hyperbolic surfaces, \Cref{Lem:embedding} implies that \(\bbB_2(\graf)\) contains an undistorted surface subgroup.

Since every subgroup of a free group is free, \(\bbB_2(\graf)\) is not free. The second assertion now follows from the equivalence recalled at the beginning of this subsection.
\end{proof}

Sabalka constructed a planar graph \(\graf\) for which \(\bbB_2(\graf)\) contains a surface subgroup \cite[Corollary~3.4]{Sab07}.
After the first draft of this paper appeared, the authors and Park obtained results \cite[Proposition~5.11 and Remark~5.12]{AOP26} which, applied to Sabalka's planar \(\sfC_6\)-halo graph, establish the existence of an undistorted surface subgroup in its graph \(2\)-braid group.

To complete the freeness classification for \(n=2\), we use the following presentation theorem for planar graphs.

\begin{theorem}[{\cite[Theorem~4.8]{KP12}}]\label{scrg}
If \(\graf\) is planar, then \(\bbB_2(\graf)\) admits a presentation whose relators are commutators associated to pairs of disjoint cycles in \(\graf\).
In particular, if \(\graf\) contains no two disjoint cycles, then \(\bbB_2(\graf)\) is free.
\end{theorem}

Graphs without two disjoint cycles admit a structural characterization due to Lov\'asz \cite{Lovasz}. The simplest examples are trees.

\begin{example}[\cite{KP12}]\label{example:trees}
For a tree \(\sfT\), the graph \(2\)-braid group \(\bbB_2(\sfT)\) is a free group of rank \(N(\sfT)\)
\[N(\sfT)=\sum_{\sfv\in V(\sfT)} \binom{\val_{\sfT}(\sfv)-1}{2},\]
which is at least the number of essential vertices of \(\sfT\).
\end{example}

\begin{proof}[Proof of \Cref{theorem:freeness classification} for \(n=2\)]
If \(\bbB_2(\graf)\) is free, then it is hyperbolic, and hence by \Cref{corollary:hyperbolicity,Non-planar}, \(\graf\) is a planar graph without two disjoint cycles. The converse implication follows from \Cref{scrg}.
\end{proof}

\subsubsection*{The graph \(3\)-braid group case}
The case \(n=3\) exhibits additional phenomena and requires a more refined analysis combining surface subgroup obstructions and decomposition techniques.

We begin by exhibiting undistorted surface subgroups in graph \(3\)-braid groups on planar graphs.

\begin{proposition}\label{proposition:K2k}
Let \(\sfK_{2,k}\) be the complete bipartite graph for \(k\ge 2\). 
Then \(\bbB_3(\sfK_{2,k})\) is free if and only if \(k\le 3\). If \(k\ge 4\), then \(\bbB_3(\sfK_{2,k})\) contains an undistorted surface subgroup, and in particular, is not free.
\end{proposition}
\begin{proof}
First, note that for any \(k\ge2\), \(\sfK_{2,k}\) is sufficiently subdivided for \(n=3\).

Since \(UD_3(\sfK_{2,2})\) is isomorphic to \(\sfK_{2,2}\) itself, we have \(\bbB_3(\sfK_{2,2})\isom\bbZ\).
In \(UD_3(\sfK_{2,3})\), every square contains an edge not shared with any other square, so the complex deformation retracts to a \(1\)-dimensional complex. Computing its Euler characteristic, \(\chi(UD_3(\sfK_{2,3}))=10-18+6=-2\), we conclude that \(\bbB_3(\sfK_{2,3})\) is a free group of rank \(3\).

For \(k\ge 4\), \(\sfK_{2,4}\) is a subgraph of \(\sfK_{2,k}\). By \Cref{Lem:embedding}, it suffices to prove that \(UD_3(\sfK_{2,4})\) is a closed hyperbolic surface.
Indeed, every edge in \(UD_3(\sfK_{2,4})\) is contained in exactly two squares and the link of every vertex in \(UD_3(\sfK_{2,4})\) is homeomorphic to a circle, which implies that \(UD_3(\sfK_{2,4})\) is a closed surface.
Since the Euler characteristic is \(\chi(UD_3(\sfK_{2,4}))=20-48+24=-4\), \(UD_3(\sfK_{2,4})\) is a closed hyperbolic surface.
\end{proof}

We now describe how attaching a hanging edge induces a group-theoretic decomposition. Let \(n\ge2\), and let \(\graf'\) be obtained from \(\graf\) by attaching an \(n\)-path \(\sfP=[\sfv_0,\dots,\sfv_n]\) at a vertex \(\sfv\in\graf\) of valency at least two, identifying \(\sfv_0\) with \(\sfv\). Assume that both \(\graf\) and \(\graf'\) are sufficiently subdivided for \(n+1\), and denote the components of \(\graf\setminus\sfv\) by \(\graf_1,\dots,\graf_k\).

\begin{theorem}\label{theorem:hanging edge}
The group \(\bbB_n(\graf')\) is obtained from \(\bbB_n(\graf)\) by iterated HNN extensions over groups of the form
\[\prod_{i=1}^k\bbB_{m_i}(\graf_i),\qquad m_i\ge0,\qquad 0\le\sum_{i=1}^k m_i\le n-2.\]
\end{theorem}

\begin{proof}
Write \(\sfD=\graf\setminus\sfv=\bigsqcup_i\graf_i\) and \(\sfW=\lk_\graf(\sfv)\). Let \(\sfe=[\sfv_0,\sfv_1]\) and \(\sfP'=[\sfv_1,\dots,\sfv_n]\).
According to whether a cell of \(UD_n(\graf')\) contains \(\sfe\), we obtain a graph-of-spaces decomposition with vertex spaces \(V_r\) and cylinders \(H_r\times\sfe\), where 
\[V_r=UD_r(\graf)\itimes UD_{n-r}(\sfP'),\qquad H_r=UD_r(\sfD)\itimes UD_{n-r-1}(\sfP'\setminus\sfv_1).\]
The endpoint inclusions give the diagram
\[V_{r+1}\longleftarrow H_r\times\{\sfv_0\} \longrightarrow H_r\times\sfe \longleftarrow H_r\times\{\sfv_1\} \longrightarrow V_r\]
All the path configuration spaces appearing here are nonempty and contractible. Projecting away these factors and collapsing the terminal vertex and edge \(UD_0(\graf)=UD_0(\sfD)=\{*\}\), we obtain
\[UD_n(\graf')\simeq
\hocolim\left(
\begin{tikzcd}[column sep=0pc]
UD_n(\graf) && UD_{n-1}(\graf) && \cdots && UD_{1}(\graf)\\
& UD_{n-1}(\sfD)\ar[lu,"\iota_{n-1}"]\ar[ru,"j_{n-1}"']
&&UD_{n-2}(\sfD)\ar[lu,"\iota_{n-2}"]\ar[ru,"j_{n-2}"']
&&UD_{1}(\sfD)\ar[lu,"\iota_1"]\ar[ru,"j_1"']
\end{tikzcd}\right),
\]
where \(\hocolim\) denotes the homotopy colimit, \(\iota_r(q)=q\cup\{\sfv\}\), and \(j_r(q)=q\). Let \(Z_n=UD_{n}(\graf)\) and 
\[Z_r=Z_{r+1}\cup^h_{UD_{r}(\sfD)}UD_r(\graf)\quad \text{for}\quad r=n-1,\dots,1,\] 
where the first map factors through the distinguished copy of \(UD_{r+1}(\graf)\) in \(Z_{r+1}\). Thus \(UD_n(\graf')\simeq Z_1\). Both \(\iota_r\) and \(j_r\) are locally isometric embeddings, so \(\pi_1(UD_{r+1}(\graf))\to\pi_1(Z_{r+1})\) is injective.

Let \(A_r\) be the closure of \(UD_r(\graf)\setminus UD_{r}(\sfD)\), and put \(B_r=A_r\cap UD_{r}(\sfD)\). Since \(UD_r(\graf)=UD_{r}(\sfD)\cup_{B_r}A_r\) is a union of subcomplexes, we have
\[Z_r=\hocolim\left(Z_{r+1}\leftarrow UD_{r}(\sfD)\rightarrow UD_r(\graf)\right) \simeq\hocolim\left(Z_{r+1}\leftarrow B_r\rightarrow A_r\right).\]
To describe this homotopy colimit while retaining the identifications inside \(B_r\), put \(\sff_\sfw=[\sfv,\sfw]\) for \(\sfw\in \sfW\). Then
\[A_r=(UD_{r-1}(\sfD)\itimes\sfv)\cup\bigcup_{\sfw\in \sfW}(UD_{r-1}(\sfD\setminus\sfw)\itimes\sff_\sfw), \qquad B_r=\bigcup_{\sfw\in \sfW}(UD_{r-1}(\sfD\setminus\sfw)\itimes\sfw).\]
Equivalently, \(UD_r(\graf)\) is obtained from \(UD_{r}(\sfD)\sqcup(UD_{r-1}(\sfD)\times\{\sfv\})\) by attaching the cylinders \(UD_{r-1}(\sfD\setminus\sfw)\times[0,1]\) along \((q,0)\mapsto q\cup\{\sfv\}\) and \((q,1)\mapsto q\cup\{\sfw\}\).
The outer endpoint images may overlap in \(UD_{r}(\sfD)\); all these identifications are retained. Since \(UD_{r}(\sfD)\hookrightarrow UD_r(\graf)\) is a cofibration, replacing its vertex space by \(Z_{r+1}\) gives a model for \(Z_r\) with endpoint maps
\[UD_{r-1}(\sfD\setminus\sfw)\hookrightarrow UD_{r-1}(\sfD),\qquad \eta_\sfw:UD_{r-1}(\sfD\setminus\sfw)\rightarrow UD_{r+1}(\graf)\hookrightarrow Z_{r+1}, \quad q\mapsto q\cup\{\sfw,\sfv\}.\]

For \(\mathbf m=(m_1,\dots,m_k)\) with \(m_i\ge0\) and \(|\mathbf m|=r-1\), set
\[UD_{\mathbf m}(\sfD)=\prod_iUD_{m_i}(\graf_i),\qquad E_{\sfw,\mathbf m}
=UD_{m_j}(\graf_j\setminus\sfw)\times\prod_{i\ne j}UD_{m_i}(\graf_i)\quad(\sfw\in\graf_j).\]
These are the corresponding components of \(UD_{r-1}(\sfD)\) and \(UD_{r-1}(\sfD\setminus\sfw)\). The subdivision assumption implies that \(\sfw\) is a leaf of \(\graf_j\), and removing it leaves a graph sufficiently subdivided for \(m_j\le n-2\). Shortening a terminal path induces a homotopy equivalence on topological configuration spaces, so \Cref{Abrams} implies that the inclusion \(E_{\sfw,\mathbf m}\hookrightarrow UD_{\mathbf m}(\sfD)\) is a homotopy equivalence. The map \(\eta_\sfw\) is a locally isometric embedding into \(UD_{r+1}(\graf)\), followed by a \(\pi_1\)-injective map. Thus \(Z_r\) has the graph-of-spaces decomposition
\[Z_{r+1} \xlongleftarrow{\eta_\sfw} E_{\sfw,\mathbf m} \xlongrightarrow{\simeq} UD_{\mathbf m}(\sfD) \qquad(\sfw\in \sfW,\ |\mathbf m|=r-1),\]
with one common vertex space \(Z_{r+1}\), one vertex space \(UD_{\mathbf m}(\sfD)\) for each \(\mathbf m\), and one edge for each \((\sfw,\mathbf m)\).
For each \(\mathbf m\), choose one of its edges in a maximal tree. Its edge homomorphism into \(\pi_1(UD_{\mathbf m}(\sfD))\) is an isomorphism, so adjoining that vertex and edge does not change the fundamental group. The remaining \(|\sfW|-1\) edges each give an HNN extension over \(\pi_1(UD_{\mathbf m}(\sfD))\cong\prod_i\bbB_{m_i}(\graf_i)\). Iterating for \(r=n-1,\dots,1\) proves the theorem.
\end{proof}

This result shows how elementary graph operations translate into iterated HNN extensions of graph braid groups. The following two corollaries are direct consequences of \Cref{theorem:hanging edge} and its proof.

\begin{corollary}\label{corollary:free factor}
The group \(\bbB_n(\graf')\) has a free factor \(\bbF_d\), where \(d=\val_\graf(\sfv)-1\). In particular, if \(n=2\), then
\[\bbB_2(\graf')\cong\bbB_2(\graf)*\bbF_d.\]
\end{corollary}
\begin{proof}
At the last stage \(r=1\), the only distribution is \(\mathbf m=0\), and its vertex and edge groups are trivial. The remaining \(d\) edges therefore contribute a free factor \(\bbF_d\). For \(n=2\), this is the only stage.
\end{proof}

\begin{corollary}\label{corollary:B3 hanging edge}
If \(2\le n\le3\) and every \(\graf_i\) is a tree, then \(\bbB_n(\graf')\cong\bbB_n(\graf)*\bbF_N\) for some \(N\ge0\). In particular, \(\bbB_3(\sfT)\) is free for every tree \(\sfT\).
\end{corollary}
\begin{proof}
Every edge group in the theorem is trivial, since \(\sum_i m_i\le n-2\le1\) and each \(\graf_i\) is a tree. The first assertion follows. The second follows by constructing a tree from a path by successive attachments of hanging paths, with subdivisions taken as needed.
\end{proof}

In particular, the previous corollary recovers \cite[Theorem~1.4]{LSUTeam25}.

\begin{proof}[Proof of \Cref{theorem:freeness classification} for \(n=3\)]
Suppose that \(\bbB_3(\graf)\) is free. Then it is hyperbolic. Hence, by \Cref{corollary:hyperbolicity}, every cycle in \(\graf\) must pass through all essential vertices. In particular, \(\graf\) cannot contain a subdivision of either \(\sfK_{3,3}\) or \(\sfK_5\), and is therefore planar.
If \(\graf\) contains at least three essential vertices and at least two cycles, then one can find a cycle disjoint from an essential vertex, contradicting hyperbolicity.
Thus, one of the following cases must occur:
\begin{enumerate}
\item \(\graf\) contains no cycles, i.e., it is a tree;
\item \(\graf\) contains only one essential vertex, i.e., it is a subdivision of an elementary bunch of grapes;
\item \(\graf\) contains a unique cycle passing through all essential vertices; or
\item \(\graf\) contains at least two cycles and exactly two essential vertices, and every cycle passes through both; equivalently, \(\graf\) is homeomorphic to a graph obtained from \(\sfK_{2,k}\) (\(k\ge3\)) by attaching edges to its essential vertices.
\end{enumerate}
In Case~(4), \Cref{proposition:K2k} implies that \(k=3\). This proves the necessity.

Conversely, in Cases~(1) and~(2), \(\bbB_3(\graf)\) is free by \Cref{corollary:B3 hanging edge,theorem:elementary rank}. 
In Case~(3), repeated applications of \Cref{corollary:B3 hanging edge} show that \(\bbB_3(\graf)\cong\bbB_3(\sfC_m)\ast\bbF_N\) for some \(m\ge 3\) and \(N\ge 0\);
since \(\bbB_3(\sfC_m)\cong\bbZ\), it follows that \(\bbB_3(\graf)\) is free.
In Case~(4) with \(k=3\), repeated applications of \Cref{corollary:B3 hanging edge} yield \(\bbB_3(\graf)\cong\bbB_3(\sfK_{2,3})\ast\bbF_N\) for some \(N\ge0\).
Since \(\bbB_3(\sfK_{2,3})\cong\bbF_3\), it follows that \(\bbB_3(\graf)\) is free.
\end{proof}

\section{Weakly special square complexes and their hierarchy}\label{Section:Hierarchy}

We now turn to the product-region geometry of weakly special square complexes, which provides the framework for our study of graph \(2\)-braid groups through unordered discrete \(2\)-configuration spaces in \Cref{section:properties}.
Since every compact weakly special square complex admits a finite-sheeted two-sided weakly special cover \cite[paragraph following Definition~5.3]{Hua17QuasiFlat}, we assume throughout this section that \(Y\) and its indexed and primed variants are compact connected two-sided weakly special square complexes. We use \(Z\) and \(Z'\) to denote either such a complex or its universal cover.

We first recall product structures and intersection complexes, together with their behavior under passage to universal covers and quasi-isometries.
We then introduce the union \(\m Y\) of the images of maximal product structures and the \(\m Y\)-hierarchy, which records geometric and algebraic properties of the inclusion \(\m Y\hookrightarrow Y\).
We also examine the roles of the embedded product property and the standard intersection property.

\subsection{Product structures and intersection complexes}\label{Subsection:Intersection complexes}

The definitions of product structures, intersection complexes, and their morphisms recalled below apply more generally to connected NPC square complexes that are either compact or \(\CAT(0)\); see \cite{OhProducts}.
We work in the above setting to apply the completion, deck-quotient, and quasi-isometry results.

\subsubsection{Product structures and elevations}

\begin{definition}[{\cite[Definition~2.3]{OhProducts}}]\label{Def:ProductSubcomplex}
A \emph{product structure} into \(Z\) is a local isometry \(\frs:\grafd_1\times\grafd_2\to Z\) whose restrictions to all coordinate fibers are embeddings. Its equivalence class under precomposition with product isomorphisms\footnote{Here, a product isomorphism is a product of graph
isomorphisms, allowing interchange of the factors.} is a \emph{product subcomplex}. The structure and its class are \emph{standard} if both graph factors are leafless.
For product structures \(\frs:P\to Z\) and \(\frs':P'\to Z\), write \(\frs\preceq\frs'\) if \(\frs=\frs'\circ h\) for a product embedding \(h:P\hookrightarrow P'\).
\end{definition}

The graph factors are finite when \(Z=Y\). When \(Z=\bar Y\), they are trees and every product structure is an isometric embedding. When no confusion can arise, we denote a product subcomplex by a representative product structure. We write \(\grafd_1\itimes \grafd_2\) when the order of the factors is unspecified.

\begin{definition}[\Cref{Def:ProductSubcomplex}, continued]\label{SPS}
A \emph{maximal product subcomplex} of \(Z\) is a standard product subcomplex represented by \(\frs\) such that, for every standard product structure \(\frs'\), the relation \(\frs\preceq\frs'\) implies that \(\frs\) and \(\frs'\) are equivalent.
\end{definition}

\begin{lemma}[Passage to subcomplexes]\label{lem:SpecialSubcomplexes}
Let \(X_0\) be a subcomplex of an (NPC) square complex \(X\).
\begin{enumerate}
\item Every component of \(X_0\) is NPC. Weak specialness and two-sidedness pass from \(X\) to \(X_0\).
\item If \(X\) is special and any two distinct hyperplanes cross in at most one square, then each component of \(X_0\) is special.
\item Composing a product structure into \(X_0\) with the inclusion \(X_0\hookrightarrow X\) gives a product structure into \(X\). Conversely, a product structure into \(X\) whose image lies in \(X_0\) restricts to a product structure into \(X_0\). These operations preserve standardness and product factorizations.
\end{enumerate}
\end{lemma}
\begin{proof}
Vertex links in \(X_0\) are subgraphs of those in \(X\), so are flag. The hyperplane assertions in the first item follow by restricting hyperplanes and their transverse orientations.
For the second, if two hyperplanes of \(X_0\) inter-osculate, then their ambient hyperplanes cross in a square \(Q\subseteq X_0\). Specialness supplies a square \(Q'\subseteq X\) at the osculation, and uniqueness of crossings gives \(Q'=Q\subseteq X_0\), a contradiction.
For the last item, the composite with inclusion is a cubical immersion and hence a local isometry by \Cref{Lem:ImmersionImpliesLocalisometry}; coordinate-fiber embeddings are unchanged. Restricting the target gives the converse, and all factorization equations are unchanged.
\end{proof}

We recall the following completion and elevation properties from \cite[Lemmas~2.7, 2.9 and Proposition~2.10]{OhProducts}. The completion statement uses the weakly special hypothesis.

\begin{lemma}[Completion and elevations]\label{ProjofPS}
\begin{enumerate}[wide]
\item\label{push-forward} (Completion)
For every product structure \(\mathfrak{t}:\sfT_1\times \sfT_2\to\bar Y\), there is a product structure \(\frs:\grafd_1\times \grafd_2\to Y\) and an elevation \(\bar{\frs}\) such that
\[\mathfrak{t}\preceq\bar{\frs},\qquad p_Y(\image \mathfrak{t})=\image\frs=p_Y(\image \bar{\frs}).\]
If \(\mathfrak{t}\) is standard, so are \(\frs\) and \(\bar{\frs}\). If \(\mathfrak{t}\) is already an elevation, then it is equivalent to this completion.
\item\label{pull-back}
Every elevation of a product structure is an isometric embedding. If \(\frs:P\to Y\) is an embedding, then its elevations identify \(\bar P\) with the connected components of the covering pullback \(\bar Y\times_Y P\), or equivalently of \(p_Y^{-1}(\frs(P))\).
\item Two product structures into \(Z\) satisfy \(\frs\preceq\frs'\) if and only if they have elevations with nested images. For such elevations, their inclusion is the elevation of a product embedding realizing this factorization.
\end{enumerate}
\end{lemma}

\begin{lemma}[{\cite[Lemma~2.6 and Proposition~2.10]{OhProducts}}]\label{Lem:MaximalProducts}
Every standard product subcomplex is contained, with respect to \(\preceq\), in a maximal product subcomplex. The maximal product subcomplexes of \(\bar Y\) are precisely the elevations of maximal product subcomplexes of \(Y\), up to equivalence. Downstairs classes correspond to deck orbits of upstairs classes.
\end{lemma}

Completion can enlarge an arbitrary standard product subcomplex upstairs. Maximal product subcomplexes, and the core maps used below, are already represented by elevations.

\subsubsection{Intersection complexes}

\begin{definition}[{\cite[Definition~3.3]{OhProducts}}]\label{IC}
The \emph{intersection complex} \(\cI(Z)\) of \(Z\) is defined by the following simplex-and-incidence data:
\begin{enumerate}
\item
Each vertex is a maximal product subcomplex of \(Z\) and a \(k\)-simplex \(\sigma\) is represented by a nondegenerate maximal common-factor diagram
\[\frs_i\circ\theta_i=\frs_\sigma:P_\sigma\longrightarrow Z\qquad(0\le i\le k),\]
where \(\frs_i:P_i\to Z\) represent maximal product subcomplexes, \(\frs_\sigma\) is a standard product structure, and the \(\theta_i:P_\sigma\hookrightarrow P_i\) are product embeddings.
The diagram is \emph{maximal} if it admits no proper common extension. It is \emph{nondegenerate} if, for distinct \(i,j\), there is no product isomorphism \(\theta:P_i\to P_j\) with \(\frs_j\circ\theta=\frs_i\) and \(\theta\circ\theta_i=\theta_j\).

\item
Diagrams are considered up to product isomorphisms over \(Z\) and permutations of their indexed arrows. Face occurrences are obtained by restricting the indexed family and taking its maximal common extension.
Repeated vertex classes, distinct face occurrences, and \emph{diagram symmetries} are retained; the latter are the index permutations induced by diagram automorphisms over \(Z\).

\item
We call \(\frs_\sigma\) the \emph{core map} of \(\sigma\) and \(P_\sigma\) its \emph{core domain}. The \emph{core-map label} of \(\sigma\) is
\[\Lb\sigma=[\frs_\sigma],\]
with the domain and the factorization maps associated to face incidences retained as part of the label data.
When \(\frs_\sigma\) is an embedding, we identify its domain with its image and call this subcomplex the \emph{core} of \(\sigma\).
\end{enumerate}
The realization \(|\cI(Z)|\) is obtained by gluing the indexed standard simplices according to the face incidences and diagram isomorphisms.
Whenever this construction gives a triangular complex with the specified simplices as its cells, we regard \(\cI(Z)\) as that complex equipped with its core-map labels.
\end{definition}

When \(Z=Y\), the simplex-and-incidence data are finite, but the natural maps from closed simplices to \(|\cI(Z)|\) need not be injective.
A sufficient condition for the intersection complex to be a triangular complex will be given in \Cref{prop:EPPquotient}.

\begin{lemma}[{\cite[Lemma~2.8 and Proposition~3.4]{OhProducts}}]\label{Lem:IntersectionCores}
For \(Z=\bar Y\), the intersection complex \(\cI(Z)\) is a labeled simplicial complex. Distinct vertices represented by \(\bar{\frs}_0,\ldots,\bar{\frs}_k\) span a simplex \(\sigma\) exactly when \(\bigcap_{i=0}^k\image(\bar{\frs}_i)\) contains a \(2\)-flat.
Its core map \(\frs_\sigma:P_\sigma\to Z\) is an isometric embedding with
\[\image\frs_\sigma=\operatorname{Core}(\bigcap_{i=0}^k\image(\bar{\frs}_i)):=\bigcup\{F\subseteq \bigcap_{i=0}^k\image(\bar{\frs}_i):F\text{ is a \(2\)-flat}\}.\]
The core map is represented by an elevation of a standard product structure downstairs. Face incidences correspond to the factorizations given by nested core subcomplexes.
\end{lemma}

For a finite face chain \(\sigma_0\subsetneq\cdots\subsetneq\sigma_k\) with specified face incidences, choose compatible elevations of its core maps and write \(\bar{P_i}\) for their domains. These domains are products of leafless trees, and the lifted factorization maps identify them with a nested chain \(\bar{P_k}\subseteq\cdots\subseteq\bar{P_0}\). 
Different compatible choices are related by product isometries commuting with the inclusions. When \(Z=\bar Y\), the core subcomplexes themselves provide this chain.

\begin{definition}[{(Iso)morphism, \cite[Definition~3.5]{OhProducts}}]\label{Def:Morphism}
A map \(\Phi:\cI(Z)\to\cI(Z')\) is a \emph{morphism} if it satisfies the following conditions:
\begin{enumerate}
\item
\(\Phi\) assigns to each indexed simplex \(\sigma\) a simplex \(\Phi(\sigma)\), together with a bijection between their index sets. These assignments respect face incidences and diagram symmetries.
We call a map satisfying this condition a \emph{combinatorial map of intersection complexes}. In particular, such a map preserves simplex dimensions.

\item
Every finite face chain as above admits product quasi-isometries \(\psi_i:\bar{P_i}\to\bar{P'_i}\), where \(\bar{P_i}\) and \(\bar{P'_i}\) are compatible elevated core domains associated to \(\sigma_i\) and \(\Phi(\sigma_i)\), respectively.
These maps must use one matching of the factors throughout the chain and satisfy
\[\sup_{x\in\bar{P_{i+1}}}d_{\bar{P'_i}}\bigl(\psi_i(x),\,\psi_{i+1}(x)\bigr)<\infty\qquad(0\le i<k).\]
The inclusions are understood, and the maps and bounds may depend on the face chain. 
\end{enumerate}
An \emph{isomorphism} is a morphism with invertible underlying combinatorial map.
\end{definition}

Morphisms compose, and isomorphisms have morphism inverses.
Along a specified face incidence, equality of labels means that the specified product embedding of core domains is onto, not merely that the domains are abstractly isomorphic. 
Morphisms preserve and reflect this equality, as well as equality of individual matched factors: nested leafless subtrees at finite Hausdorff distance coincide, as do standard product subcomplexes in a \(\CAT(0)\) square complex \cite[Lemma~2.4]{OhProducts}.
Thus, for a specified face incidence \(\tau\subsetneq\sigma\),
\[\Lb\sigma=\Lb\tau\quad\Longleftrightarrow\quad\Lb{\Phi(\sigma)}=\Lb{\Phi(\tau)}.\]

The following two results describe the deck quotient and quasi-isometry invariance of intersection complexes.

\begin{theorem}[{\cite[Theorem~3.8]{OhProducts}}]\label{Thm:TPBCM}
The deck action on \(\cI(\bar Y)\) is by isomorphisms. There is a quotient morphism
\[\rho_Y:\cI(\bar Y)\longrightarrow\cI(Y),\qquad|\cI(Y)|\cong|\cI(\bar Y)|/\pi_1(Y),\]
where the quotient retains indexed face incidences and diagram symmetries. The downstairs intersection complex is finite, and the upstairs intersection complex is finite-dimensional.
\end{theorem}

\begin{theorem}[{\cite[Theorems~4.4 and~4.6]{OhProducts}}]\label{MaxtoMax}\label{theorem:IsobetInt}
Every quasi-isometry \(\phi:\bar Y\to\bar{Y'}\) induces an isomorphism \(\cI(\phi):\cI(\bar Y)\to\cI(\bar{Y'})\). For every simplex \(\sigma\),
\[d_H\bigl(\phi(\image \frs_\sigma),\image \frs'_{\cI(\phi)(\sigma)}\bigr)\le D,\]
where \(D\) depends only on the ambient complexes and quasi-isometry constants. In particular, its vertex correspondence is the unique correspondence between maximal product subcomplexes satisfying this uniform Hausdorff bound.
\end{theorem}

The previous theorem is proved in \cite{OhProducts} using results on the preservation of top-dimensional flats up to uniformly bounded Hausdorff distance \cite{BKS(b),Hua17QuasiFlat}.

\subsection{The union of product regions}\label{Subsection:Union of product regions}
We now study the union of product regions and its inclusion into \(Y\).

\begin{definition}\label{definition:union of maximals}
Define
\[\m Y=\bigcup_{[\frs]\text{ maximal in \(Y\)}}\image(\frs)\subseteq Y,\]
where the union ranges over all maximal product subcomplexes of \(Y\), represented by \(\frs\).
By \Cref{Lem:MaximalProducts}, this is also the union of the images of all standard product structures into \(Y\). The subcomplex \(\m Y\) may be empty or disconnected.
\end{definition}

Intersection complexes of disconnected subcomplexes are interpreted componentwise, with \(\cI(\varnothing)=\varnothing\).

\begin{lemma}\label{lemma:Y max facts}
The components of \(\m Y\) are compact two-sided weakly special square complexes, and
\[(\m Y)^{\mathsf{max}}=\m Y,\qquad \cI(\m Y)\cong\cI(Y).\]
Moreover, \(p_Y^{-1}(\m Y)\) is the union of the maximal product subcomplexes of \(\bar Y\), and
\[\m Y\ne\varnothing\quad\Longleftrightarrow\quad \cI(Y)\ne\varnothing\quad\Longleftrightarrow\quad\pi_1(Y)\text{ is non-hyperbolic}.\]
\end{lemma}
\begin{proof}
By \Cref{lem:SpecialSubcomplexes}, composition with inclusion and restriction of the target identify the standard product structures into \(\m Y\) and \(Y\), preserving all factorizations. Every such structure into \(Y\) has image in \(\m Y\) by definition. Their maximal elements and common-factor diagrams therefore agree, proving the first two identities. The assertions about the components follow from \Cref{lem:SpecialSubcomplexes}.

By \Cref{ProjofPS,Lem:MaximalProducts}, the maximal product subcomplexes upstairs are precisely the elevations of the downstairs maximal structures. Their union is exactly \(p_Y^{-1}(\m Y)\).
A standard product structure downstairs has finite leafless graph factors, so the fundamental group of its domain contains \(\bbZ^2\); local isometry makes this a subgroup of \(\pi_1(Y)\). Conversely, the Flat Plane Theorem \cite[Theorem~III.H.1.5]{BH} gives a \(2\)-flat in \(\bar Y\) when \(\pi_1(Y)\) is non-hyperbolic, and completion produces a standard product structure downstairs.
\end{proof}

For the following covering construction and lemma, assume that \(\m Y\) is nonempty and connected. Write \(\m\iota:\m Y\hookrightarrow Y\), and put
\[\bbH=\pi_1(\m Y),\qquad \bbG=\pi_1(Y),\qquad \bbH_0=\ker \m\iota_*,\qquad \bbG_0=\image \m\iota_*.\]
The pullback of the universal covering is
\begin{equation}\label{eq:covering-pullback}
\begin{tikzcd}[column sep=large]
E=\m Y\times_Y\bar Y \arrow[r,hook,"j"]\arrow[d,"r"']&\bar Y\arrow[d,"p_Y"]\\
\m Y\arrow[r,hook,"\m\iota"']&Y.
\end{tikzcd}
\end{equation}
Thus \(E=p_Y^{-1}(\m Y)\), with components indexed by \(\bbG/\bbG_0\). Let \(E_0\) be the component containing a chosen basepoint lift, and put \(r_0=r|_{E_0}\).
Covering theory gives
\[q:\bar{\m Y}\longrightarrow E_0,\qquad E_0\cong\bar{\m Y}/\bbH_0,\qquad r_0\circ q=p_{\m Y}. \]
The corresponding lift of \(\m\iota\circ p_{\m Y}\) to \(\bar Y\) is
\[\bar{\m\iota}=j|_{E_0}\circ q:\bar{\m Y}\longrightarrow\bar Y,\]
where \(j:E\hookrightarrow\bar Y\) is the inclusion in the pullback diagram \eqref{eq:covering-pullback}.
Although \(q\) is a local isometry for the intrinsic metric of \(E_0\), \(\bar{\m\iota}\) need not be a local isometry. Let \(\mathcal J(E_0)\) be the full subcomplex of \(\cI(\bar Y)\) on maximal product subcomplexes contained in \(E_0\).

\begin{lemma}\label{lem:UnionofSPS}
The map \(\bar{\m\iota}\) induces a morphism
\[\cI(\bar{\m\iota}):\cI(\bar{\m Y})\longrightarrow\cI(\bar Y)\]
with image \(\mathcal{J}(E_0)\), inducing an isomorphism of underlying simplicial complexes \(\cI(\bar{\m Y})/\bbH_0\cong \mathcal{J}(E_0)\). Its maps on core domains are cubical isometries. 

Moreover, the following are equivalent:
\begin{enumerate}
\item \(\m\iota_*:\pi_1(\m Y)\to\pi_1(Y)\) is injective;
\item \(\bar{\m\iota}:\bar{\m Y}\to\bar Y\) is an embedding;
\item \(\cI(\bar{\m\iota})\) is injective.
\end{enumerate}
\end{lemma}
\begin{proof}
The map \(\bar{\m\iota}\) is a cubical immersion. Its restriction to any product subcomplex of \(\bar{\m Y}\) is a local isometry by \Cref{Lem:ImmersionImpliesLocalisometry}, hence an isometric embedding. Conversely, every standard product subcomplex of \(\bar Y\) projects into \(\m Y\), by completion. One contained in \(E_0\) is simply connected, so each component of its inverse image under \(q\) maps isometrically onto it and is a standard product subcomplex of \(\bar{\m Y}\).

These operations preserve maximality. Indeed, an extension of \(\bar{\m\iota}(M)\), for a maximal product subcomplex \(M\subseteq\bar{\m Y}\), lies in \(E_0\) and lifts to an extension of \(M\); maximality forces equality. Conversely, an extension of a lift of a maximal product subcomplex in \(E_0\) would map isometrically to an extension of that product subcomplex. This defines the vertex map of \(\cI(\bar{\m\iota})\).

Suppose that distinct maximal product subcomplexes \(M_0,\ldots,M_k\subseteq\bar{\m Y}\) have a common \(2\)-flat, and put \(M_i'=\bar{\m\iota}(M_i)\). The convex intersection \(W=\bigcap_iM_i'\) lifts through a chosen point of \(\bigcap_iM_i\).
Uniqueness of this lift inside each \(M_i\) gives
\[\bar{\m\iota}\left(\bigcap_iM_i\right)=\bigcap_iM_i'\]
isometrically, and hence identifies their cores. The \(M_i'\) are distinct, since distinct lifts of the same embedded product subcomplex under \(q\) are disjoint.
Conversely, lifting all vertices of a simplex of \(\mathcal J(E_0)\) through one point over its common core gives a simplex upstairs. The core isometries commute with face factorizations, so \(\cI(\bar{\m\iota})\) is a morphism with the stated image.

Two simplices with the same image differ by an element of \(\bbH_0\): align points over one point of the common core and use uniqueness of the product lifts. This proves the quotient assertion.
Covering theory gives \(\bbH_0=\{\mathrm{id}\}\) if and only if \(q\), equivalently \(\bar{\m\iota}\), is an embedding. When \(\bbH_0=\{\mathrm{id}\}\), the quotient assertion makes \(\cI(\bar{\m\iota})\) injective. Conversely, if \(\cI(\bar{\m\iota})\) is injective, then \(\cI(\bar{\m\iota})(h.M)=\cI(\bar{\m\iota})(M)\) implies \(h.M=M\) for every \(h\in \bbH_0\) and any maximal product subcomplex \(M\). Since \(q|_M\) is injective, \(h\) fixes \(M\) pointwise; the free deck action gives \(h=\mathrm{id}\). 
\end{proof}

Every simplex of \(\cI(\bar Y)\) has its maximal product subcomplexes in one component of \(E\). Consequently,
\[\cI(\bar Y)=\bigsqcup_{C\in\pi_0(E)}\mathcal{J}(C).\]
The image of \(\cI(\bar{\m\iota})\) is therefore one such subcomplex, and its fibers record the kernel of \(\m\iota_*\). The following example illustrates the possible failure of \(\pi_1\)-injectivity.

\begin{figure}[ht]
\centering
\begin{subfigure}{.7\textwidth}
\begin{align*}
Y_0&=\begin{tikzpicture}[baseline=-.5ex,rotate=-90,yscale=-1]
\draw[thick, gray] (-1,0) -- (1,0) (0,0) -- (0,1); 
\draw[thick] (-1,1) -- (1,1) node[midway, left] {\(\scriptstyle\partial^+ Y_0\)};
\draw[thick,blue] (-1,0) -- (-1,1) (1,0) -- (1,1);
\draw[fill, blue] (-1,0) circle (2pt) (1,0) circle (2pt);
\draw[thick, gray] (-0.8,-0.6) -- (0.8,0.6) (0,0) -- (0,-1);
\draw[thick,red] (-1.8,-0.6) -- ++(2,0) (-0.2,0.6) -- ++(2,0);
\draw[thick,blue] (-1.8,-.6) -- (-0.2,.6) (0.2,-.6) -- (1.8,.6);
\draw[fill, red] (-0.8,-0.6) circle (2pt) (0.8,0.6) circle (2pt);
\draw[thick] (-0.8,-1.6) -- (0.8,-0.4) node[midway, right] {\(\scriptstyle\partial^- Y_0\)};
\draw[thick,red] (-0.8,-0.6) -- ++(0,-1) (0.8,0.6)-- ++(0,-1);
\draw[thick, gray] (0,0) -- (0,-1);
\draw[fill] (0,0) circle (2pt) node[above left] {\(\scriptstyle\bfzero\)};
\end{tikzpicture}\bigg/\sim=
\begin{tikzpicture}[baseline=-.5ex]
\draw[thick] (0,0) circle (2.5 and 1.25);
\draw[thick, gray] (0,0) circle (2 and 1) (2.5,0) -- ++(-0.5,0) -- ++(-0.4,0);
\draw[thick, dashed] (1.2,0) arc (180:0:0.2 and 0.1);
\draw[thick,dashed] (1.2,0) arc (-180:0:0.2 and 0.1);
\draw[thick, red] (0,0) circle (0.8 and 0.4) (0.8,0) -- ++(0.4,0);
\draw[thick, blue] (-2,0) arc (180:0:0.6 and 0.4) (-2.5,0) -- ++(0.5,0);
\draw[fill, blue] (-2,0) circle (2pt);
\draw[thick, dashed, blue] (-2,0) arc (-180:0:0.6 and 0.4);
\draw[thick, gray] (2,0) arc (0:180:0.6 and 0.4);
\draw[thick, gray, dashed] (2,0) arc (0:-180:0.6 and 0.4);
\draw[fill] (2,0) circle (2pt);
\draw[fill,red] (0.8,0) circle (2pt);
\draw (0,1.2) node[above] {\(\scriptstyle\partial^+ Y_0\)};
\draw (1.4,0.2) node {\(\scriptstyle\partial^- Y_0\)};
\end{tikzpicture}
\end{align*}
\caption{The building block \(Y_0\)}
\label{figure:five tori 1}
\end{subfigure}
\begin{subfigure}{\textwidth}
\begin{align*}
Y'&=\begin{tikzpicture}[baseline=-.5ex,scale=0.8,transform shape]
\begin{scope}[rotate=-90,yscale=-1]
\draw[thick, gray] (-1,0) -- (1,0) (0,0) -- (0,1); 
\draw[thick, green] (-1,1) -- (1,1);
\draw[thick,cyan] (-1,0) -- (-1,1) (1,0) -- (1,1);
\draw[fill, cyan] (-1,0) circle (2pt) (1,0) circle (2pt);
\draw[thick, gray] (-0.8,-0.6) -- (0.8,0.6) (0,0) -- (0,-1);
\draw[thick,red] (-1.8,-0.6) -- ++(2,0) (-0.2,0.6) -- ++(2,0);
\draw[thick,cyan] (-1.8,-.6) -- (-0.2,.6) (0.2,-.6) -- (1.8,.6);
\draw[fill, red] (-0.8,-0.6) circle (2pt) (0.8,0.6) circle (2pt);
\draw[thick] (-0.8,-1.6) -- (0.8,-0.4);
\draw[thick,red] (-0.8,-0.6) -- ++(0,-1) (0.8,0.6)-- ++(0,-1);
\draw[thick, gray] (0,0) -- (0,-1);
\draw[fill] (0,0) circle (2pt) node[above=1ex, left=-.5ex, transform shape=false] {\(\scriptstyle\bfzero_1\)};
\end{scope}
\begin{scope}[xshift=2cm,rotate=90,yscale=-1]
\draw[thick, gray] (-0.8,-0.6) -- (0.8,0.6) (0,0) -- (0,-1); 
\draw[thick] (-1,1) -- (1,1);
\draw[thick,red] (-1.8,-0.6) -- ++(2,0) (-0.2,0.6) -- ++(2,0);
\draw[thick] (-0.8,-1.6) -- (0.8,-0.4);
\draw[thick,red] (-0.8,-0.6) -- ++(0,-1) (0.8,0.6)-- ++(0,-1);
\draw[thick,blue] (-1.8,-.6) -- (-0.2,.6) (0.2,-.6) -- (1.8,.6);
\draw[thick,blue] (-1,0) -- (-1,1) (1,0) -- (1,1);
\draw[thick, gray] (-1,0) -- (1,0) (0,0) -- (0,1);
\draw[fill, red] (-0.8,-0.6) circle (2pt) (0.8,0.6) circle (2pt);
\draw[fill, blue] (-1,0) circle (2pt) (1,0) circle (2pt);
\draw[fill] (0,0) circle (2pt) node[above=1ex, left=-.5ex, transform shape=false] {\(\scriptstyle\bfzero_2\)};
\end{scope}
\begin{scope}[xshift=4cm,rotate=-90,yscale=-1]
\draw[thick, gray] (-1,0) -- (1,0) (0,0) -- (0,1); 
\draw[thick] (-1,1) -- (1,1);
\draw[thick,blue] (-1,0) -- (-1,1) (1,0) -- (1,1);
\draw[fill, blue] (-1,0) circle (2pt) (1,0) circle (2pt);
\draw[thick, gray] (-0.8,-0.6) -- (0.8,0.6) (0,0) -- (0,-1);
\draw[thick,orange] (-1.8,-0.6) -- ++(2,0) (-0.2,0.6) -- ++(2,0);
\draw[thick,blue] (-1.8,-.6) -- (-0.2,.6) (0.2,-.6) -- (1.8,.6);
\draw[fill, orange] (-0.8,-0.6) circle (2pt) (0.8,0.6) circle (2pt);
\draw[thick] (-0.8,-1.6) -- (0.8,-0.4);
\draw[thick,orange] (-0.8,-0.6) -- ++(0,-1) (0.8,0.6)-- ++(0,-1);
\draw[thick, gray] (0,0) -- (0,-1);
\draw[fill] (0,0) circle (2pt) node[above=1ex, left=-.5ex, transform shape=false] {\(\scriptstyle\bfzero_3\)};
\end{scope}
\begin{scope}[xshift=6cm,rotate=90,yscale=-1]
\draw[thick, gray] (-0.8,-0.6) -- (0.8,0.6) (0,0) -- (0,-1); 
\draw[thick] (-1,1) -- (1,1);
\draw[thick,orange] (-1.8,-0.6) -- ++(2,0) (-0.2,0.6) -- ++(2,0);
\draw[thick] (-0.8,-1.6) -- (0.8,-0.4);
\draw[thick,orange] (-0.8,-0.6) -- ++(0,-1) (0.8,0.6)-- ++(0,-1);
\draw[thick,violet] (-1.8,-.6) -- (-0.2,.6) (0.2,-.6) -- (1.8,.6);
\draw[thick,violet] (-1,0) -- (-1,1) (1,0) -- (1,1);
\draw[thick, gray] (-1,0) -- (1,0) (0,0) -- (0,1);
\draw[fill, orange] (-0.8,-0.6) circle (2pt) (0.8,0.6) circle (2pt);
\draw[fill, violet] (-1,0) circle (2pt) (1,0) circle (2pt);
\draw[fill] (0,0) circle (2pt) node[above=1ex, left=-.5ex, transform shape=false] {\(\scriptstyle\bfzero_4\)};
\end{scope}
\begin{scope}[xshift=8cm,rotate=-90,yscale=-1]
\draw[thick, gray] (-1,0) -- (1,0) (0,0) -- (0,1); 
\draw[thick] (-1,1) -- (1,1);
\draw[thick,violet] (-1,0) -- (-1,1) (1,0) -- (1,1);
\draw[fill, violet] (-1,0) circle (2pt) (1,0) circle (2pt);
\draw[thick, gray] (-0.8,-0.6) -- (0.8,0.6) (0,0) -- (0,-1);
\draw[thick,pink] (-1.8,-0.6) -- ++(2,0) (-0.2,0.6) -- ++(2,0);
\draw[thick,violet] (-1.8,-.6) -- (-0.2,.6) (0.2,-.6) -- (1.8,.6);
\draw[fill, pink] (-0.8,-0.6) circle (2pt) (0.8,0.6) circle (2pt);
\draw[thick] (-0.8,-1.6) -- (0.8,-0.4);
\draw[thick,pink] (-0.8,-0.6) -- ++(0,-1) (0.8,0.6)-- ++(0,-1);
\draw[thick, gray] (0,0) -- (0,-1);
\draw[fill] (0,0) circle (2pt) node[above=1ex, left=-.5ex, transform shape=false] {\(\scriptstyle\bfzero_5\)};
\end{scope}
\begin{scope}[xshift=10cm,rotate=90,yscale=-1]
\draw[thick, gray] (-0.8,-0.6) -- (0.8,0.6) (0,0) -- (0,-1); 
\draw[thick,green] (-1,1) -- (1,1);
\draw[thick,pink] (-1.8,-0.6) -- ++(2,0) (-0.2,0.6) -- ++(2,0);
\draw[thick] (-0.8,-1.6) -- (0.8,-0.4);
\draw[thick,pink] (-0.8,-0.6) -- ++(0,-1) (0.8,0.6)-- ++(0,-1);
\draw[thick,cyan] (-1.8,-.6) -- (-0.2,.6) (0.2,-.6) -- (1.8,.6);
\draw[thick,cyan] (-1,0) -- (-1,1) (1,0) -- (1,1);
\draw[thick, gray] (-1,0) -- (1,0) (0,0) -- (0,1);
\draw[fill, pink] (-0.8,-0.6) circle (2pt) (0.8,0.6) circle (2pt);
\draw[fill, cyan] (-1,0) circle (2pt) (1,0) circle (2pt);
\draw[fill] (0,0) circle (2pt) node[above=1ex, left=-.5ex, transform shape=false] {\(\scriptstyle\bfzero_6\)};
\end{scope}
\end{tikzpicture}\,\,\bigg/\sim\quad\text{and}\quad
H=\begin{tikzpicture}[baseline=-.5ex,scale=0.8,transform shape]
\draw[thick, gray] (0:1) -- (60:1) -- (120:1) -- (180:1) -- (240:1) -- (300:1) -- cycle;
\foreach \x in {1,2,3,4,5,6} {
\draw (0,0) -- ({\x*60+30}:{sqrt(3)/2});
\draw[fill] ({\x*60}:1) circle (2pt);
\draw ({\x*60}:1.3) node {\(\scriptstyle\bfzero_\x\)};
}
\draw[fill] (0,0) circle (2pt);
\end{tikzpicture}
\end{align*}
\caption{The complex \(Y=Y'\cup H/\sim\)}
\label{figure:five tori 2}
\end{subfigure}
\caption{The union of six tori with laces.}
\label{figure:five tori}
\end{figure}

\begin{example}[\(\pi_1\)-non-injective]\label{example:pi1-non-injective}
Let \(Y_0\) be the quotient of the union of three rectangles 
\[Y_0 = \left(\{0\}\times[-1,1]\times[-1,1]\right)\cup \left([0,1]\times[-1,1]\times\{0\}\right) \cup \left([-1,0]\times\{0\}\times[-1,1]\right)\big/\sim,\]
where \(\sim\) is generated by the boundary identifications
\[(x,-1,z)\sim(x,1,z),\qquad(x,y,-1)\sim(x,y,1),\]
whenever the indicated points belong to the union.
We denote by \(\partial^\pm Y_0\) the images under the quotient of the subsets with first coordinate \(x=\pm1\), and by \(\bfzero\) the image of the origin. See \Cref{figure:five tori 1}.

For \(N\ge 5\), let \(Y_1,\dots, Y_N\) be \(N\) copies of \(Y_0\) with origins \(\bfzero_i\) for \(1\le i\le N\). Identifying \(\partial^+ Y_i\) with \(\partial^- Y_{i+1}\) in the natural way (where \(Y_{N+1}\) is understood to be \(Y_1\)) produces a complex \(Y'\) with a unique locally isometrically embedded cycle \(\gamma\subsetneq Y'\) of length \(2N\) passing through all \(\bfzero_i\).
Let \(H\) be the union of \(N\) squares arranged in an \(N\)-gon, and define \(Y\) by gluing \(Y'\) and \(H\) along the cycle \(\gamma\) and the boundary \(\partial H\); see \Cref{figure:five tori 2} for the case \(N=6\). Then \(\m Y=Y'\), and the inclusion \(\m\iota\colon \m Y\to Y\) is neither a local isometry nor \(\pi_1\)-injective, due to the presence of \(H\).

In either \(Y\) or \(Y'\), there are \(N\) maximal product subcomplexes \(M_i\), each isometric to the product of a handcuff graph and a circle,
\[
M_i\quad\isom\quad \begin{tikzpicture}[baseline=-.5ex]
\draw[thick] (-1,0) circle (0.5) (-0.5,0)--(0.5,0) (1,0) circle (0.5);
\draw[fill] (-1.5,0) circle (2pt) (-0.5,0) circle (2pt) (0.5,0) circle (2pt) (1.5,0) circle (2pt);
\end{tikzpicture}\,\times\,
\begin{tikzpicture}[baseline=-.5ex]
\draw[thick] (0,0) circle (0.5);
\draw[fill] (-0.5,0) circle (2pt) (0.5,0) circle (2pt);
\end{tikzpicture},
\]
arranged so that, for distinct \(i,j\), \(M_i \cap M_j \neq \varnothing\) precisely when \(j\equiv i\pm1\pmod{N}\), in which case it is a torus. Hence, \(\cI(Y)\) is an \(N\)-cycle, and the cycle \(\gamma\) can be regarded as a loop representing \(\cI(Y)\). 
Since \(\gamma\) is null-homotopic in \(Y\), every lift of \(\gamma\) to \(\bar Y\) is again a loop, and \(\cI(\bar Y)\) contains infinitely many disjoint copies of the \(N\)-cycle \(\cI(Y)\).
Consequently, the induced map \(\cI(\bar{\m\iota}):\cI(\bar{\m Y})\rightarrow \cI(\bar Y)\) is a morphism but is not injective.

If instead \(N=4\), then \(\m Y=Y\) is the product of two handcuff graphs.
In this case \(\cI(Y)\isom\cI(\m Y)\) consists of a single vertex, whereas \(\cI(Y')\) is a \(4\)-cycle. There is no contradiction here, since \(Y'\subsetneq \m Y\).
\end{example}

\subsubsection{The hierarchy}

The covering pullback separates the intrinsic universal cover of \(\m Y\) from its image in \(\bar Y\). The following hierarchy records geometric and algebraic properties of the inclusion \(\m\iota:\m Y\hookrightarrow Y\).

\begin{definition}[\(\m Y\)-hierarchy]\label{definition:Ymax hierarchy}
Let \(\cY\) be the class of compact connected two-sided weakly special square complexes. With \(\m Y\) as in \Cref{definition:union of maximals}, put
\[\cY_{(0)}=\{Y\in\cY:\m Y\text{ is nonempty and connected}\}.\]
For \(Y\in\cY_{(0)}\), write \(\m\iota:\m Y\hookrightarrow Y\) for inclusion and define
\begin{align*}
\cY_{(1)}&=\{Y\in\cY_{(0)}:\m\iota_*\text{ is injective}\},\\
\cY_{(2)}&=\{Y\in\cY_{(1)}:\pi_1(Y)=\image(\m\iota_*)*\bbF_N  \text{ for some }N\ge0\},\\
\cY_{(3)}&=\{Y\in\cY_{(0)}:\m\iota\text{ is a local isometry}\},\\
\cY_{(4)}&=\{Y\in\cY_{(0)}:\m Y\text{ contains every square of }Y\},\\
\cY_{(5)}&=\{Y\in\cY_{(0)}:\m Y=Y\}.
\end{align*}
The splitting in the definition of \(\cY_{(2)}\) has the indicated image subgroup as a free factor; it is not merely an abstract isomorphism of groups.
\end{definition}

These classes satisfy
\[\cY_{(5)}\subseteq\cY_{(4)}\subseteq\cY_{(2)}\cap\cY_{(3)},\qquad\cY_{(2)}\cup\cY_{(3)}\subseteq\cY_{(1)}.\]
Indeed, a local isometry is \(\pi_1\)-injective. If \(\m Y\) contains every square, then its vertex links are full in those of \(Y\), and \(Y\) is obtained from the connected subcomplex \(\m Y\) by attaching only vertices and edges. This gives both local convexity and the stated free product splitting.

\Cref{lem:UnionofSPS} interprets membership in \(\cY_{(1)}\) as injectivity of the induced intersection-complex map. Membership in \(\cY_{(3)}\) additionally gives an isometric embedding of universal covers. These interpretations use no embeddedness assumption on the downstairs product structures.

\begin{lemma}\label{Lemma:Well-defined}
For every \(Y\in\cY_{(0)}\), the subcomplex \(\m Y\) belongs to \(\cY_{(5)}\). In particular, \(Y\mapsto\m Y\) is an idempotent function from \(\cY_{(0)}\) to itself.
\end{lemma}
\begin{proof}
\Cref{lemma:Y max facts} gives \((\m Y)^{\mathsf{max}}=\m Y\), and the remaining assertions follow from compactness, connectedness, and inheritance of weak specialness.
\end{proof}

\begin{example}[Triangle-free RAAG]\label{example:raag}
Let \(\grafl\) be a finite, \(3\)-cycle-free, possibly disconnected, simple graph, and let \(\grafl_0\) denote the subgraph consisting of all edges of \(\grafl\).
If \(\grafl_0=\varnothing\), then \(\m S_\grafl=\varnothing\). Otherwise, \(\m S_\grafl=S_{\grafl_0}\), which is a compact, connected, special square complex.
Moreover, \(S_\grafl\not\in\cY_{(0)}\) if \(\grafl_0=\varnothing\), and \(S_\grafl\in\cY_{(4)}\) otherwise. 
Finally, \(S_\grafl\in\cY_{(5)}\) if and only if \(\grafl_0=\grafl\).
This example shows that the inclusion \(\cY_{(5)}\subsetneq\cY_{(4)}\) is proper.
\end{example}
 
Further examples in \Cref{figure:counterexamples in cY} together with \Cref{example:pi1-non-injective} demonstrate that the remaining inclusions in the hierarchy are also strict.

\begin{figure}[ht]
\centering
\begin{subfigure}{.3\textwidth}
\[
\begin{tikzpicture}[baseline=-.5ex]
\draw[thick] (-1,0) circle (1 and 0.5) (-1,0) circle (0.4 and 0.2);
\draw[thick] (-2,0) arc (180:0:0.3 and 0.2) (0,0) arc (0:180:0.3 and 0.2);
\draw[dashed,thick] (-2,0) arc (-180:0:0.3 and 0.2) (0,0) arc (0:-180:0.3 and 0.2);
\draw[fill] (-2,0) circle (2pt) (-1.4,0) circle (2pt) (-0.6,0) circle (2pt);
\draw[fill] (0,0) circle (2pt) (0,0.6) circle (2pt) (0.6,0) circle (2pt);
\draw[thick] (0,0) arc (0:180:0.3 and 0.2) arc (180:0:0.6) arc (0:180:0.3 and 0.2);
\end{tikzpicture}
\]
\caption{\(Y\in\cY_{(4)}\setminus\cY_{(5)}\)}
\end{subfigure}
\begin{subfigure}{.3\textwidth}
\[
\begin{tikzpicture}[baseline=-.5ex]
\draw[thick,dashed] (0,0) arc (0:60:1 and 0.5);
\draw[thick] ({-1+cos(60)},{0.5*sin(60)}) arc (60:360:1 and 0.5);
\draw[thick] (-1,0) circle (0.4 and 0.2);
\draw[thick] (-2,0) arc (180:0:0.3 and 0.2) (0,0) arc (0:180:0.3 and 0.2);
\draw[dashed,thick] (-2,0) arc (-180:0:0.3 and 0.2) (0,0) arc (0:-180:0.3 and 0.2);
\draw[fill] (-2,0) circle (2pt) (-1.4,0) circle (2pt) (-0.6,0) circle (2pt);
\draw[fill] (0,0) circle (2pt) (0,0.6) circle (2pt) (0.6,0) circle (2pt);
\draw[thick] (0,0) arc (0:180:0.3 and 0.2) arc (180:0:0.6) arc (0:180:0.3 and 0.2);
\end{tikzpicture}
\]
\caption{\(Y\in(\cY_{(2)}\cap\cY_{(3)})\setminus\cY_{(4)}\)}
\end{subfigure}
\begin{subfigure}{.3\textwidth}
\[
\begin{tikzpicture}[baseline=-.5ex]
\draw[thick,dashed] (0,0) arc (0:60:1 and 0.5);
\draw[thick] ({-1+cos(60)},{0.5*sin(60)}) arc (60:360:1 and 0.5);
\draw[thick] (-1,0) circle (0.4 and 0.2);
\draw[thick] (-2,0) arc (180:0:0.3 and 0.2) (0,0) arc (0:180:0.3 and 0.2);
\draw[dashed,thick] (-2,0) arc (-180:0:0.3 and 0.2) (0,0) arc (0:-180:0.3 and 0.2);
\draw[fill] (-2,0) circle (2pt) (-1.4,0) circle (2pt) (-0.6,0) circle (2pt);
\draw[thick] (-.6,0) arc (180:90:0.6);
\begin{scope}[xscale=-1]
\draw[thick,dashed] (0,0) arc (0:60:1 and 0.5);
\draw[thick] ({-1+cos(60)},{0.5*sin(60)}) arc (60:360:1 and 0.5);
\draw[thick] (-1,0) circle (0.4 and 0.2);
\draw[thick] (-2,0) arc (180:0:0.3 and 0.2) (0,0) arc (0:180:0.3 and 0.2);
\draw[dashed,thick] (-2,0) arc (-180:0:0.3 and 0.2) (0,0) arc (0:-180:0.3 and 0.2);
\draw[fill] (-2,0) circle (2pt) (-1.4,0) circle (2pt) (-0.6,0) circle (2pt);
\draw[fill] (0,0) circle (2pt) node (A) {};
\draw[thick] (-.6,0) arc (180:90:0.6);
\end{scope}
\draw[fill] (0,0.6) circle (2pt);
\end{tikzpicture}
\]
\caption{\(Y\in\cY_{(2)}\setminus\cY_{(3)}\)}
\end{subfigure}
\begin{subfigure}{.4\textwidth}
\[
\begin{tikzpicture}[baseline=-.5ex]
\begin{scope}[xscale=-1]
\draw[thick] (-1,0) circle (1 and 0.5) (-1,0) circle (0.4 and 0.2);
\draw[thick] (-2,0) arc (180:0:0.3 and 0.2) (0,0) arc (0:180:0.3 and 0.2);
\draw[dashed,thick] (-2,0) arc (-180:0:0.3 and 0.2) (0,0) arc (0:-180:0.3 and 0.2);
\draw[fill] (-2,0) circle (2pt) (-1.4,0) circle (2pt) (-0.6,0) circle (2pt);
\draw[fill] (0,0) circle (2pt) node (A) {};
\end{scope}
\begin{scope}[xshift=3cm]
\draw[thick] ({cos(90)},{0.5*sin(90)}) arc (90:270:1 and 0.5) -- ({1.5+cos(-90)},{0.5*sin(-90)}) arc (-90:90:1 and 0.5) -- cycle;
\draw[fill] ({0.75+cos(90)},{0.5*sin(90)}) circle (2pt) node(A) {}
({0.75+cos(-90)},{0.5*sin(-90)}) circle (2pt) node (B) {}
(2.5,0) circle (2pt)
(-0.4,0) circle (2pt)
(0,0.2) circle (2pt)
(0,-0.2) circle (2pt)
(1.5,0.2) circle (2pt)
(1.5,-0.2) circle (2pt)
(1.9,0) circle (2pt);
\draw[thick] (0,0) circle (0.4 and 0.2) (1.5,0) circle (0.4 and 0.2);
\draw[thick] (-0.4,0) arc (0:180:0.3 and 0.2)
(1.9,0) arc (180:0:0.3 and 0.2)
(0.75,0.5) arc (90:-90:0.1 and 0.5);
\draw[thick,dashed] (-0.4,0) arc (0:-180:0.3 and 0.2)
(1.9,0) arc (-180:0:0.3 and 0.2)
(0.75,0.5) arc (90:270:0.1 and 0.5);
\draw[thick] (0,0.2) to[out=60,in=180] (A.center)
(0,-0.2) to[out=-60,in=180] (B.center)
(1.5,0.2) to[out=120,in=0] (A.center)
(1.5,-0.2) to[out=-120,in=0] (B.center);
\draw[thick,dashed] (0,0.2) to[out=0,in=-120] (A.center)
(0,-0.2) to[out=0,in=120] (B.center)
(1.5,0.2) to[out=180,in=-60] (A.center)
(1.5,-0.2) to[out=-180,in=60] (B.center);
\end{scope}
\end{tikzpicture}
\]
\caption{\(Y\in\cY_{(3)}\setminus\cY_{(2)}\)}
\end{subfigure}
\begin{subfigure}{.4\textwidth}
\[
\begin{tikzpicture}[baseline=-.5ex]
\begin{scope}
\draw[thick,dashed] (0,0) arc (0:60:1 and 0.5);
\draw[thick] ({-1+cos(60)},{0.5*sin(60)}) arc (60:360:1 and 0.5);
\draw[thick] (-1,0) circle (0.4 and 0.2);
\draw[thick] (-2,0) arc (180:0:0.3 and 0.2) (0,0) arc (0:180:0.3 and 0.2);
\draw[dashed,thick] (-2,0) arc (-180:0:0.3 and 0.2) (0,0) arc (0:-180:0.3 and 0.2);
\draw[fill] (-2,0) circle (2pt) (-1.4,0) circle (2pt) (-0.6,0) circle (2pt);
\draw[thick] (-.6,0) arc (180:90:0.6);
\begin{scope}[xscale=-1]
\draw[thick,dashed] (0,0) arc (0:60:1 and 0.5);
\draw[thick] ({-1+cos(60)},{0.5*sin(60)}) arc (60:360:1 and 0.5);
\draw[thick] (-1,0) circle (0.4 and 0.2);
\draw[thick] (-2,0) arc (180:0:0.3 and 0.2) (0,0) arc (0:180:0.3 and 0.2);
\draw[dashed,thick] (-2,0) arc (-180:0:0.3 and 0.2) (0,0) arc (0:-180:0.3 and 0.2);
\draw[fill] (-2,0) circle (2pt) (-1.4,0) circle (2pt) (-0.6,0) circle (2pt);
\draw[fill] (0,0) circle (2pt) node (A) {};
\draw[thick] (-.6,0) arc (180:90:0.6);
\end{scope}
\draw[fill] (0,0.6) circle (2pt);
\end{scope}
\begin{scope}[xshift=3cm]
\draw[thick] ({cos(90)},{0.5*sin(90)}) arc (90:270:1 and 0.5) -- ({1.5+cos(-90)},{0.5*sin(-90)}) arc (-90:90:1 and 0.5) -- cycle;
\draw[fill] ({0.75+cos(90)},{0.5*sin(90)}) circle (2pt) node(A) {}
({0.75+cos(-90)},{0.5*sin(-90)}) circle (2pt) node (B) {}
(2.5,0) circle (2pt)
(-0.4,0) circle (2pt)
(0,0.2) circle (2pt)
(0,-0.2) circle (2pt)
(1.5,0.2) circle (2pt)
(1.5,-0.2) circle (2pt)
(1.9,0) circle (2pt);
\draw[thick] (0,0) circle (0.4 and 0.2) (1.5,0) circle (0.4 and 0.2);
\draw[thick] (-0.4,0) arc (0:180:0.3 and 0.2)
(1.9,0) arc (180:0:0.3 and 0.2)
(0.75,0.5) arc (90:-90:0.1 and 0.5);
\draw[thick,dashed] (-0.4,0) arc (0:-180:0.3 and 0.2)
(1.9,0) arc (-180:0:0.3 and 0.2)
(0.75,0.5) arc (90:270:0.1 and 0.5);
\draw[thick] (0,0.2) to[out=60,in=180] (A.center)
(0,-0.2) to[out=-60,in=180] (B.center)
(1.5,0.2) to[out=120,in=0] (A.center)
(1.5,-0.2) to[out=-120,in=0] (B.center);
\draw[thick,dashed] (0,0.2) to[out=0,in=-120] (A.center)
(0,-0.2) to[out=0,in=120] (B.center)
(1.5,0.2) to[out=180,in=-60] (A.center)
(1.5,-0.2) to[out=-180,in=60] (B.center);
\end{scope}
\end{tikzpicture}
\]
\caption{\(Y\in\cY_{(1)}\setminus(\cY_{(2)}\cup\cY_{(3)})\)}
\end{subfigure}
\caption{(Counter)examples in \(\cY_{(i)}\). In each case, the subcomplex \(\m Y\) is a union of tori. All edges are contained in squares, except in (A), where the \(4\)-cycle on the right does not bound a square.}
\label{figure:counterexamples in cY}
\end{figure}

\subsubsection{Embedded products and standard intersections}

We now introduce two properties that give a geometric description of product subcomplexes and their intersections.



\begin{definition}\label{definition:SIP_EPP}
We say that \(Y\) has the \emph{embedded product property} if every product structure into \(Y\) is an embedding, including those with non-leafless factors; in this case, we identify a product subcomplex with the embedded subcomplex given by its structure. 
For such \(Y\), we say that it has the \emph{standard intersection property} if every connected component of every nonempty finite intersection of maximal product subcomplexes is a standard product subcomplex.
\end{definition}

\begin{example}\label{Ex:EPP}
Two-dimensional Salvetti complexes and the two-dimensional connected components of \(D_2(\graf)\) and \(UD_2(\graf)\) have the embedded product property.
Indeed, these complexes are special and any two distinct hyperplanes cross in at most one square, so the unique crossing criterion applies; see \cite[Proposition~5.5 and the proof of Corollary~5.6]{OhProducts}.
\end{example}

The embedded product property passes to cubical subcomplexes by \Cref{lem:SpecialSubcomplexes}. When \(Y\) has this property, the standard intersection property holds for \(Y\) if and only if it holds for \(\m Y\). These are distinct conditions: embeddedness of the individual structures does not imply standardness of their intersections.

\begin{proposition}[Embedded description]\label{prop:EPPquotient}
Suppose that \(Y\) has the embedded product property. Then \(\cI(Y)\) is a finite labeled triangular complex. Distinct maximal product subcomplexes \(M_0,\ldots,M_k\) of \(Y\) determine one simplex for each component \(W\) of \(\bigcap_iM_i\) containing a standard product subcomplex. Its core label is the inclusion \(P_W\hookrightarrow Y\), where \(P_W\) is the largest standard product subcomplex in \(W\).

The deck action on \(\cI(\bar Y)\) is without inversions, and \(\rho_Y\) is injective on each closed simplex and sends maximal simplices to maximal simplices. The restriction of \(p_Y\) to each upstairs core is a universal covering of the corresponding embedded core subcomplex downstairs.
\end{proposition}
\begin{proof}
All standard product structures are embeddings, so the recovery theorem \cite[Theorem~5.2]{OhProducts} applies. Its common-factor diagrams are exactly the inclusions of intersection cores just described. Distinct elevations of one embedded maximal product subcomplex are disjoint. Thus the vertices of an upstairs simplex have distinct projected vertices, giving injectivity on closed simplices and excluding inversions. Any downstairs coface lifts through a point of its core in the chosen core lift, so maximal simplices are preserved. The core statement follows from \Cref{ProjofPS}.
\end{proof}

Under this identification, the core-map label of \(\sigma\) is the embedded core \(P_\sigma\subseteq Y\). Along a face incidence, equality of labels is equality under the specified inclusion.
In particular, equal nested downstairs cores have equal compatible elevations. The action determines a developable complex of groups over each downstairs component; an upstairs component is its universal development when it is simply connected. We use only this action-associated construction, as in \cite[Chapter~III.C]{BH}.

\begin{example}\label{Ex:disconnectedExamples}
Let \(Y=T_1^2\vee T_2^2\) be the wedge of two square tori. Then \(Y\) has the embedded product property, and \(\m Y=Y\) is connected.
However, its two maximal product subcomplexes intersect in a single vertex, so the standard intersection property fails and \(\cI(Y)\) consists of two isolated vertices.
\end{example}



The standard intersection property rules out this phenomenon and gives the following stronger conclusion upstairs.

\begin{lemma}[Standard intersections]\label{lem:Connected}
Suppose that \(Y\) has the embedded product and standard intersection properties, and that \(\m Y\) is nonempty and connected. Then \(\cI(\m Y)\cong\cI(Y)\) is connected and \(\cI(\bar{\m Y})\) is contractible. The associated complex of groups over \(\cI(\m Y)\) has universal development \(\cI(\bar{\m Y})\).
\end{lemma}
\begin{proof}
The elevations of the finitely many maximal product subcomplexes cover \(\bar{\m Y}\) by convex subcomplexes. This cover is locally finite, since each product subcomplex downstairs is embedded. Every nonempty finite intersection is convex; it is a component of the inverse image of the corresponding downstairs intersection, and the standard intersection property makes it standard. The nerve is therefore exactly \(\cI(\bar{\m Y})\). The nerve theorem for subcomplex covers, after triangulation, gives \(|\cI(\bar{\m Y})|\simeq\bar{\m Y}\); see \cite[Theorem~4.8]{BKRR23}. Since \(\bar{\m Y}\) is \(\CAT(0)\), the intersection complex is contractible. Its quotient \(\cI(\m Y)\) is connected, and \Cref{prop:EPPquotient} gives the asserted universal development.
\end{proof}

\begin{lemma}\label{lem:One-endedness}
Suppose that \(\m Y\) is nonempty and connected and has the embedded product property. If \(\cI(\bar{\m Y})\) is connected, then \(\bar{\m Y}\) is one-ended.
\end{lemma}
\begin{proof}
The maximal product subcomplexes give a locally finite cover of \(\bar{\m Y}\) by one-ended spaces. Adjacent vertices of the connected intersection complex correspond to product subcomplexes with unbounded convex intersection. Outside a sufficiently large ball, these intersections join the unbounded components of the product subcomplexes. Local finiteness of the cover gives one-endedness of \(\bar{\m Y}\).
\end{proof}

\begin{remark}\label{Rmk:thickness}
Under the hypotheses of \Cref{lem:One-endedness}, the same cover makes \(\bar{\m Y}\) thick of order at most one: there are finitely many deck orbits of maximal product subcomplexes, each such subcomplex is a product of two unbounded trees, and successive intersections are convex and unbounded. See \cite{BDM,BD14}.
\end{remark}

We now ask when the quasi-isometry type of \(\bar{\m Y}\) is determined by that of \(\bar Y\).

\begin{definition}[Quasi-isometry invariant]\label{Definition:QIinvariant}
Let \(\cY_1,\cY_2\subseteq\cY\).
A function \(\cF:\cY_1\to\cY_2\) is said to \emph{define a quasi-isometry invariant} (with respect to universal covers) if for all \(Y,Y'\in\cY_1\),
\[\bar Y\qisom \bar{Y'}\quad\Longrightarrow\quad \bar{\cF(Y)}\qisom \bar{\cF(Y')}.\]
It is said to be \emph{complete} if the converse implication also holds.
\end{definition}

\begin{example}
Let \(X\in\cY_{(5)}\) with \(\pi_1(X)\) one-ended, and let \(Y=X\vee S^1\), with the circle attached at a vertex. 
Every edge in the image of a standard product structure belongs to a square, so every such structure into \(Y\) has image in \(X\).
Since \(\m X=X\), we obtain
\[\m Y=X,\qquad \pi_1(Y)\cong\pi_1(X)\ast\bbZ.\]
By \Cref{PW}, the assignment \(Y\mapsto\m Y\) defines a complete quasi-isometry invariant on the class of complexes obtained in this way.
Nevertheless, \(\bar Y\) is infinitely-ended, whereas \(\bar{\m Y}\) is one-ended.
\end{example}

This leads us to ask how broadly this phenomenon occurs.

\begin{question}\label{Question:DefiningQII}
For which subclasses of \(\cY_{(0)}\) does the assignment \(Y\mapsto\m Y\) define a (complete) quasi-isometry invariant?
\end{question}

In the next section, we study this question for the discrete \(2\)-configuration spaces of graphs. Their product subcomplexes admit combinatorial descriptions, which we use to verify the geometric hypotheses needed for quasi-isometry invariance.
For normal bunches of grapes, \Cref{theorem:QIbetweenGBGs} will show that \(Y\mapsto\m Y\) defines a complete quasi-isometry invariant.


\section{Graph \texorpdfstring{\(2\)}{2}-braid groups and \texorpdfstring{\(UP_2(\graf)\)}{UP2(Gamma)}}\label{section:properties}
In this section, we apply the framework of \Cref{Section:Hierarchy} to graph \(2\)-braid groups on finite connected nontrivial simple graphs \(\graf\).
By \Cref{thm:GBGSpecial}, the unordered discrete configuration space \(UD_2(\graf)\) is a compact two-sided weakly special square complex, and hence belongs to \(\cY\).
By \Cref{Ex:EPP}, it also has the embedded product property, so the geometric description in \Cref{prop:EPPquotient} applies. The combinatorial description in \Cref{Lem:SPSinGBG,corollary:maximal} will allow us to study the position of \(UP_2(\graf)\) in the \(\m Y\)-hierarchy and its relation to the large-scale geometry of \(\bbB_2(\graf)\).


Recall that the homotopy type of \(UD_2(\graf)\) is invariant under the two operations on \(\graf\): smoothing (replacing two edges incident to a bivalent vertex by a single edge) and subdivision (the inverse operation), provided that \(\graf\) remains simple. 
For any simple graph \(\graf\), there exists a unique simple graph, up to isomorphism, with the minimal number of vertices among all simple graphs homeomorphic to \(\graf\); this graph is called the \emph{minimal simplicial model} of \(\graf\) \cite[Proposition~4.2]{AK2022}. 
Accordingly, we may restrict attention to graphs in the class 
\[\graphs \coloneqq \{\graf : \graf \text{ is the minimal simplicial model of a simple graph}\}.\]

\subsection{Maximal product subcomplexes for \texorpdfstring{\(UD_2(\graf)\)}{UD2(Gamma)}}
For \(\graf \in \graphs\), the intersection complexes \(\cI(UD_2(\graf))\) and \(\cI(\overline{UD_2(\graf)})\) are well defined. 
To describe their structure in terms of \(\graf\), we first characterize the standard product subcomplexes of \(UD_2(\graf)\) and of its universal cover \(\bar{UD_2(\graf)}\).

By \Cref{Ex:EPP}, product structures into \(UD_2(\graf)\) are embeddings. 
As discussed before \Cref{lemma:local isometric embedding into UD}, a pair of disjoint subgraphs \(\graf_1,\graf_2\subsetneq\graf\) determines a canonical product subcomplex \(\graf_1\itimes\graf_2\subseteq UD_2(\graf)\). 
The following lemma shows that every product subcomplex arises in this way.


\begin{lemma}\label{Lem:SPSinGBG}
Every product structure into \(UD_2(\graf)\) is an embedding whose image is \(\graf_1\itimes\graf_2\) for disjoint subgraphs \(\graf_1,\graf_2\subsetneq\graf\). Conversely, each such canonical inclusion is a product structure.
If \(K\) is its image, the images of its elevations are precisely the connected components of \(p_{UD_2(\graf)}^{-1}(K)\).
\end{lemma}
\begin{proof}
The converse follows from the canonical inclusion. Let \(\iota:\grafl_1\times\grafl_2\to UD_2(\graf)\) be a product structure.
By \Cref{lemma:local isometric embedding into UD}, there are disjoint subgraphs \(\graf_1,\graf_2\subsetneq\graf\) and surjective local isometries \(q_i:\grafl_i\to\graf_i\) such that \(\iota\) factors through the canonical inclusion \(\graf_1\times\graf_2\hookrightarrow UD_2(\graf)\). 
Coordinate-fiber injectivity makes each \(q_i\) injective, so both are graph isomorphisms. Thus \(\iota\) is an embedding with the stated image. The assertion about elevation images follows from \Cref{ProjofPS}.
\end{proof}

This description translates containment, standardness, and maximality into conditions on the defining subgraphs.

\begin{corollary}\label{corollary:maximal}
Let \(K=\graf_1\itimes\graf_2\subseteq UD_2(\graf)\) be a product subcomplex.
\begin{enumerate}
\item If \(K'=\grafl_1\itimes\grafl_2\) is a product subcomplex containing \(K\), then either \(\graf_1\subseteq\grafl_1,\ \graf_2\subseteq\grafl_2\) or \(\graf_1\subseteq\grafl_2,\ \graf_2\subseteq\grafl_1\).
\item \(K\) is standard if and only if \(\graf_1\) and \(\graf_2\) are connected, nontrivial and leafless.
\item \(K\) is maximal if and only if it is a standard product subcomplex and there is no pair of disjoint leafless subgraphs \(\graf_1'\) and \(\graf_2'\) such that \(\graf_1\subseteq\graf_1'\), \(\graf_2\subseteq\graf_2'\), and at least one of these inclusions is proper
\end{enumerate}
\end{corollary}
\begin{proof}
The proof follows from \Cref{SPS,Lem:SPSinGBG} and the fact that any cubical embedding between products of two graphs splits factorwise, up to interchanging factors, and its factor maps are embeddings; see \cite[Lemma~2.1]{OhProducts}.
\end{proof}

\begin{definition}[(Maximally) standard subgraph]\label{Def:OrthogonalComplement}
Let \(\graf_1 \subseteq \graf\) be a leafless (possibly trivial or disconnected) subgraph. The \emph{orthogonal complement} of \(\graf_1\) in \(\graf\), denoted by \(\graf_1^{\perp}\), is the (possibly empty or disconnected) subgraph obtained from \(\graf \setminus \graf_1\) by iteratively deleting leaves and isolated vertices until stabilization.

A leafless subgraph \(\graf_1\) is called \emph{standard} if it is connected and nontrivial, and its orthogonal complement \(\graf_1^\perp\) is nontrivial. It is called \emph{maximally standard} if, in addition, \(\graf_1^\perp\) is standard and \(\graf_1=(\graf_1^\perp)^\perp\).

We denote by \(\cS(\graf)\) and \(\cM(\graf)\) the sets of all standard and maximally standard subgraphs of \(\graf\), respectively.
We endow \(\cS(\graf)\) and \(\cM(\graf)\) with graph structures by declaring an edge between a pair of standard (maximally standard, resp.) subgraphs \(\graf_1,\graf_2\subsetneq\graf\) whenever \(\graf_1\subsetneq \graf_2\cup\graf_2^\perp\) or \(\graf_2\subsetneq \graf_1\cup\graf_1^\perp\).
\end{definition}

\begin{lemma}\label{lemma:properties of standard subgraphs}
Let \(\graf_1\) be a standard subgraph of \(\graf\). Then the following hold.
\begin{enumerate}
\item Any cycle contained in \(\graf_1\), as well as any component of \(\graf^\perp_1\), is standard.
\item There is a natural bijection between the subset of \(E(\cS(\graf))\) consisting of edges \(\{\grafl_1,\grafl_2\}\) with \(\grafl_2\subseteq\grafl_1^\perp\) and the set of standard product subcomplexes of \(UD_2(\graf)\). Moreover, there is a natural \(2\)-to-\(1\) map \[ V(\cM(\graf))\to V(\cI(UD_2(\graf)))\quad\text{ given by }\quad\graf_1\mapsto \graf_1\itimes\graf_1^\perp.\]
\item\label{Item} There exists a maximally standard subgraph containing \(\graf_1\). Moreover, for any standard subgraph \(\graf_2\) disjoint from \(\graf_1\), there is a unique maximally standard subgraph \(M(\graf_1;\graf_2)\) (\(N(\graf_1;\graf_2)\), resp.) containing \(\graf_1\) and disjoint from \(\graf_2\) that is maximal (minimal, resp.) with respect to inclusion.
\end{enumerate}
\end{lemma}

\begin{proof}
Items~(1) and~(2) follow from \Cref{corollary:maximal,Def:OrthogonalComplement} together with the fact that any subgraph with no leaves or isolated vertices survives unchanged under the iterative deletion of leaves and isolated vertices.
With the additional fact that every standard product subcomplex is contained in a maximal product subcomplex, the first assertion of Item~(3) follows.

Now, let \(\graf_2\) be a standard subgraph disjoint from \(\graf_1\).
Let \(M(\graf_1;\graf_2)\) denote the component of \(\graf^\perp_2\) containing \(\graf_1\); by \Cref{corollary:maximal}, \(M(\graf_1;\graf_2)\) is maximally standard.
If \(\graf'_1\) is a maximally standard subgraph containing \(\graf_1\) and disjoint from \(\graf_2\), then \(\graf_1\subseteq\graf'_1=(\graf'^\perp_1)^\perp\subseteq\graf^\perp_2\). Since \(\graf'_1\) is connected, it must be contained in \(M(\graf_1;\graf_2)\), proving the uniqueness of \(M(\graf_1;\graf_2)\).
The minimal case \(N(\graf_1;\graf_2)\) is analogous; in fact, \(N(\graf_1;\graf_2)=M(\graf_2;\graf_1)^\perp\).
\end{proof}

Although the curve graph is classically defined for surfaces, one can define an analogous object for a graph \(\graf\). We define the \emph{curve graph} \(\cC(\graf)\) to be the induced subgraph of \(\cS(\graf)\) spanned by cycles; thus two vertices are adjacent if and only if the corresponding cycles are disjoint in \(\graf\).

\begin{theorem}\label{theorem:connected basis}
The following are equivalent:
\[\text{\(\cC(\graf)\) is connected}\ \Longleftrightarrow\ \text{\(\cS(\graf)\) is connected}\ \Longleftrightarrow\ \text{\(\cM(\graf)\) is connected}\ \Longleftrightarrow\ \text{\(\cI(UD_2(\graf))\) is connected}.\]
\end{theorem}
\begin{proof}
\noindent (1)\(\Rightarrow\)(2)\quad This follows directly from \Cref{Def:OrthogonalComplement,lemma:properties of standard subgraphs}.

\smallskip

\noindent (2)\(\Rightarrow\)(3)\quad
Let \(\grafl\) and \(\grafl'\) be maximally standard subgraphs.
By assumption, there exists a path in \(\cS(\graf)\) 
\[\sfP=(\grafl=\grafl_1,\,\grafl_2,\dots,\,\grafl_N=\grafl').\]
Without loss of generality, we may assume that every consecutive pair \(\grafl_i\) and \(\grafl_{i+1}\) is disjoint. 
Indeed, if \(\grafl_i\subsetneq\grafl_{i+1}\), choose a component of \(\grafl_{i+1}^\perp\); if \(\grafl_{i+1}\subsetneq\grafl_i\), choose a component of \(\grafl_i^\perp\).
In either case, this component is standard and disjoint from both \(\grafl_i\) and \(\grafl_{i+1}\), so we may insert it between them. Repeating this procedure yields a path with the same endpoints in which consecutive vertices correspond to disjoint subgraphs.

For each such disjoint pair \(\grafl_i\) and \(\grafl_{i+1}\), there is a subgraph of \(\cS(\graf)\) of the form
\[
\begin{tikzcd}
N(\grafl_i;\grafl_{i+1})\ar[r,-]& N(\grafl_{i+1};\grafl_i)\\
\grafl_i\ar[-,u]\ar[r,-] & \grafl_{i+1}\ar[u,-]
\end{tikzcd}
\]
where \(N(\grafl_{i};\grafl_{i+1})\) and \(N(\grafl_{i+1};\grafl_{i})\) are maximally standard.
Note that \(N(\grafl_{1};\grafl_{2})=\grafl_1\) and \(N(\grafl_{N};\grafl_{N-1})=\grafl_N\).

Choose \(\sfP\) with the fewest edges among paths whose consecutive vertices correspond to disjoint subgraphs. In particular, \(\grafl_i\) and \(\grafl_{i+2}\) are not disjoint.
Then \(\grafl_i\cup \grafl_{i+2}\) is a standard subgraph disjoint from \(\grafl_{i+1}\). It follows that the three subgraphs
\[N(\grafl_{i+1};\grafl_{i}),\quad N(\grafl_{i+1};\grafl_{i+2}),\quad N(\grafl_{i+1};\grafl_{i}\cup\grafl_{i+2})\] are maximally standard, and that the first two both contain the third.
For each internal vertex \(\grafl_i\), this gives the zigzag
\[N(\grafl_i;\grafl_{i-1})\supseteq N(\grafl_i;\grafl_{i-1}\cup\grafl_{i+1})\subseteq N(\grafl_i;\grafl_{i+1}).\]
Together with the edges joining \(N(\grafl_i;\grafl_{i+1})\) and \(N(\grafl_{i+1};\grafl_i)\), these zigzags give a path in \(\cM(\graf)\) from \(\grafl\) to \(\grafl'\), after deleting any repetitions.
\[
\begin{tikzcd}[column sep=small, row sep=small]
\cdots  N(\grafl_{i};\grafl_{i-1}) \ar[from=r,hookrightarrow] &N(\grafl_i;\grafl_{i-1}\cup\grafl_{i+1})\ar[r,hookrightarrow]& N(\grafl_i;\grafl_{i+1})\ar[r,-] & N(\grafl_{i+1};\grafl_i) \ar[from=r,hookrightarrow] & N(\grafl_{i+1};\grafl_{i}\cup\grafl_{i+2})\ar[r,hookrightarrow]  &\cdots
\end{tikzcd}
\]

\smallskip

\noindent (3)\(\Rightarrow\)(4)\quad 
Let \(M=\graf_1\itimes\graf_1^\perp\) and \(M'=\graf_1'\itimes(\graf_1')^\perp\) be two maximal product subcomplexes of \(UD_2(\graf)\).
Since \(\cM(\graf)\) is connected by assumption, there is a path \((\graf_1,\,\graf_2,\,\dots,\,\graf_N=\graf_1')\) in \(\cM(\graf)\).
By the definition of an edge in \(\cM(\graf)\), for each \(i=1,\dots,N-1\) the intersection \(\graf_i\itimes\graf_i^\perp \cap \graf_{i+1}\itimes\graf_{i+1}^\perp\) contains a standard product subcomplex. Hence, \[(M=\graf_1\itimes\graf_1^\perp,\,\graf_2\itimes\graf_2^\perp,\dots,\, \graf_N\itimes\graf_N^\perp=M')\] is a path in \(\cI(UD_2(\graf))\), proving that \(\cI(UD_2(\graf))\) is connected.

\smallskip

\noindent (4)\(\Rightarrow\)(1)\quad 
To show that \(\cC(\graf)\) is connected, it suffices to consider two distinct cycles \(\sfC_1,\sfC_1'\in\cC(\graf)\) that are not disjoint; if they were disjoint, then by definition there would be an edge joining them.

By \Cref{lemma:properties of standard subgraphs}(1)\&(3), there exist maximally standard subgraphs \(\graf_1\) and \(\graf_1'\) containing \(\sfC_1\) and \(\sfC_1'\), respectively.
If \(\graf_1=\graf_1'\), then the sequence \((\sfC_1,\sfC_2,\sfC_1')\) is a path connecting \(\sfC_1,\sfC_1'\), where \(\sfC_2\) is a cycle in \(\graf^\perp_1\).  
Otherwise, let \(M=\graf_1\itimes\graf_1^\perp\) and \(M'=\graf_1'\itimes(\graf_1')^\perp\) be two maximal product subcomplexes (in particular, \(\graf'_1\nsubseteq\graf_1^\perp\) and \(\graf_1\nsubseteq(\graf'_1)^\perp\)).
By assumption, there is a path \[(M=M_1,\,\dots,\, M_N=M')\] 
in \(\cI(UD_2(\graf))\), where \(M_i=\graf_i\itimes\graf_i^\perp\) for some \(\graf_i\in\cM(\graf)\).
For each \(i=1,\dots,N\), let \(A_i\) denote the set of cycles contained in \(\graf_i\cup\graf^\perp_i\).
Each \(A_i\) induces a connected subgraph of \(\cC(\graf)\). Moreover, for each \(i=1,\dots,N-1\), the intersection \(A_i\cap A_{i+1}\) contains cycles in \(\grafl_i\cup\grafl'_i\), where \(\grafl_i\itimes\grafl'_i\) is a standard product subcomplex contained in \(M_i\cap M_{i+1}\).
Thus consecutive induced subgraphs intersect, and it follows that \(\cC(\graf)\) is connected.
\end{proof}

\begin{example}
Let \(\graf'\) and \(\graf''\) be the graphs depicted in \Cref{figure:nonexample of sip} and \Cref{figure:not all boundary cycles}, respectively. The curve graph \(\cC(\graf')\) is not connected but \(\cC(\graf'')\) is connected.
\end{example}

\begin{figure}[ht]
\centering
\begin{subfigure}{.33\textwidth}
\[
\begin{tikzpicture}[baseline=-.5ex, scale=0.9,transform shape]
\draw[thick] (0,0) node (A) {} node[below] {\(\sfv_0\)} -- node[midway,left] {\(\sfe_1\)} (90:1) node (B) {} node[right] {\(\sfv_1\)} (0,0) -- node[midway,below] {\(\sfe_2\)} (210:1) node (C) {} node[above] {\(\sfv_2\)} (0,0) -- node[midway,below] {\(\sfe_3\)} (330:1) node (D) {} node[above] {\(\sfv_3\)};
\draw[fill] (0,0) circle (2pt);
\draw (30:1.2) node {\(\sfe'_0\)};
\draw (90:2.1) node {\(\sfe'_1\)};
\draw (210:2.2) node {\(\sfe'_2\)};
\draw (330:2.2) node {\(\sfe'_3\)};
\grape[30]{A};
\grape[90]{B};
\grape[210]{C};
\grape[330]{D};
\end{tikzpicture}
\]
\caption{\(\cC(\graf)=\{\text{four \(3\)-cycles}\}\), connected}
\label{figure:example of sip}
\end{subfigure}
\begin{subfigure}{.33\textwidth}
\[
\begin{tikzpicture}[baseline=-.5ex, scale=0.9,transform shape]
\draw[fill] (-1.5,-1.5) circle (2pt) (-1.5,1.5) circle (2pt) (1.5,1.5) circle (2pt) (1.5,-1.5) circle (2pt)
(-.5,-.5) circle (2pt) (-.5,.5) circle (2pt) (.5,.5) circle (2pt) (.5,-.5) circle (2pt);
\draw[thick] (-1.5,-1.5) rectangle (1.5,1.5)
(-.5,-.5) rectangle node {\(\sfC_4^0\)} (.5,.5)
(-.5,-.5) -- (-1.5,-1.5) (-.5,.5) -- (-1.5,1.5) (.5,.5) -- (1.5,1.5) (.5,-.5) -- (1.5,-1.5);
\draw (0,1) node {\(\sfC_4^1\)}
(1,0) node {\(\sfC_4^2\)}
(0,-1) node {\(\sfC_4^3\)}
(-1,0) node {\(\sfC_4^4\)}
(2,0) node {\(\sfC_4^5\)}
;
\end{tikzpicture}
\]
\caption{\(\cC(\graf')=\{\sfC_4^0,\dots,\sfC_4^5\}\), disconnected}
\label{figure:nonexample of sip}
\end{subfigure}
\begin{subfigure}{.33\textwidth}
\vspace{10pt}
\[
\begin{tikzpicture}[baseline=-.5ex, scale=0.9,transform shape]
\draw[thick] (-1.5,-0.5) rectangle node {\(\sfC_4^1\)} (-0.5,0.5) -- (0.5,0.5)
(-0.5,-0.5) -- (0.5,-0.5) rectangle node {\(\sfC_4^3\)} (1.5,0.5)
;
\draw (0,0) node {\(\sfC_4^2\)};
\draw[fill] 
(-1.5,-0.5) circle (2pt)
(-.5,-0.5) circle (2pt)
(.5,-0.5) circle (2pt)
(1.5,-0.5) circle (2pt)
(-1.5,0.5) circle (2pt)
(-.5,0.5) circle (2pt)
(.5,0.5) circle (2pt)
(1.5,0.5) circle (2pt)
;
\end{tikzpicture}
\]
\vspace{30pt}
\caption{\(\cC(\graf'')=\{\sfC_4^1, \sfC_4^3\}\), connected}
\label{figure:not all boundary cycles}
\end{subfigure}
\caption{(Non-)examples of graphs with connected curve graphs}
\label{figure:sip}
\end{figure}

\subsection{\texorpdfstring{\(UP_2\)}{UP2}-hierarchy}\label{Section:UP_2_hierarchy}
For \(\graf \in \graphs\), we denote by \(UP_2(\graf)\) the subcomplex given by the union of all maximal product subcomplexes of \(UD_2(\graf)\). Equivalently, \(UP_2(\graf)\coloneqq\m{UD_2(\graf)}\) in the sense of \Cref{definition:union of maximals}. 
The underlying topological space of \(UP_2(\graf)\) depends only on the homeomorphism type of \(\graf\), as follows.

\begin{proposition}\label{theorem:subdivision and UP2}
Let \(\graf\) be a finite, possibly non-simple graph, and let \(\graf'\) be obtained from \(\graf\) by subdivision. Then \(UP_2(\graf')\) is a cubical subdivision of \(UP_2(\graf)\).
If \(UP_2(\graf)\neq\varnothing\), then
\[\cI(\bar{UP_2(\graf)})\isom \cI(\bar{UP_2(\graf')}),\]
and the following diagram commutes:
\[
\begin{tikzcd}
\pi_1(UP_2(\graf))\ar[r, "\m{\iota}_*"]\ar[d, "(\eta|)_*"', "\isom"] & \pi_1(UD_2(\graf))\ar[d, "\eta_*"]\\
\pi_1(UP_2(\graf'))\ar[r, "\m{\iota}_*"]& \pi_1(UD_2(\graf'))
\end{tikzcd}
\]
where \(\eta:UD_2(\graf)\hookrightarrow UD_2(\graf')\) is the canonical topological inclusion.
\end{proposition}
\begin{proof}
Subdivision gives a bijection between connected nontrivial leafless subgraphs of \(\graf\) and those of \(\graf'\), preserving disjointness and inclusion.
Indeed, any such subgraph containing a newly inserted vertex must contain both edges incident to it. Thus, by \Cref{corollary:maximal}, the restriction of \(\eta\) identifies corresponding maximal product subcomplexes by subdivision, proving the first assertion.
The resulting homeomorphism induces a quasi-isometry of universal covers, so \Cref{theorem:IsobetInt} gives the asserted isomorphism. The diagram commutes by functoriality.
\end{proof}


\begin{lemma}\label{Lem:UP2empty}
For \(\graf\in\graphs\), the following are equivalent:
\[UP_2(\graf)=\varnothing\ \Longleftrightarrow\ \cI(UP_2(\graf))=\varnothing\ \Longleftrightarrow\ \text{\(\graf\) has no two disjoint cycles}\ \Longleftrightarrow\ \text{\(\bbB_2(\graf)\) is hyperbolic}.\]
\end{lemma}
\begin{proof}
The equivalences (1)\(\Leftrightarrow\)(2) and (3)\(\Leftrightarrow\)(4) follow from \Cref{lemma:Y max facts} and from \Cref{corollary:hyperbolicity} together with \cite[Theorem~1.1]{Gen21GBG}, respectively.
By \Cref{Lem:SPSinGBG,corollary:maximal}, \(UD_2(\graf)\) contains a maximal product subcomplex if and only if \(\graf\) contains two disjoint nontrivial leafless subgraphs, which is equivalent to the existence of two disjoint cycles. Hence, (2)\(\Leftrightarrow\)(3).
\end{proof}


If \(UP_2(\graf)\neq\varnothing\), then the following fact shows that \(UD_2(\graf)\) belongs to the subclass \(\cY_{(0)}\), defined in \Cref{definition:Ymax hierarchy}.

\begin{lemma}\label{Lem:UP_2Special}
For \(\graf\in\graphs\) with \(UP_2(\graf)\neq\varnothing\), the subcomplex \(UP_2(\graf)\) is a connected special square complex.
\end{lemma}
\begin{proof}
By \Cref{thm:GBGSpecial}, \(UD_2(\graf)\) is a special square complex. Every hyperplane of \(UD_2(\graf)\) inherits a well-defined label in \(E(\graf)\), since opposite edges of a square have the same label; see the proof of \Cref{lemma:local isometric embedding into UD}.
If two hyperplanes with labels \(\sfe,\sff\) cross, their crossing square is represented by the unordered pair \(\{\sfe,\sff\}\). There is at most one such square, so any two distinct hyperplanes of \(UD_2(\graf)\) cross in at most one square. It follows from \Cref{lem:SpecialSubcomplexes} that \(UP_2(\graf)\) is special. Thus, it remains to prove that \(UP_2(\graf)\) is connected.

Let \(M=\graf_1\itimes\graf_1^\perp\) and \(M'=\graf_1'\itimes(\graf_1')^\perp\) be maximal product subcomplexes of \(UD_2(\graf)\).
If both \(\graf_1\cap\graf_1'\) and \(\graf_1^\perp\cap{\graf_1'}^\perp\) are nonempty, then \(M \cap M'\) contains \((\graf_1\cap\graf_1')\itimes(\graf_1^\perp\cap{\graf_1'}^\perp)\), which is nonempty.
Otherwise, we may assume that \(\graf_1\cap\graf_1'=\varnothing\).
Since \(\graf_1\) and \(\graf_1'\) are leafless subgraphs of \(\graf\setminus\graf_1'\) and \(\graf\setminus\graf_1\), respectively, we have \(\graf_1\subseteq{\graf_1'}^\perp\) and \(\graf_1'\subseteq\graf_1^\perp\), and thus,
\[M\cap M'=(\graf_1\cap {\graf_1'}^\perp)\itimes (\graf_1^\perp\cap\graf_1') = \graf_1\itimes\graf_1'\neq\varnothing.\qedhere\]
\end{proof}

Motivated by \Cref{definition:Ymax hierarchy}, we introduce an analogous hierarchy on \(\graphs\), organized according to how well \(UP_2(\graf)\) embeds into \(UD_2(\graf)\).
Since the effect of removing leaves on graph \(2\)-braid groups is described by \Cref{corollary:free factor}, we restrict attention to leafless graphs here.

\begin{definition}[\(UP_2\)-hierarchy]\label{definition:up2hierarchy}
Define the subclass \(\graphs_{(0)}\subsetneq\graphs\) by
\[\graphs_{(0)} \coloneqq \{\graf \in \graphs : \graf \text{ is leafless and contains at least two disjoint cycles}\},\]
and the following further subclasses of \(\graphs_{(0)}\) for \(1\le i\le 5\): \(\graf\in\graphs_{(i)}\Leftrightarrow UD_2(\graf)\in\cY_{(i)}\), where the classes \(\cY_{(i)}\) are given in \Cref{definition:Ymax hierarchy}. 
These subclasses satisfy the inclusions 
\[\graphs_{(5)}\subseteq\graphs_{(4)}\subseteq\graphs_{(2)}\cap\graphs_{(3)}\quad\text{and}\quad \graphs_{(2)}\cup\graphs_{(3)}\subseteq\graphs_{(1)}.\]
\end{definition}

Let \(\grafl_1\) and \(\grafl_2\) be two disjoint subgraphs of a graph \(\graf\).
We say that \(\grafl_1\) and \(\grafl_2\) are \emph{separable} if there are two disjoint leafless subgraphs \(\graf_1\) and \(\graf_2\) such that \(\grafl_1\subseteq\graf_1\) and \(\grafl_2\subseteq\graf_2\).
For instance, a vertex \(\sfv\) and an edge \(\sfe\) are separable if and only if \(\{\sfv,\sfe\}\) lies in \(UP_2(\graf)\).

We now introduce three combinatorial conditions for \(\graf\in\graphs_{(0)}\), which reinterpret several subclasses in \Cref{definition:up2hierarchy} in purely combinatorial terms:
\begin{itemize}
\item [(A)]\label{Item:(A)} No triangle contains a bivalent vertex.
\item [(B)] Every pair of disjoint edges is separable.
\item [(C)] If two disjoint edges \(\sfe_1=[\sfv_1,\sfw_1]\), \(\sfe_2=[\sfv_2,\sfw_2]\) are such that \((\sfv_1,\sfe_2)\) and \((\sfe_1,\sfv_2)\) are separable, then \(\sfe_1\) and \(\sfe_2\) are separable.
\end{itemize}

We first relate condition~(A) to the cubical structure of \(UD_2(\graf)\). Suppose that \(\graf\) does not satisfy condition~(A) so that there exists a \(3\)-cycle in \(\graf\) with vertices \(\sfv_1,\sfv_2,\sfv_3\) and \(\val_{\graf}(\sfv_1)=2\).
Then the edge \(e=\sfv_1\itimes [\sfv_2,\sfv_3]\) in \(UD_2(\graf)\) is not a face of any square.
Conversely, such an edge exists only if \(\graf\) fails to satisfy condition~(A). In particular, \(\graf\) satisfies condition~(A) if and only if \(UD_2(\graf)\) is the union of all squares.

We next compare the three conditions. It is immediate that (B) implies (C). Moreover, (B) also implies (A).
Indeed, if (A) fails, then there is a \(3\)-cycle with vertices \(\sfv_1,\sfv_2,\sfv_3\) such that \(\val_\graf(\sfv_1)=2\).
Since \(\graf\) is leafless and connected and has at least two disjoint cycles, at least one of \(\sfv_2\) and \(\sfv_3\), say \(\sfv_2\), has valency at least \(3\).
Hence there is an edge \(\sfe\) joining \(\sfv_2\) and a vertex distinct from \(\sfv_1\) and \(\sfv_3\).
The two disjoint edges \(\sfe\) and \([\sfv_1,\sfv_3]\) are then not separable, contradicting (B).

\begin{lemma}\label{Lem:CombinatorialConditions}
For a graph \(\graf\in\graphs_{(0)}\), we have the following equivalences: 
\begin{enumerate}
\item \(\graf\in\graphs_{(3)}\) if and only if condition~(C) holds for \(\graf\).
\item \(\graf\in\graphs_{(4)}\) if and only if \(\graf\in\graphs_{(5)}\) if and only if condition~(B) holds for \(\graf\).
\end{enumerate}
\end{lemma}

\begin{proof}
\noindent{(1)}\quad This is immediate from the definitions of a local isometry and separability.

\smallskip

\noindent{(2)}\quad Suppose that condition~(B) holds. Then (A) holds as observed above, so \(UD_2(\graf)\) is the union of its squares. Moreover, (B) ensures that every square is contained in some standard product subcomplex. Thus, \(UD_2(\graf)=UP_2(\graf)\), proving \(\graf\in\graphs_{(4)}\cap\graphs_{(5)}\).

Conversely, if (B) fails, then there exists a square in \(UD_2(\graf)\) which is not contained in any standard product subcomplex. Hence, \(UP_2(\graf)\) omits at least one square, so \(\graf\notin \graphs_{(4)}\cup\graphs_{(5)}\).
\end{proof}

\begin{remark}
We do not know the exact classification of graphs satisfying condition~(B).
However, there are infinitely many graphs that are \(1\)-skeletons of polytopes and satisfy condition~(B).
\end{remark}

The following examples show that, apart from the equality \(\graphs_{(4)}=\graphs_{(5)}\), the classes in the \(UP_2\)-hierarchy are distinct. They distinguish local convexity, \(\pi_1\)-injectivity, and the existence of the prescribed free product splitting.


\begin{example}[\(\graphs_{(4)}=\graphs_{(5)}\), dodecahedral graph]\label{example:dodecahedral graph 2}
Let \(\graf\) be the \(1\)-skeleton of the dodecahedron.
Then any two disjoint edges are separable, and thus \(UP_2(\graf)=UD_2(\graf)\).
\end{example}

\begin{example}[\(\graphs_{(2)}\cap \graphs_{(3)}\setminus \graphs_{(5)}\)]\label{example:bog2}
Consider the graph \(\graf\) shown in \Cref{figure:example of sip}. 
Since condition~(C) holds for \(\graf\), we have \(\graf\in\graphs_{(3)}\). 
However, \(\graf\) contains a \(3\)-cycle with bivalent vertices, so condition~(A), and hence condition~(B), fails. 
Indeed, \(UP_2(\graf)\) omits the squares \(\sfe_i\itimes\sfe_0'\) and \(\sfe_i\itimes\sfe_i'\) for \(1\le i\le 3\). Thus \(\graf\notin\graphs_{(5)}\).

Let \(\sfe_0'=[\sfw_0,\sfw_0']\). 
Each square \(\sfe_i\itimes\sfe_i'\) (\(\sfe_i\itimes\sfe_0'\), resp.) contains the edge \(\sfv_i\itimes\sfe_i'\) (\(\sfe_i\itimes\sfw_0\), resp.), and these edges are not contained in any other squares of \(UD_2(\graf)\). 
Collapsing these edges together with the unique squares containing them yields a cube complex \(UD_2'(\graf)\), which is a deformation retract of \(UD_2(\graf)\). 
Since \(UP_2(\graf)\) contains all squares of \(UD_2'(\graf)\), it follows that
\(\bbB_2(\graf)\cong \pi_1(UD_2'(\graf))\cong \pi_1(UP_2(\graf))\ast \bbF_N\),
and hence \(\graf\in\graphs_{(2)}\).
\end{example}

\begin{figure}[ht]
\centering
\begin{subfigure}{0.45\textwidth}
\[
\begin{tikzpicture}[baseline=-.5ex,scale=0.8,rotate=-90,transform shape,every node/.style={rotate=90}]
\draw[thick] (-30:2) node (A) {} node[right] {\(\sfv_3\)} -- node[midway,left] {\(\sfe_3\)} (-30:1) -- node[midway,below] {\(\sfe_2''\)} (90:1) -- node[midway,above] {\(\sfe_1\)} (90:2) node (B) {} node[above] {\(\sfv_1\)} (90:1) -- (210:1) node[midway,above] {\(\sfe_3''\)} -- node[midway,left] {\(\sfe_2\)} (210:2) node (C) {} node[right] {\(\sfv_2\)} (210:1) -- (-30:1) node[midway,left] {\(\sfe_1''\)};
\draw[fill] (-30:1) circle (2pt) node[left] {\(\sfw_3\)} (90:1) circle (2pt) node[above] {\(\sfw_1\)} (210:1) circle (2pt) node[left] {\(\sfw_2\)};
\draw (-30:3.2) node {\(\sfe_3'\)} (90:3.2) node {\(\sfe_1'\)} (210:3.2) node {\(\sfe_2'\)};
\grape[-30]{A};
\grape[90]{B};
\grape[210]{C};
\end{tikzpicture}
\]
\caption{Non-locally convex \(UP_2(\graf)\)}
\label{figure:non quasi-isometry}
\end{subfigure}
\begin{subfigure}{0.45\textwidth}
\[
\begin{tikzpicture}[baseline=-.5ex,scale=0.9,transform shape]
\useasboundingbox (-3,-3) rectangle (1,2);
\foreach \i in {1,2,3,4,5} {
\draw[fill] ({\i*72+180}:1) circle (2pt);
\draw[thick] ({\i*72+180}:1) -- ({(\i+1)*72+180}:1) -- ({(\i+3)*72+180}:1);
}
\draw[thick] (-2,0) node[above] {\(\sfv\)} -- node[midway,above] {\(\sfe\)} (-1,0) node[above] {\(\sfw\)};
\draw[fill] (-2,0) circle (2pt) node (A) {};
\grape[180]{A};
\end{tikzpicture}
\]
\caption{Non-free factor \(\pi_1(UP_2(\graf))\)}
\label{figure:non free factor}
\end{subfigure}
\begin{subfigure}{0.4\textwidth}
\[
\begin{tikzpicture}[baseline=-.5ex,scale=0.8,rotate=-90,transform shape,every node/.style={rotate=90}]
\draw[thick] (-30:2) node (A) {} node[right] {\(\sfv_3\)} -- node[midway,left] {\(\sfe_3\)} (-30:1) -- node[midway,below] {\(\sfe_2''\)} (90:1) -- node[midway,above] {\(\sfe_1\)} (90:2) node (B) {} node[above] {\(\sfv_1\)} (90:1) -- (210:1) node[midway,above] {\(\sfe_3''\)} -- node[midway,left] {\(\sfe_2\)} (210:2) node (C) {} node[right] {\(\sfv_2\)} (210:1) -- (-30:1) node[midway,left] {\(\sfe_1''\)};
\draw[fill] (-30:1) circle (2pt) node[left] {\(\sfw_3\)} (90:1) circle (2pt) node[above] {\(\sfw_1\)} (210:1) circle (2pt) node[left] {\(\sfw_2\)};
\draw (-30:3.2) node {\(\sfe_3'\)} (210:3.2) node {\(\sfe_2'\)};
\grape[-30]{A};
\grape[210]{C};
\begin{scope}[yshift=3cm]
\foreach \i in {1,2,3,4,5} {
\draw[fill] ({\i*72-90}:1) circle (2pt);
\draw[thick] ({\i*72-90}:1) -- ({(\i+1)*72-90}:1) -- ({(\i+3)*72-90}:1);
}
\end{scope}
\end{tikzpicture}
\]
\caption{Non-locally convex \(UP_2(\graf)\) whose \(\pi_1\) is not a free factor of \(\bbB_2(\graf)\).}
\label{figure:only (1)}
\end{subfigure}
\begin{subfigure}{0.4\textwidth}
\[
\begin{tikzpicture}[baseline=-.5ex,scale=0.9,transform shape]
\useasboundingbox (-2,-3) rectangle (2,2);
\draw[thick] (-1.7,0) -- (-0.5,0) node[midway,below] {\(\sfe_2\)} (-0.5,1.5) arc (90:270:1.2 and 1.5) node[pos=0.25,above left] {\(\sfe_1\)} node[pos=0.75,below left] {\(\sfe_3\)};
\draw[thick] (1.7,0) -- (0.5,0) node[midway,below] {\(\sfe_2'\)} (0.5,1.5) arc (90:-90:1.2 and 1.5) node[pos=0.25,above right] {\(\sfe_1'\)} node[pos=0.75,below right] {\(\sfe_3'\)};
\draw[thick,fill] 
(-0.5,1.5) circle (2pt) node[above left] {\(\sfv_1\)}
(-0.5,0) circle (2pt) node[above left] {\(\sfv_2\)}
(-0.5,-1.5) circle (2pt) node[below left] {\(\sfv_3\)}
(-1.7,0) circle (2pt) node[left] {\(\sfv\)}
(0.5,1.5) circle (2pt) node[above right] {\(\sfw_1\)}
(0.5,0) circle (2pt) node[above right] {\(\sfw_2\)}
(0.5,-1.5) circle (2pt) node[below right] {\(\sfw_3\)}
(1.7,0) circle (2pt) node[right] {\(\sfw\)}
(0,2) circle (2pt) node[above] {\(\sfz_1\)}
(0,0.5) circle (2pt) node[above] {\(\sfz_2\)}
(0,-1) circle (2pt) node[above] {\(\sfz_3\)}
;
\draw[thick] (0,1.5) circle (0.5 and 0.5)
(0,0) circle (0.5 and 0.5)
(0,-1.5) circle (0.5 and 0.5);
\draw (0,-1.5) node {\(\sfC_3^3\)}
(0,0) node {\(\sfC_3^2\)}
(0,1.5) node {\(\sfC_3^1\)};
\end{tikzpicture}
\]
\caption{Non-\(\pi_1\)-injective}
\label{figure:only (0)}
\end{subfigure}
\caption{Counterexamples}
\label{figure:counterexamples}
\end{figure}

\begin{example}[\(\graphs_{(2)}\setminus \graphs_{(3)}\)]\label{example:not quasiisometry}
Let \(\graf\) be obtained from a graph consisting of a single \(3\)-cycle and three leaves by attaching a \(3\)-cycle at each leaf (see \Cref{figure:non quasi-isometry}).
Then \(UD_2(\graf)\) has exactly three maximal product subcomplexes, each determined by an attached \(3\)-cycle and its orthogonal complement.
Consequently, \(UP_2(\graf)\) omits nine squares of \(UD_2(\graf)\), namely \(\sfe_i\itimes\sfe_i'\), \(\sfe_i\itimes\sfe_i''\), and \(\sfe_i\itimes\sfe_{i+1}\). 
The absence of the squares \(\sfe_i\itimes\sfe_{i+1}\) violates the local convexity of \(UP_2(\graf)\) since it contains the edges \(\sfv_i\itimes\sfe_{i+1}\) and \(\sfe_i\itimes\sfv_{i+1}\). Hence \(\graf\not\in\graphs_{(3)}\). Here \(1\le i\le3\), with \(\sfe_4=\sfe_1\), \(\sfv_4=\sfv_1\), and \(\sfw_4=\sfw_1\).

However, each square \(\sfe_i\itimes\sfe_i'\) (\(\sfe_i\itimes\sfe_i''\), resp.) contains an edge \(\sfv_i\itimes\sfe_i'\) (\(\sfw_i\itimes\sfe_i''\), resp.) not contained in any other square of \(UD_2(\graf)\). As before, performing elementary collapses on these pairs gives a subcomplex \(UD_2'(\graf)\).

In \(UD_2'(\graf)\), each square of the form \(\sfe_i\itimes\sfe_{i+1}\) has an edge \(\sfw_i\itimes \sfe_{i+1}\)\footnote{One may choose \(\sfe_i\itimes\sfw_{i+1}\) instead.} not contained in any other squares of \(UD_2'(\graf)\).
By collapsing these cells again, we obtain a subcomplex \(UD_2''(\graf)\) that is homotopy equivalent to \(UD_2(\graf)\) and whose squares all lie in \(UP_2(\graf)\).
Therefore, \(\bbB_2(\graf)\isom\pi_1(UD_2''(\graf))\isom\pi_1(UP_2(\graf))\ast\bbF_N\), and hence \(\graf\in\graphs_{(2)}\).
\end{example}

\begin{example}[\(\graphs_{(3)}\setminus \graphs_{(2)}\)]\label{example:(3) but not (2)}
Let \(\graf\) be obtained by connecting a \(3\)-cycle \(\sfC_3\) and the complete graph \(\sfK_5\) by an edge \(\sfe=[\sfv,\sfw]\) (see \Cref{figure:non free factor}).
Then \(UD_2(\graf)\) has a unique maximal product subcomplex, namely, \(UP_2(\graf)=\sfC_3\itimes\sfK_5\). It is then straightforward to verify that condition~(C) holds and thus \(\graf\in\graphs_{(3)}\).

Suppose, for contradiction, that \(\graf\in\graphs_{(2)}\). Then
\[\bbB_2(\graf)\isom \pi_1(UP_2(\graf))\ast\bbF_N\isom (\bbZ\times\bbF_6)\ast\bbF_N\ \text{ for some } N\ge0.\]
By \Cref{Lem:embedding}, \(\bbB_2(\sfK_5)\), which is a surface group as observed in the proof of \Cref{Non-planar}, embeds undistortedly in \(\bbB_2(\graf)\). 
If \((\bbZ\times\bbF_6)\ast\bbF_N\) contained such a surface subgroup, then the Kurosh subgroup theorem would imply that it lies in a conjugate of \(\bbZ\times\bbF_6\). However, this contradicts a theorem of Baumslag--Roseblade \cite{BS84}, which
states that any finitely presented subgroup of \(\bbZ\times\bbF_6\) is either free or virtually a direct product of free groups. Thus \(\graf\notin\graphs_{(2)}\).

Indeed, if there were a quasi-isometry between \(\bbB_2(\graf)\) and \((\bbZ\times\bbF_6)\ast\bbF_N\), then it would induce a quasi-isometric embedding of a surface group into \((\bbZ\times\bbF_6)\ast\bbF_N\). 
By \cite[Lemma~3.2]{PW02}, the image of such an embedding must lie within a bounded neighborhood of a conjugate of \(\bbZ\times\bbF_6\). 
However, \cite{BS08} shows that the hyperbolic plane does not quasi-isometrically embed into a product of a tree with \(\bbR\). 
Therefore, \(\bbB_2(\graf)\) is not even quasi-isometric to \(\pi_1(UP_2(\graf))\ast\bbF_N\).
\end{example}


\begin{example}[\(\graphs_{(1)}\setminus (\graphs_{(2)}\cup \graphs_{(3)})\)]\label{example:only (1)}
Let \(\graf\) be as in \Cref{figure:only (1)}. Then \(UP_2(\graf)\) is the union of three maximal product subcomplexes corresponding to pairs of subgraphs separated by \(\sfe_1,\sfe_2\), and \(\sfe_3\), and contains the edges \(\sfv_i\itimes\sfe_{i+1}\) and \(\sfe_i\itimes\sfv_{i+1}\) but not the squares \(\sfe_i\itimes\sfe_{i+1}\). Here \(1\le i\le3\), with \(\sfe_4=\sfe_1\) and \(\sfv_4=\sfv_1\).
Thus \(UP_2(\graf)\) is not locally convex in \(UD_2(\graf)\) and \(\graf\not\in\graphs_{(3)}\).

By the same elementary collapsing argument as in \Cref{example:not quasiisometry}, we obtain a deformation retract \(UD_2''(\graf)\subsetneq UD_2(\graf)\) such that \(\iota:UP_2(\graf)\to UD_2''(\graf)\) is a local isometry.
Hence \(\iota_*:\pi_1(UP_2(\graf))\to \pi_1(UD_2''(\graf))\isom\bbB_2(\graf)\) is an injection, and therefore \(\graf\in\graphs_{(1)}\).

Since every omitted square lies in either \(\sfe_1\itimes(\sfK_5\setminus\sfv_1)\) or \(UD_2(\sfK_5)\), we have 
\[UD_2''(\graf) = \{\text{some edges}\}\cup UP_2(\graf) \cup (\sfe_1\itimes (\sfK_5\setminus\sfv_1)) \cup UD_2(\sfK_5),\]
which implies that \(\bbB_2(\graf)\isom \bbF_M\ast\pi_1(UP_2(\graf))\ast_{\bbF_3} \bbB_2(\sfK_5)\) for some \(M\ge 0\).

Suppose now that \(\graf\in\graphs_{(2)}\). 
Then there would exist an isomorphism
\[f:\bbF_M\ast\pi_1(UP_2(\graf))\ast_{\bbF_3}\bbB_2(\sfK_5)\longrightarrow\pi_1(UP_2(\graf))\ast\bbF_N\]
for some \(N\ge0\), preserving the distinguished subgroup \(\pi_1(UP_2(\graf))\) (the image of \(\iota_*\) in both the domain and the codomain).
On the one hand, the image of \(\bbB_2(\sfK_5)\) in the quotient 
\[\bbB_2(\graf)\big/\langle\!\langle \pi_1(UP_2(\graf))\rangle\!\rangle\]
is nontrivial. Indeed, the displayed amalgam identifies this image with \(\bbB_2(\sfK_5)/\langle\!\langle\bbF_3\rangle\!\rangle\). 
Since \(UD_2(\sfK_5)\) is a closed surface of Euler characteristic \(-5\), its first Betti number is at least \(6\); killing three generators therefore leaves nontrivial rational first homology.
On the other hand, the Kurosh subgroup theorem implies that \(f(\bbB_2(\sfK_5))\) lies in a conjugate of \(\pi_1(UP_2(\graf))\), so its image in the corresponding quotient is trivial.
This contradiction shows that \(\graf\notin\graphs_{(2)}\).
\end{example}

\begin{example}[\(\graphs_{(0)}\setminus \graphs_{(1)}\)]\label{example:only (0)}
Let \(\graf\) be obtained from the complete bipartite graph \(\sfK_{2,3}\) by replacing bivalent vertices with \(3\)-cycles \(\sfC_3^i\) (see \Cref{figure:only (0)}).
Then \(UD_2(\graf)\) has exactly three maximal product subcomplexes \(\sfC_3^i\itimes (\sfC_3^i)^\perp\).
Moreover, \(UP_2(\graf)\) omits all cells of \(UD_2(\graf)\) incident to the vertex \(v=\{\sfv,\sfw\}\). More precisely, 
\[UD_2(\graf) = UP_2(\graf) \cup (\st_\graf(\sfv)\itimes \st_\graf(\sfw)),\]
where the intersection \(UP_2(\graf) \cap (\st_\graf(\sfv)\itimes \st_\graf(\sfw))\) consists of the \(12\) edges \(\sfv_i\itimes \sfe_j'\) and \(\sfe_i\itimes \sfw_j\) for \(1\le i\neq j\le 3\).
These edges form a \(12\)-cycle \(\sfC_{12}\), which is locally isometrically embedded in \(UP_2(\graf)\) but null-homotopic in \(\st_\graf(\sfv)\itimes\st_\graf(\sfw)\).
Consequently, the inclusion \(\sfC_{12}\hookrightarrow UP_2(\graf)\) induces an injective homomorphism \(\bbZ\isom\pi_1(\sfC_{12})\rightarrow \pi_1(UP_2(\graf))\), whose image becomes trivial in \(\bbB_2(\graf)\) under \(\m\iota_*:\pi_1(UP_2(\graf))\to\bbB_2(\graf)\). Therefore, \(\m\iota_*\) is not injective and hence \(\graf\not\in\graphs_{(1)}\).
\end{example}

In summary, \Cref{figure:Venn diagram} collects representative examples for each level of the \(UP_2\)-hierarchy.


\begin{figure}[ht]
\[
\begin{tikzpicture}[baseline=-.5ex, scale=0.9, transform shape]
\draw[rounded corners] (-2,3.5) -- (-5.5,3.5) -- (-5.5,-3.5) -- (8.5,-3.5) -- (8.5,3.5) -- (2,3.5);
\draw (7.5, 1) node {\(\sfK_5\)};
\draw (7.5, 0) node {\(\sfK_{3,3}\)};
\draw (7.5, -1) node {\(\grafl(0,\ell\ge2)\)};
\draw (0,3.5) node {\(\{\graf\in\graphs: \graf\text{ is leafless}\}\)};
\draw[rounded corners] (-1.05,3) -- (-5,3) -- (-5,-3) -- (6.5,-3) -- (6.5,3) -- (1.05,3);
\draw (0,3) node {\(\graphs_{(0)}\)};
\draw (0,-0.5) node {Ex.~\ref{example:dodecahedral graph 2}} ++(130:1.5 and 1) arc (130:410:1.5 and 1);
\draw (0,0.3) node {\(\graphs_{(4)}=\graphs_{(5)}\)};
\draw (-1.25,-0.25) ++(100:3 and 2) arc (100:440:3 and 2);
\draw (-1.25,1.75) node {\(\graphs_{(2)}\)};
\draw (1.25,-0.25) ++(100:3 and 2) arc (100:440:3 and 2);
\draw (1.25,1.75) node {\(\graphs_{(3)}\)};
\draw (0,1) node {Ex.~\ref{example:bog2}};
\draw (-3,0) node {Ex.~\ref{example:not quasiisometry}};
\draw (3,0) node {Ex.~\ref{example:(3) but not (2)}};
\draw[rounded corners] (-0.75,2.5) -- (-4.5,2.5) -- (-4.5,-2.5) -- (4.5,-2.5) -- (4.5,2.5) -- (0.75,2.5);
\draw (0,2.5) node {\(\graphs_{(1)}\)};
\draw (3,2) node {Ex.~\ref{example:only (1)}};
\draw (5.5,0) node {Ex.~\ref{example:only (0)}};
\end{tikzpicture}
\]
\caption{Examples of the \(UP_2\)-hierarchy}
\label{figure:Venn diagram}
\end{figure}

Despite the preceding examples, a complete combinatorial characterization of graphs in the subclasses \(\graphs_{(1)}\) and \(\graphs_{(2)}\) remains unclear. 
These observations lead naturally to the following question.

\begin{question}\label{Qn:LocalConvexity}
Which graphs belong to \(\graphs_{(1)}\) or \(\graphs_{(2)}\)? 
In particular, does there exist a planar graph in \(\graphs_{(1)}\setminus\graphs_{(2)}\)?
\end{question}

In the context of the \(UP_2\)-hierarchy, a natural next step is to address \Cref{Question:DefiningQII}. More precisely, one may ask whether the restriction of \(\m{(\cdot)}\) to the subclass \(UD_2(\graphs_{(0)})=\{UD_2(\graf): \graf\in\graphs_{(0)}\}\) defines a (complete) quasi-isometry invariant in the sense of \Cref{Definition:QIinvariant}. 

Clearly, the restriction to \(UD_2(\graphs_{(4)})(=UD_2(\graphs_{(5)}))\) is the identity map and therefore trivially defines a complete quasi-isometry invariant. We now establish quasi-isometry invariance for a further subclass.

\begin{lemma}\label{Lem:UP_2One-ended}
Let \(\graf\in\graphs\) with \(UP_2(\graf)\neq\varnothing\). Suppose that \(UD_2(\graf)\) (equivalently, \(UP_2(\graf)\)) satisfies the standard intersection property.
Then \(\cI(UP_2(\graf))\) is connected and \(\cI(\bar{UP_2(\graf)})\) is contractible. Moreover, \(\cI(UP_2(\graf))\) admits the structure of a developable complex of groups whose development is \(\cI(\bar{UP_2(\graf)})\). In particular, \(\pi_1(UP_2(\graf))\) is thick, and hence one-ended.
\end{lemma}
\begin{proof}
By \Cref{Lem:UP_2Special}, \(UP_2(\graf)\) is connected. It inherits the embedded product property from \(UD_2(\graf)\) by \Cref{Ex:EPP}.
Together with the standard intersection property, this allows us to apply \Cref{lem:Connected}, giving the assertions about the intersection complexes and the complex of groups. The final assertion follows from \Cref{lem:One-endedness,Rmk:thickness}.
\end{proof}

\begin{example}[Graphs with \(\diam(\cM(\graf))=1\)]\label{Ex:diamM_graf=1}
Let \(\graf\) be a graph with \(\diam(\cM(\graf))=1\). For two distinct maximal product subcomplexes \(M=\graf_1\itimes\graf_1^\perp\) and \(M'=\graf_1'\itimes(\graf_1')^\perp\) of \(UD_2(\graf)\), we may interchange factors so that either \(\graf_1\subsetneq\graf_1'\) or \(\graf_1\subsetneq(\graf_1')^\perp\).
Their intersection is then \(\graf_1\itimes(\graf_1')^\perp\) or \(\graf_1\itimes\graf_1'\), respectively, and is a standard product subcomplex.
Therefore, \(UD_2(\graf)\) satisfies the standard intersection property.

For example, let \(\graf\) be as described in \Cref{figure:example of sip}. Then \(\cC(\graf)\) is connected and \(\cM(\graf)\) has diameter \(1\).
This graph is an example of a \emph{bunch of grapes}, a class we study in the next section.
\end{example}

\begin{theorem}\label{Thm:SubclassdefiningQII}
Let \(\graphs'\) be the subclass of \(\graphs_{(0)}\) defined by
\[\graphs'=\{\graf\in\graphs_{(2)}\cup\graphs_{(3)}: UD_2(\graf)\text{ has the standard intersection property}\}.\]
Then the restriction of \(\m{(\cdot)}\) to \(UD_2(\graphs')\) defines a quasi-isometry invariant. 
Moreover, for \(\graf_1,\graf_2\in\graphs'\cap\graphs_{(2)}\), if \(\bbB_2(\graf_1)\) and \(\bbB_2(\graf_2)\) have the same number of ends, then
\[\bbB_2(\graf_1)\qisom\bbB_2(\graf_2)\quad\Longleftrightarrow\quad
\pi_1(UP_2(\graf_1))\qisom\pi_1(UP_2(\graf_2)).\]
\end{theorem}

\begin{proof}
For \(\graf\in\graphs'\), the inclusion \(UP_2(\graf)\hookrightarrow UD_2(\graf)\) is \(\pi_1\)-injective, so \Cref{lem:UnionofSPS} gives an embedding \(\bar{UP_2(\graf)}\hookrightarrow\bar{UD_2(\graf)}\).
This embedding is isometric when \(\graf\in\graphs_{(3)}\), by local convexity. When \(\graf\in\graphs_{(2)}\), it is a quasi-isometric embedding by compactness and the undistortion of free factors.

By the standard intersection property and \Cref{Lem:UP_2One-ended}, the intersection complex \(\cI(\bar{UP_2(\graf)})\) is connected. Together with \Cref{lem:UnionofSPS}, this shows that the intersection complex of each copy of \(\bar{UP_2(\graf)}\) is a connected component of \(\cI(\bar{UD_2(\graf)})\).

Let \(\graf,\graf'\in\graphs'\), and let \(\phi:\bar{UD_2(\graf)}\to\bar{UD_2(\graf')}\) be a quasi-isometry. By \Cref{theorem:IsobetInt}, the induced isomorphism \(\cI(\phi)\) maps these connected components bijectively. Since each copy is the union of its maximal product subcomplexes, the uniform Hausdorff bound in \Cref{MaxtoMax} implies that \(\phi\) sends each copy of \(\bar{UP_2(\graf)}\) to within finite Hausdorff distance of a copy of \(\bar{UP_2(\graf')}\). 
Both copies are quasi-isometrically embedded in their ambient universal covers, so a bounded adjustment of the restriction of \(\phi\) gives a quasi-isometry between them with their intrinsic metrics. This proves the first assertion.

For the moreover statement, write \(\bbB_2(\graf_i)\cong \pi_1(UP_2(\graf_i))\ast\bbF_{N_i}\) for \(i=1,2\), where \(N_i\ge0\). By \Cref{Lem:UP_2One-ended}, each \(\pi_1(UP_2(\graf_i))\) is one-ended. Thus the assumption on ends implies that either \(N_1=N_2=0\) or \(N_1,N_2\ge1\).
The forward implication follows from the first assertion. The converse is immediate in the former case and follows from \Cref{PW} in the latter.
\end{proof}

By \Cref{Ex:diamM_graf=1}, any \(\graf\in\graphs_{(2)}\cup\graphs_{(3)}\) with \(\diam(\cM(\graf))=1\) belongs to \(\graphs'\).

Without the standard intersection property, \Cref{Lem:UP_2One-ended} no longer guarantees that \(\pi_1(UP_2(\graf))\) is one-ended. The free product decompositions alone do not exclude the existence of graphs \(\graf,\graf'\in\graphs_{(2)}\) satisfying
\[\pi_1(UP_2(\graf))\cong \pi_1(UP_2(\graf'))\ast\bbF_r,\quad
\bbB_2(\graf)\cong \pi_1(UP_2(\graf))\ast\bbF_m,\quad
\bbB_2(\graf')\cong \pi_1(UP_2(\graf'))\ast\bbF_n,\]
for some \(r,m,n\ge1\), with \(\pi_1(UP_2(\graf'))\) one-ended. If such graphs existed, then \Cref{PW} would imply that \(\bbB_2(\graf)\) and \(\bbB_2(\graf')\) are quasi-isometric, whereas \(\pi_1(UP_2(\graf))\) and \(\pi_1(UP_2(\graf'))\) would not be.
Nevertheless, we expect that this phenomenon does not occur and propose the following conjecture.

\begin{conjecture}
The restriction of \(\m{(\cdot)}\) to \(UD_2(\graphs_{(2)})\) defines a complete quasi-isometry invariant.
\end{conjecture}



\section{Bunches of grapes}\label{section:grapes}

In this section, we introduce an infinite family of graphs, called \emph{bunches of grapes}, which includes both trees and the elementary bunches of grapes defined in \Cref{definition:FreeGroupCase}. 
Our motivation is to identify graphs belonging to the class \(\graphs'\) appearing in \Cref{Thm:SubclassdefiningQII}. 
More precisely, we seek graphs satisfying the following properties:
\begin{itemize}
\item they belong to the subclasses \(\graphs_{(2)}\) or \(\graphs_{(3)}\) from \Cref{definition:up2hierarchy};
\item their discrete \(2\)-configuration spaces satisfy the standard intersection property (see \Cref{Lem:UP_2One-ended}).
\end{itemize}
For the second requirement, \Cref{Ex:diamM_graf=1} provides a sufficient condition. 

A key structural feature of bunches of grapes is that all cycles are attached along a tree; this tree-like organization of cycles leads to refined structural properties for the associated graph \(2\)-braid groups.

\begin{definition}[Bunches of grapes]\label{Def:BunchesofGrapes}
Let \(\sfT\) be a (possibly trivial) finite tree, and let \(\loops:V(\sfT)\to \bbZ_{\ge 0}\) be a function. A \emph{bunch of grapes} \(\graf=(\sfT,\loops)\) is the graph obtained by attaching \(\loops(\sfv)\) copies of a \(3\)-cycle, called \emph{grapes}, to each vertex \(\sfv\in V(\sfT)\); the tree \(\sfT\) is called the \emph{stem} of \(\graf\). We denote by \(\grapegraph\) the set of all bunches of grapes.

A \emph{substem} \(\sfT'\subseteq\sfT\) is a (possibly trivial) subtree of \(\sfT\); the restriction \(\graf_{\sfT'}=(\sfT', \loops|_{\sfT'})\) is a bunch of grapes over the substem \(\sfT'\). An edge of the stem \(\sfT\), regarded as a substem, is called a \emph{twig} of \(\graf\).

For a subset \(V_0\subseteq V(\sfT)\) or a substem \(\sfT'\subseteq\sfT\), we define \(\loops(V_0) = \sum_{\sfv\in V_0}\loops(\sfv)\) or \(\loops(\sfT')=\loops(V(\sfT'))\).
\end{definition}

\begin{example}
An elementary bunch of grapes in \Cref{definition:FreeGroupCase} has a star graph as its stem. 
The graph given in \Cref{figure:example of sip} is a (non-elementary) bunch of grapes whose stem is the star graph \(\sfS_3\) with the function \(\loops(\sfv)\equiv 1\) constant on all vertices.
\end{example}

\begin{example}\label{example:bunch of grapes}
Let a tree \(\sfT\) and a function \(\loops:V(\sfT)\to\bbZ_{\ge0}\) be given as follows:
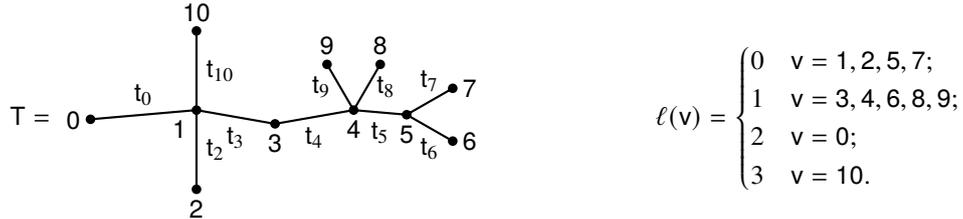
\begin{align*}
\sfT&=\begin{tikzpicture}[baseline=-.5ex, scale=0.7]
\draw[thick,fill] (0,0) circle (2pt) node (A) {} node[left] {\(\sf0\)} -- node[above] {\(\sft_1\)} ++(5:2) circle(2pt) node(B) {} node[below left] {\(\sf1\)} -- node[right] {\(\sft_{10}\)} +(0,1.5) circle (2pt) node (C) {} node[above] {\(\sf{10}\)} +(0,0) -- node[right] {\(\sft_2\)} +(0,-1.5) circle (2pt) node(D) {} node[below] {\(\sf2\)} +(0,0) -- node[below] {\(\sft_3\)} ++(-10:1.5) circle(2pt) node(E) {} node[below] {\(\sf3\)} -- node[below] {\(\sft_4\)} ++(10:1.5) circle(2pt) node(F) {} node[below] {\(\sf4\)} -- node[left] {\(\sft_9\)} +(120:1) circle(2pt) node(G) {} node[above] {\(\sf9\)} +(0,0) -- node[right] {\(\sft_8\)} +(60:1) circle (2pt) node(H) {} node[above] {\(\sf8\)} +(0,0) -- node[below] {\(\sft_5\)} ++(-5:1) circle (2pt) node(I) {} node[below] {\(\sf5\)} -- node[above] {\(\sft_7\)} +(30:1) circle(2pt) node(J) {} node[right] {\(\sf7\)} +(0,0) -- node[below] {\(\sft_6\)} +(-30:1) circle (2pt) node(K) {} node[right] {\(\sf6\)};
\end{tikzpicture}&
\loops(\sfv) &=\begin{cases}
0 & \sfv=\sf1, \sf2, \sf5, \sf7;\\
1 & \sfv=\sf3, \sf4, \sf6, \sf8, \sf9;\\
2 & \sfv=\sf0;\\
3 & \sfv=\sf{10}.
\end{cases}
\end{align*}
Then the bunch of grapes \(\graf=(\sfT,\loops)\) is depicted on the left of \Cref{figure:grapes stem and twigs}.
\end{example}


\begin{assumption}
Unless otherwise specified (for instance, when considering a proper subgraph), any bunch of grapes \(\graf=(\sfT,\loops)\) is assumed to be a minimal simplicial model. That is, \(\loops(\sfv)\ge 1\) for every vertex \(\sfv\in V(\sfT)\) with \(\val_{\sfT}(\sfv)=2\).
\end{assumption}

We say that a bunch of grapes \(\graf=(\sfT,\loops)\) is:
\begin{enumerate}
\item \emph{large} if there exist at least two distinct vertices \(\sfv_1, \sfv_2\) of \(\sfT\) such that \(\loops(\sfv_i)>0\) for \(i=1,2\), and \emph{small} otherwise;
\item \emph{normal} if \(\graf\) is large and \(\loops(\sfv)\ge 1\) for each leaf \(\sfv\in\sfT\).
\end{enumerate}
Using this terminology, we define the subsets
\begin{align*}
\grapegraph_\larg&=\{\graf\in\grapegraph:\graf\text{ is large}\},&
\grapegraph_\norm&=\{\graf\in\grapegraph_\larg:\graf\text{ is normal}\}
\end{align*}

For instance, the graph \(\graf\) described in \Cref{example:bunch of grapes} is large but not normal. It becomes normal if we require \(\loops(\sf2),\loops(\sf7)\ge 1\); see the modified graph \(\graf'\) on the right of \Cref{figure:grapes stem and twigs}, where the required grapes are highlighted in red.

\begin{remark}\label{Lem:WhyNormal}
The terminology for a bunch of grapes \(\graf=(\sfT,\loops)\) may be compared with the graph classes introduced in the previous section as follows:
\begin{enumerate}
\item \(\graf\in\grapegraph_{\larg}\) if and only if \(\graf\in\graphs\) and \(\graf\) contains at least one pair of disjoint cycles. By \Cref{theorem:freeness classification,corollary:hyperbolicity}, \(\graf\notin\grapegraph_{\larg}\) if and only if \(\bbB_2(\graf)\) is free if and only if \(\bbB_2(\graf)\) is hyperbolic.
\item \(\graf\in\grapegraph_{\norm}\) if and only if \(\graf\in\graphs_{(0)}\). In \Cref{Prop:free factor}, we further show that this condition is equivalent to \(\graf\in\graphs_{(2)}\cap\graphs_{(3)}\).
\end{enumerate}
This comparison explains the choice of the adjectives \emph{large} and \emph{normal}.
Moreover, for any large bunch of grapes \(\graf\), the minimal simplicial model \(\graf'\) of its maximal leafless subgraph is a normal bunch of grapes, and by \Cref{corollary:free factor} we have \(\bbB_2(\graf)\cong\bbB_2(\graf')\ast\bbF_N\) for some \(N\ge0\).
Consequently, apart from the free cases, the study of the large-scale geometry of graph \(2\)-braid groups on bunches of grapes reduces to the normal case.
\end{remark}

In the remainder of this section, we assume that \(\graf=(\sfT,\loops)\) is a normal bunch of grapes; that is, \(\graf\in\grapegraph_{\norm}\). We begin by showing that \(\graf\) belongs to \(\graphs_{(2)}\cap\graphs_{(3)}\). Combined with \Cref{Lem:RIconnected} below, this will imply that \(\graf\) indeed belongs to the subclass \(\graphs'\) in \Cref{Thm:SubclassdefiningQII}. 

\begin{proposition}\label{Prop:free factor}
Let \(\graf=(\sfT,\loops)\in\grapegraph_\norm\). Then \(\graf\in\graphs_{(2)}\cap\graphs_{(3)}\).
More precisely, \(UP_2(\graf)\) is locally convex in \(UD_2(\graf)\), and \(\bbB_2(\graf)\isom \pi_1(UP_2(\graf))\ast\bbF_N\) for some \(N\ge\loops(\sfT)\).
\end{proposition}
\begin{proof}
To show that \(UP_2(\graf)\) is locally convex in \(UD_2(\graf)\), it suffices, by \Cref{Lem:CombinatorialConditions}, to verify that \(\graf\) satisfies Condition~(C).

For each vertex \(\sfx\) of \(\sfT\), let \(\grafl_\sfx=\graf_\sfx\cup\st_\sfT(\sfx)\), which is isomorphic to an elementary bunch of grapes \(\grafl(\val_{\sfT}(\sfx),\loops(\sfx))\) from \Cref{definition:FreeGroupCase}.
Observe that two disjoint edges are not separable if and only if both are contained in \(\grafl_{\sfx}\) for some vertex \(\sfx\in\sfT\).
More precisely, two disjoint edges \(\sfe_1\) and \(\sfe_2\) are not separable if either
\begin{itemize}
\item \(\sfe_1\) and \(\sfe_2\) are contained in distinct grapes attached to the same vertex \(\sfx\), or
\item \(\sfe_1\) is a twig incident to \(\sfx\) and \(\sfe_2\) lies in a grape attached to \(\sfx\).
\end{itemize}
In the former case, no vertex of \(\sfe_1\) is separable from \(\sfe_2\). In the latter case, the only separable pair of a vertex and an edge related to \((\sfe_1,\sfe_2)\) is \((\sfw,\sfe_2)\), where \(\sfw\) is the endpoint of \(\sfe_1\) distinct from \(\sfx\). In both cases, Condition~(C) is satisfied. Thus, by \Cref{Lem:CombinatorialConditions}(1), \(UP_2(\graf)\) is locally convex in \(UD_2(\graf)\).

To show that \(\graf\in\graphs_{(2)}\), we consider the following decomposition
\[UD_2(\graf)=UP_2(\graf)\cup \left(\bigcup_{\sfx\in V(\sfT)} UD_2(\grafl_{\sfx})\right).\]
For any subset \(V_0\subsetneq V(\sfT)\) with \(\sfx\not\in V_0\), we have
\begin{align*}
\left(UP_2(\graf)\cup\left(\bigcup_{\sfx'\in V_0} UD_2(\grafl_{\sfx'})\right)\right)\cap UD_2(\grafl_{\sfx})&=
UP_2(\graf)\cap UD_2(\grafl_{\sfx})=UD_2(\st_{\sfT}(\sfx))\cup\left(\graf_{\sfx}\itimes\lk_{\sfT}(\sfx)\right).
\end{align*}
For each \(\sfx'\in\lk_\sfT(\sfx)\), the intersection \(UD_2(\st_\sfT(\sfx))\cap(\graf_\sfx\itimes \sfx')=\sfx\itimes\sfx'\) is a vertex in \(UP_2(\graf)\). Hence,
\begin{align*}
\pi_1(UP_2(\graf)\cap UD_2(\grafl_\sfx))&\isom \pi_1(UD_2(\st_\sfT(\sfx)))\ast\left(\Ast_{\sfx'\in\lk_\sfT(\sfx)}\pi_1(\graf_\sfx\itimes\sfx')\right)\isom \bbB_2(\st_\sfT(\sfx))\ast\left(\Ast_{\sfx'\in\lk_\sfT(\sfx)}\pi_1(\graf_\sfx)\right),
\end{align*}
which is isomorphic to \(\bbF_{N'_\sfx}\) since \(\bbB_2(\st_{\sfT}(\sfx))\) is a free group by \Cref{theorem:elementary rank}.

Moreover, we may decompose 
\begin{align*}
UD_2(\grafl_{\sfx})=\underbrace{\left(UD_2(\st_{\sfT}(\sfx))\cup  \left(\graf_\sfx \itimes \lk_{\sfT}(\sfx)\right)\right)}_{A_1} \cup \underbrace{\left((\graf_\sfx\setminus \sfx)\itimes \st_{\sfT}(\sfx)\right)}_{A_2}\cup \underbrace{UD_2(\graf_\sfx)}_{A_3}, 
\end{align*}
where \(A_1\cap A_2=(\graf_\sfx\setminus\sfx) \itimes \lk_{\sfT}(\sfx)\), \(A_2\cap A_3=(\graf_\sfx\setminus\sfx)\itimes \sfx\) and \(A_1\cap A_3=\varnothing\).
Each component of \(A_2\), \(A_1\cap A_2\), and \(A_2\cap A_3\) is contractible, so van Kampen's theorem shows that \(\pi_1(A_1)\) is a free factor of \(\pi_1(UD_2(\grafl_\sfx))\).
Applying \Cref{theorem:elementary rank} again, we obtain
\[\bbF_{N_\sfx}\isom\pi_1(UD_2(\grafl_\sfx))\isom\bbB_2(\grafl_\sfx)\isom
\pi_1(UP_2(\graf)\cap UD_2(\grafl_\sfx))\ast \bbF_{N''_\sfx}.\]


Adding the subcomplexes \(UD_2(\grafl_\sfx)\) one at a time and applying van Kampen's theorem, using the preceding intersection formula, gives
\[\bbB_2(\graf)\cong\pi_1(UP_2(\graf))\ast\left(\Ast_{\sfx\in V(\sfT)}\bbF_{N_\sfx''}\right).\]
Hence \(\graf\in\graphs_{(2)}\).
\end{proof}
\begin{remark}
For completeness, we record the explicit formulas for the ranks of the free groups appearing in the proof of \Cref{Prop:free factor}:
\[
N_\sfx'=\binom{\val_\sfT(\sfx)-1}2+\val_\sfT(\sfx)\loops(\sfx),\quad\quad 
N_\sfx = \frac{(\val_\sfT(\sfx)+\loops(\sfx))(\val_\sfT(\sfx)+3\loops(\sfx)-3)}2+1,\]
\[N_\sfx'' = N_\sfx-N_\sfx'=\frac{3\loops(\sfx)(\loops(\sfx)-1)}2+\val_\sfT(\sfx)\loops(\sfx)\quad\text{and}\quad N=\sum_{\sfx\in V(\sfT)}\left(\frac{3\loops(\sfx)(\loops(\sfx)-1)}2 + \val_\sfT(\sfx)\loops(\sfx)\right)\ge \loops(\sfT).\]
\end{remark}

\subsection{Structure induced from stems}
We now describe maximal product subcomplexes and their intersections directly in terms of the stem. Twigs correspond to maximal product subcomplexes, while paths determine the standard product subcomplexes arising as their intersections.

Let \(\sfP=[\sfv_0,\dots,\sfv_n]\) (\(n\ge1\)) be a path substem of \(\sfT\).
The \emph{\(\sfP\)-components} of \(\sfT\) are the two components \(\sfT_{\sfP,1}\) and \(\sfT_{\sfP,2}\) of the forest with vertex set \(V(\sfT)\) and edge set \(E(\sfT)\setminus E(\sfP)\), containing \(\sfv_0\) and \(\sfv_n\), respectively.
The two restrictions \(\graf_{\sfT_{\sfP,1}}\) and \(\graf_{\sfT_{\sfP,2}}\) will be denoted by \(\graf_{\sfP,1}\) and \(\graf_{\sfP,2}\), respectively, and called \emph{\(\sfP\)-components} of \(\graf\).
We also denote the standard product subcomplex \(\graf_{\sfP,1}\itimes\graf_{\sfP,2}\) by \(K(\sfP)\).
See \Cref{figure:paths in grapes} for an example.

\begin{figure}[ht]
\begin{align*}
\graf_{\sfP,1}&=
\begin{tikzpicture}[baseline=-.5ex, scale=0.7]
\draw[thick,gray, dashed] (B) +(0,0) -- ++(-10:1.5) -- ++(10:1.5) -- ++(-5:1);
\draw[thick,fill,blue] (0,0) circle (2pt) node (A) {} -- ++(5:2) circle(2pt) node(B) {} -- +(0,1.5) circle (2pt) node (C) {}+(0,0) -- +(0,-1.5) circle (2pt) node(D) {};
\draw (B) ++(-10:0.75) node[above] {\(\sfP\)};
\draw[thick,gray] (F) +(0,0) -- +(60:1) +(0,0) -- +(120:1);
\draw[thick,fill,blue] (B) +(0,0) ++ (-10:1.5) circle(2pt) node(E) {} ++(10:1.5) circle(2pt) node(F) {} +(120:1) circle(2pt) node(G) {} +(0,0) +(60:1) circle (2pt) node(H) {} +(0,0) ++(-5:1) circle (2pt) node(I) {} -- +(30:1) circle(2pt) node(J) {} +(0,0) -- +(-30:1) circle (2pt) node(K) {};
\draw[blue] (1,0.5) node {\(\sfT_{\sfP,1}\)};
\draw[blue] (6,-1) node {\(\sfT_{\sfP,2}\)};
\grape[90]{A};
\grape[270]{A};
\grape[0]{C}; \grape[90]{C}; \grape[180]{C};
\grape[-90]{D};
\grape[-30]{K};
\begin{scope}[gray]
\grape[-90]{E};
\grape[-90]{F};
\grape[120]{G};
\grape[60]{H};
\end{scope}
\grape[30]{J};
\end{tikzpicture}=\graf_{\sfP,2}&
\loops(\sfT_{\sfP,1})&=6,&
\loops(\sfT_{\sfP,2})&=2
\end{align*}
\caption{\(\sfP\)-components of a bunch of grapes \(\graf\).}
\label{figure:paths in grapes}
\end{figure}

Since \(\graf\) is normal, any \(\sfP\)-component of \(\graf\) is a standard subgraph of \(\graf\) in the sense of \Cref{Def:OrthogonalComplement}.
In particular, if \(\sfP\) is a twig \(\sft\), then \(\graf_{\sft,1}^\perp = \graf\setminus\graf_{\sft,1}=\graf_{\sft,2}\), and therefore \(\graf_{\sft,1}\itimes\graf_{\sft,2}\) is a maximal product subcomplex \(M(\sft)\) of \(UD_2(\graf)\).

\begin{proposition}\label{proposition:twig_maximally standard}
Every standard subgraph \(\graf'\) of \(\graf\) is itself a bunch of grapes \(\graf'=(\sfT',\loops')\) for some proper substem \(\sfT'\subsetneq\sfT\).
In particular, every maximally standard subgraph is a \(\sft\)-component of \(\graf\) for some twig \(\sft\), and \textit{vice versa}.
\end{proposition}
\begin{proof}
Since \(\graf'\) is connected and leafless, it contains either all or none of the edges of each grape of \(\graf\). Thus \(\graf'\) consists of the subtree \(\sfT'=\graf'\cap\sfT\) together with grapes attached at its vertices, and we may write \(\graf'=(\sfT',\loops')\).
Since \({\graf'}^\perp\) is nontrivial, it contains a grape disjoint from \(\graf'\), whose attaching vertex lies outside \(\sfT'\). Hence \(\sfT'\) is a proper substem of \(\sfT\).

Assume that \(\graf'\) is maximally standard, i.e., \({\graf'}^{\perp}\) is standard and \(\graf'=({\graf'}^{\perp})^{\perp}\). If \(\sfT\setminus\sfT'\) were disconnected, then since \(\sfT\) is a tree, each component of \(\sfT\setminus\sfT'\) would contain a leaf of \(\sfT\).
As \(\graf\) is normal, each such component contains a vertex carrying a grape. These grapes lie in distinct components of \({\graf'}^\perp\), contradicting its connectedness. Hence \(\sfT\setminus\sfT'\) is connected.
Since \(\sfT\) is a tree, there is a unique twig \(\sft\) incident to but not contained in \(\sfT'\), and therefore \(\sfT'\) is a \(\sft\)-component of \(\sfT\).

The converse was observed prior to the proposition.
\end{proof}

\begin{corollary}\label{Corollary:TwigCorrespondence}
There is a one-to-one correspondence
\(E(\sfT)\rightarrow V(\cI(UP_2(\graf)))\) defined by \(\sft\mapsto M(\sft)=\graf_{\sft,1}\itimes\graf_{\sft,2}.\)
\end{corollary}

For two distinct twigs \(\sft,\sft'\) of \(\graf\), we may assume without loss of generality that \(\graf_{\sft,1}\subsetneq\graf_{\sft',1}\) and \(\graf_{\sft',2}\subsetneq\graf_{\sft,2}\).
Then the intersection \(M(\sft)\cap M(\sft')\) is exactly the standard product subcomplex \(\graf_{\sft,1}\itimes\graf_{\sft',2}\).
This observation generalizes as follows.

\begin{lemma}\label{Lem:RIconnected}
Let \(\{M(\sft_1),\dots,M(\sft_m)\}\) be a finite collection of maximal product subcomplexes of \(UP_2(\graf)\) corresponding to twigs \(\{\sft_1,\dots,\sft_m\}\) of \(\graf\).
Then the following are equivalent:
\begin{enumerate}
\item \(\bigcap_{i=1}^m M(\sft_i)\) is nonempty.
\item All \(\sft_i\)'s are colinear and \(\bigcap_{i=1}^m M(\sft_i)\) is the standard product subcomplex \(K(\sfP)\) where \(\sfP=\sfP\{\sft_1,\dots,\sft_m\}\) is the shortest path substem containing all the \(\sft_i\).
\end{enumerate}
In particular, \(UD_2(\graf)\) satisfies the standard intersection property.
\end{lemma}
We say that twigs \(\sft_1,\dots,\sft_m\) are \emph{colinear} if there exists a path substem \(\sfP\subseteq\sfT\) containing all of them.
\begin{proof}
A direct check shows that if the twigs \(\sft_1,\dots,\sft_m\) are colinear, then \(\bigcap_{i=1}^m M(\sft_i)\) is precisely the nonempty standard product subcomplex \(K(\sfP)\), where \(\sfP=\sfP\{\sft_1,\dots,\sft_m\}\). Thus, it suffices to prove \((1)\Rightarrow(2)\).

Suppose instead that the twigs are not colinear in \(\sfT\). After reordering indices if necessary, we may assume that \(m\ge 3\) and that \(\sft_1,\sft_2,\sft_3\) are non-colinear.
For each \(i\in\{1,2,3\}\), let \(\graf_{\sft_i,1}\) denote the \(\sft_i\)-component of \(\graf\) that is disjoint from the other two twigs.
Then, for \(i\neq j\), the intersection \(M(\sft_i)\cap M(\sft_j)\) is the standard product subcomplex
\(\graf_{\sft_i,1}\itimes\graf_{\sft_j,1}\).
Since \(\graf_{\sft_2,1}\) and \(\graf_{\sft_3,1}\) are disjoint, the product subcomplexes \(\graf_{\sft_1,1}\itimes\graf_{\sft_2,1}\) and \(\graf_{\sft_1,1}\itimes\graf_{\sft_3,1}\) are disjoint as well.
Hence \(\bigcap_{i=1}^3 M(\sft_i)=\varnothing\), and consequently \(\bigcap_{i=1}^m M(\sft_i)=\varnothing\).
\end{proof}

Every \(\CAT(0)\) cube complex satisfies the \emph{Helly} property: any finite collection of pairwise intersecting convex subcomplexes has nonempty intersection.
Using the Helly property, we obtain an analogue of \Cref{Lem:RIconnected} for \(\bar{UP_2(\graf)}\).

\begin{lemma}\label{lem:Intofmaximal}
Let \(\{\bar{M(\sft_1)},\dots,\bar{M(\sft_m)}\}\) be a finite collection of pairwise intersecting maximal product subcomplexes of \(\bar{UP_2(\graf)}\), where each \(\bar{M(\sft_i)}\) is a connected component of \(p^{-1}(M(\sft_i))\) and \(p=p_{UP_2(\graf)}\). Here \(M(\sft_i)\) corresponds to a twig \(\sft_i\) of \(\graf\).
Then the \(\sft_i\)'s are colinear, and \(\bigcap_i\bar{M(\sft_i)}\) is a standard product subcomplex and a connected component of \(p^{-1}(\bigcap_iM(\sft_i))\).
\end{lemma}
\begin{proof}
Distinct components of \(p^{-1}(M(\sft_i))\) are disjoint. Thus, if \(\sft_i=\sft_j\), pairwise intersection implies \(\bar{M(\sft_i)}=\bar{M(\sft_j)}\).
By \Cref{ProjofPS}, the chosen components are convex, so the Helly property gives \(W:=\bigcap_i\bar{M(\sft_i)}\ne\varnothing\).
Consequently,
\[\varnothing\ne p(W)\subseteq\bigcap_iM(\sft_i).\]
By \Cref{Lem:RIconnected}, the twigs are colinear and their downstairs intersection is \(K(\sfP)\), where \(\sfP\) is the shortest path containing them. Since \(W\) is connected, it lies in one connected component \(\bar{K(\sfP)}\) of \(p^{-1}(K(\sfP))\).
For each \(i\), this component lies in \(p^{-1}(M(\sft_i))\) and meets \(\bar{M(\sft_i)}\), so it is contained in that component. Hence \(W=\bar{K(\sfP)}\), which is the image of an elevation of the standard product structure with image \(K(\sfP)\).
\end{proof}

By \Cref{Prop:free factor}, the group \(\bbB_2(\graf)\) has infinitely many ends for every \(\graf\in\grapegraph_\norm\). Thus \Cref{Thm:SubclassdefiningQII,Lem:RIconnected} yield the following.

\begin{theorem}\label{theorem:QIbetweenGBGs}
For \(\graf_1, \graf_2\in\grapegraph_\norm\), the graph \(2\)-braid groups \(\bbB_2(\graf_1)\) and \(\bbB_2(\graf_2)\) are quasi-isometric if and only if \(\pi_1(UP_2(\graf_1))\) and \(\pi_1(UP_2(\graf_2))\) are quasi-isometric. Namely, 
\[\m{(\cdot)}:UD_2\left(\grapegraph_\norm\right)\longrightarrow UP_2\left(\grapegraph_\norm\right)\]
defines a complete quasi-isometry invariant.
\end{theorem}

Moreover, the intersection complexes \(\cI(UP_2(\graf))\) and \(\cI(\bar{UP_2(\graf)})\) capture the algebraic structure of \(\pi_1(UP_2(\graf))\).

\begin{theorem}\label{theorem:structureofI}
For \(\graf\in\grapegraph_\norm\), the intersection complex \(\cI(UP_2(\graf))\) is a connected simplicial complex admitting the structure of a developable complex of groups whose development is \(\cI(\bar{UP_2(\graf)})\). 
Moreover, \(\cI(\bar{UP_2(\graf)})\) is a contractible flag complex. In particular, \(\pi_1(UP_2(\graf))\) is thick, and hence one-ended.
\end{theorem}
\begin{proof}
By \Cref{Lem:RIconnected}, every nonempty finite intersection of maximal product subcomplexes of \(UP_2(\graf)\) is connected and standard. Thus a vertex set determines at most one simplex, and \(\cI(UP_2(\graf))\) is simplicial. By \Cref{Lem:IntersectionCores}, \(\cI(\bar{UP_2(\graf)})\) is simplicial, and by \Cref{Corollary:TwigCorrespondence,lem:Intofmaximal}, it is a flag complex.

By \Cref{Lem:UP_2Special}, \(UP_2(\graf)\) is connected. Moreover, it inherits the embedded product property from \(UD_2(\graf)\) by \Cref{Ex:EPP}, and it satisfies the standard intersection property by \Cref{Lem:RIconnected}.
Thus \Cref{Lem:UP_2One-ended} gives the asserted complex of groups structure, contractibility of \(\cI(\bar{UP_2(\graf)})\), and thickness and one-endedness of \(\pi_1(UP_2(\graf))\).
\end{proof}

\begin{example}\label{Ex:3-star}
By \Cref{Lem:RIconnected}, the intersection complex \(\cI(UP_2(\graf))\) is simplicial and its \(1\)-skeleton is the complete graph \(\sfK_{\#E(\sfT)}\).
For instance, let \(\grafl=(\sfS_3,\loops\equiv 1)\in\grapegraph_\norm\) be the bunch of grapes depicted in \Cref{figure:example of sip}.
Then \(\cI(UP_2(\grafl))\) is a graph given as follows:
\[
\cI(UP_2(\grafl))=\begin{tikzpicture}[baseline=-.5ex]
\draw[fill] (90:1) circle (2pt) node[above] {\(\grafl_{\sfv_1}\itimes\grafl_{\sfe_2\cup\sfe_3}\)};
\draw[fill] (210:1) circle (2pt) node[left] {\(\grafl_{\sfv_2}\itimes\grafl_{\sfe_1\cup\sfe_3}\)};
\draw[fill] (330:1) circle (2pt) node[right] {\(\grafl_{\sfv_3}\itimes\grafl_{\sfe_1\cup\sfe_2}\)};
\draw (90:1)--(210:1) node[midway, left] {\(\grafl_{\sfv_1}\itimes\grafl_{\sfv_2}\)};
\draw (210:1)--(330:1) node[midway, below] {\(\grafl_{\sfv_2}\itimes\grafl_{\sfv_3}\)};
\draw (330:1)--(90:1) node[midway, right] {\(\grafl_{\sfv_1}\itimes\grafl_{\sfv_3}\)};
\end{tikzpicture}
\]
Note that there is no \(2\)-simplex in \(\cI(UP_2(\grafl))\) and \(\grafl_{\sfv_i}\) is a proper subgraph of \(\grafl^\perp_{\sfv_j}=\grafl_{\sfe_{j+1}\cup\sfe_{j+2}}\) for \(i\neq j\), where \(\sfe_4=\sfe_1\) and \(\sfe_5=\sfe_2\).
\end{example}

\subsection{Deformation retractions of intersection complexes}\label{Subsection:IntofBOG}
We use paths in the stem to filter the intersection complexes and construct deformation retractions onto subcomplexes of dimension at most one, compatibly with the canonical quotient map. 
We then show that the filtration, although defined using the stem, is determined by the labeled intersection complex itself.

For a simplex \(\triangle=\{M(\sft_0),\dots,M(\sft_k)\}\subseteq\cI(UP_2(\graf))\), let \(\sfP_\triangle=\sfP\{\sft_0,\dots,\sft_k\}\) be the shortest path substem containing its corresponding twigs. Such a path exists by \Cref{Lem:RIconnected}.
For \(\bar\triangle\subseteq\cI(\bar{UP_2(\graf)})\), set \(\sfP_{\bar\triangle}=\sfP_{\rho(\bar\triangle)}\), where \(\rho=\rho_{UP_2(\graf)}\). Conversely, for a nontrivial path substem \(\sfP\), let
\[\triangle_{\sfP}=\{M(\sft):\sft\in E(\sfP)\}.\]

\begin{lemma}\label{lemma:maximal is longest paths}
The following statements hold:
\begin{enumerate}
\item If \(\triangle\subsetneq\triangle'\) are simplices in \(\cI(UP_2(\graf))\), then \(\sfP_\triangle\subseteq\sfP_{\triangle'}\).
\item The assignment \(\triangle\mapsto\sfP_\triangle\) gives a bijection between maximal simplices in \(\cI(UP_2(\graf))\) and paths in \(\sfT\) joining leaves.
\item Two maximal simplices \(\triangle,\triangle'\) intersect if and only if \(\sfP_\triangle\) and \(\sfP_{\triangle'}\) share an edge.
\end{enumerate}
\end{lemma}
The lemma follows directly from \Cref{Lem:RIconnected}.

For \(X=UP_2(\graf)\) or \(\bar{UP_2(\graf)}\) and \(k\ge1\), let \(\cI_k(X)\) be the subcomplex consisting of all simplices \(\triangle\) such that \(\len(\sfP_\triangle)\le k\). Writing \(D=\diam(\sfT)\), we have \(\cI_D(X)=\cI(X)\). The canonical quotient map respects these filtrations, and we write
\[\rho_k:\cI_k(\bar{UP_2(\graf)})\longrightarrow\cI_k(UP_2(\graf))\]
for its restriction. In particular, the vertices and edges of \(\cI_2(UP_2(\graf))\) correspond to twigs and pairs of adjacent twigs in \(\sfT\), respectively; see \Cref{figure:Ile2UP2}.

\begin{figure}[ht]
\begin{align*}
\cI_{2}(UP_2(\graf))&=
\begin{tikzpicture}[baseline=-.5ex, scale=0.8]
\begin{scope}[lightgray]
\draw[thick,fill] (0,0) circle (2pt) node (A) {} -- ++(5:2) circle(2pt) node(B) {} -- +(0,1.5) circle (2pt) node (C) {}+(0,0) -- +(0,-1.5) circle (2pt) node(D) {} +(0,0) -- ++(-10:1.5) circle(2pt) node(E) {} -- ++(10:1.5) circle(2pt) node(F) {} -- +(120:1) circle(2pt) node(G) {} +(0,0) -- +(60:1) circle (2pt) node(H) {} +(0,0) -- ++(-5:1) circle (2pt) node(I) {} -- +(30:1) circle(2pt) node(J) {} +(0,0) -- +(-30:1) circle (2pt) node(K) {};
\grape[90]{A};
\grape[270]{A};
\grape[0]{C}; \grape[90]{C}; \grape[180]{C};
\grape[-90]{D};
\grape[-90]{E};
\grape[-90]{F};
\grape[120]{G};
\grape[60]{H};
\grape[-30]{K};
\grape[30]{J};
\grape[40]{B};
\grape[-105]{I};
\end{scope}
\draw[thick,fill] 
(5:1) circle (2pt) node (E1) {} node[above] {\(\sft_1\)} (B) +(0,0.75) circle (2pt) node (E2) {} node[above] {\(\sft_{10}\)}
+(0,-0.75) circle (2pt) node (E3) {} node[below] {\(\sft_2\)}
+(-10:0.75) circle (2pt) node (E4) {} node[below] {\(\sft_3\)}
(E) +(10:0.75) circle (2pt) node (E5) {} node[below] {\(\sft_4\)}
(F) +(120:0.5) circle (2pt) node (E6) {} node[above] {\(\sft_9\)}
+(60:0.5) circle (2pt) node (E7) {} node[above] {\(\sft_8\)}
+(-5:0.5) circle (2pt) node (E8) {} node[below] {\(\sft_5\)}
(I) +(30:0.5) circle (2pt) node (E9) {} node[right] {\(\sft_7\)}
+(-30:0.5) circle (2pt) node (E10) {} node[right] {\(\sft_6\)}
;
\draw[thick] 
(E1.center) -- (E2.center) -- (E3.center) -- (E4.center) -- (E1.center) -- node[midway,below left] {\(\sfK(\sf1)\)} (E3.center)
(E2.center) -- (E4.center) -- node[midway,above] {\(\sfK(\sf3)\)} (E5.center) -- (E6.center) -- (E7.center) -- (E8.center) -- node[midway,below] {\(\sfK(\sf4)\)} (E5.center) -- (E7.center) 
(E6.center) -- (E8.center) -- node[midway,above] {\(\sfK(\sf5)\)} (E9.center) -- (E10.center) -- (E8.center)
;
\end{tikzpicture}&
\begin{cases}
\sfK(\sf1)\isom \sfK_4,\\
\sfK(\sf3)\isom \sfK_2,\\
\sfK(\sf4)\isom \sfK_4,\\
\sfK(\sf5)\isom \sfK_3
\end{cases}
\end{align*}
\caption{\(\cI_{2}(UP_2(\graf))\) superimposed on a normal bunch of grapes \(\graf\), with the latter shown in gray.}
\label{figure:Ile2UP2}
\end{figure}

\begin{lemma}\label{lemma:deformation retract}
For each \(k\ge 2\), there exist strong deformation retractions 
\[r_k:\cI_{k+1}(UP_2(\graf))\to \cI_{k}(UP_2(\graf))\quad\text{and}\quad\bar r_k:\cI_{k+1}(\bar{UP_2(\graf)})\to \cI_{k}(\bar{UP_2(\graf)})\]
such that \(\rho_{k}\circ\bar r_k=r_k\circ\rho_{k+1}\).
\end{lemma}
\begin{proof}
If \(D\le2\), all the asserted retractions are identities. We may therefore assume that \(D\ge3\). By the definition of the filtration and \(\rho_k\), there is a commutative diagram whose rows are filtrations:
\[
\begin{tikzcd}[column sep=2pc, row sep=2pc]
\cI_{2}(\bar{UP_2(\graf)})\arrow[r]\arrow[d,"\rho_{2}"]&\cI_{3}(\bar{UP_2(\graf)})\arrow[r]\arrow[d,"\rho_{3}"]&\cdots\arrow[r]&
\cI_{D}(\bar{UP_2(\graf)})
\arrow[r,equal]\arrow[d,"\rho_{D}"]&\cI(\bar{UP_2(\graf)})\arrow[d,"\rho"]\\
\cI_{2}(UP_2(\graf))\arrow[r]&\cI_{3}(UP_2(\graf))\arrow[r]&
\cdots\arrow[r]&
\cI_{D}(UP_2(\graf))\arrow[r,equal]&\cI(UP_2(\graf)),
\end{tikzcd}
\] 
where \(D\) is the diameter of \(\sfT\).

Fix \(k\ge2\). A \(k\)-simplex \(\triangle\) of \(\cI_{k+1}(UP_2(\graf))\) corresponds to a path consisting of consecutive twigs \(\sft_0,\dots,\sft_k\).
Let \(\triangle_0=\{M(\sft_0),M(\sft_k)\}\), and let \(\triangle_1,\triangle_2\) be the two codimension-one faces obtained by omitting \(M(\sft_k)\) and \(M(\sft_0)\), respectively.
Every simplex of \(\cI_{k+1}(UP_2(\graf))\) containing \(\triangle_0\) is a face of \(\triangle\). Thus \(\triangle\) meets the remainder of the complex only in \(\triangle_1\cup\triangle_2\).

Choose the following symmetric deformation retraction of \(\triangle\) onto \(\triangle_1\cup\triangle_2\). In barycentric coordinates \((t_0,\dots,t_k)\), put \(s=\min\{t_0,t_k\}\) and set
\[H_u(t_0,\dots,t_k)=\left(t_0-us, t_1+\frac{2us}{k-1},\dots, t_{k-1}+\frac{2us}{k-1},t_k-us\right), \qquad 0\le u\le1.\]
This homotopy fixes \(\triangle_1\cup\triangle_2\) pointwise and is invariant under permutations preserving the unordered pair \(\{M(\sft_0),M(\sft_k)\}\).
See \Cref{figure:deformation retract on simplex}. Gluing these homotopies, and fixing the remaining simplices, gives the required strong deformation retraction \(r_k\).

For a \(k\)-simplex \(\bar\triangle\) with \(\rho(\bar\triangle)=\triangle\), let \(\bar\triangle_0\) be its face over \(\triangle_0\). By \Cref{Lem:RIconnected,lem:Intofmaximal}, they have the same core \(\bar K\), a connected component of \(p^{-1}(K(\sfP_\triangle))\), where \(p=p_{UP_2(\graf)}\). Any \(k\)-simplex at this level containing \(\bar\triangle_0\) projects to \(\triangle\), its unique downstairs coface at this level. Its core is a component of the same inverse image contained in \(\bar K\), and therefore equals \(\bar K\). For each \(0\le j\le k\), its vertex over \(M(\sft_j)\) is the unique component of \(p^{-1}(M(\sft_j))\) containing \(\bar K\). Hence \(\bar\triangle\) is the unique such upstairs coface.
Transport the preceding homotopy through \(\rho|_{\bar\triangle}\). These homotopies agree on intersections and glue, together with the identity on the remaining simplices, to a strong deformation retraction \(\bar r_k\).
By construction, \(\rho_k\circ\bar r_k=r_k\circ\rho_{k+1}\).
\end{proof}

\begin{figure}[ht]
\begin{subfigure}{0.49\textwidth}
\centering
\begin{adjustbox}{width=\textwidth}
\begin{tikzcd}
\begin{tikzpicture}[baseline=-.5ex]
\fill[lightgray,opacity=0.5] (90:2) -- (210:2) -- (330:2) -- cycle;
\draw[thick] (0,0) node {\(\triangle\)}
(90:2) node[above] {\(M(\sft_0)\)} -- node[midway,sloped, above] {\(\triangle_1=\{M(\sft_0),M(\sft_1)\}\)}
(210:2) node[left] {\(M(\sft_1)\)} -- node[midway,below] {\(\triangle_2=\{M(\sft_1),M(\sft_2)\}\)} 
(330:2) node[right] {\(M(\sft_2)\)} -- node[midway,sloped, above] {\(\triangle_0=\{M(\sft_0),M(\sft_2)\}\)}
cycle
;
\draw[fill] 
(90:2) circle (2pt)
(210:2) circle (2pt)
(330:2) circle (2pt)
;
\foreach \i in {-4,-3,-2,-1,0} {
\draw[thin,->,gray] (30:1) +({\i*sqrt(3)*cos(-60)/5},{\i*sqrt(3)*sin(-60)/5}) -- ({2*cos(210)-2*sqrt(3)*cos(60)*\i/5},{2*sin(210)-2*sqrt(3)*sin(60)*\i/5});
}
\foreach \i in {1,2,3,4} {
\draw[thin,->,gray] (30:1) +({\i*sqrt(3)*cos(-60)/5},{\i*sqrt(3)*sin(-60)/5}) -- ({2*cos(210)+2*sqrt(3)*\i/5},{2*sin(210)});
}
\end{tikzpicture}\ar[r,"\displaystyle r_\triangle"]&
\begin{tikzpicture}[baseline=-.5ex]
\draw[thick]
(90:2) node[above] {\(M(\sft_0)\)} -- node[midway,sloped, above] {\(\triangle_1=\{M(\sft_0),M(\sft_1)\}\)}
(210:2) node[left] {\(M(\sft_1)\)} -- node[midway,below] {\(\triangle_2=\{M(\sft_1),M(\sft_2)\}\)} 
(330:2) node[right] {\(M(\sft_2)\)}
;
\draw[fill] 
(90:2) circle (2pt)
(210:2) circle (2pt)
(330:2) circle (2pt)
;
\end{tikzpicture}
\end{tikzcd}
\end{adjustbox}
\caption{\(k=2\)}
\end{subfigure}
\begin{subfigure}{0.49\textwidth}
\centering
\begin{adjustbox}{width=\textwidth}
\begin{tikzcd}
\begin{tikzpicture}[baseline=-.5ex]
\fill[lightgray,opacity=0.5] (-1,1.5) -- ++(15:4) -- ++(0,{-3-8*sin(15)}) -- cycle;
\fill[lightgray,opacity=0.5] (-1,1.5) -- (-1,-1.5) -- ++(-15:4) -- cycle;
\draw[thick] (-1,-1.5) ++(-15:4) -- (-1,1.5);
\foreach \i in {-4,-3,-2,-1,0} {
\draw[thin,->] (-1,-1.5) ++(-15:4) ++(0,{1.5+4*sin(15)+\i*(1.5+4*sin(15))/5}) -- ++({-(5+\i)*(4*cos(15))/5},0);
}
\foreach \i in {1,2,3,4} {
\draw[thin,->] (-1,-1.5) ++(-15:4) ++(0,{1.5+4*sin(15)+\i*(1.5+4*sin(15))/5}) -- ++({-(5-\i)*(4*cos(15))/5},0);
}
\fill[lightgray,opacity=0.5] (-1,-1.5) -- (-1,1.5) -- ++(15:4) -- cycle;
\fill[lightgray,opacity=0.5] (-1,-1.5) -- ++(-15:4) -- ++(0,{3+8*sin(15)}) -- cycle;
\draw[thick] 
(-1,-1.5) node[left] {\(M({\sft_2})\)} -- (-1,1.5) node[left] {\(M({\sft_1})\)} -- node[midway,sloped,below] {\(\triangle_1=\{M({\sft_0}),M({\sft_1}),M({\sft_2})\}\)} ++(15:4) node (E1) {} node[right] {\(M({\sft_0})\)} -- cycle -- node[midway,sloped,above] {\(\triangle_2=\{M({\sft_1}),M({\sft_2}),M({\sft_3})\}\)} ++(-15:4) node (E2) {} node[right] {\(M({\sft_3})\)}
({-1+4*cos(15)},{1.5+4*sin(15)}) -- ({-1+4*cos(-15)},{-1.5+4*sin(-15)}) node[midway, above,sloped] {\(\triangle_0=\{M({\sft_0}),M({\sft_3})\}\)}
;
\draw (1.5,0) node {\(\triangle\)};
\end{tikzpicture}\ar[r,"\displaystyle r_\triangle"]&
\begin{tikzpicture}[baseline=-.5ex]
\fill[lightgray,opacity=0.5] (-1,1.5) -- (-1,-1.5) -- ++(-15:4) -- cycle;
\draw[thick] (-1,-1.5) ++(-15:4) -- (-1,1.5);
\fill[lightgray,opacity=0.5] (-1,-1.5) -- (-1,1.5) -- ++(15:4) -- cycle;
\draw[thick] 
(-1,-1.5) node[left] {\(M({\sft_2})\)} -- (-1,1.5) node[left] {\(M({\sft_1})\)} -- node[midway,sloped,below] {\(\triangle_1=\{M({\sft_0}),M({\sft_1}),M({\sft_2})\}\)} ++(15:4) node (E1) {} node[right] {\(M({\sft_0})\)} -- cycle -- node[midway,sloped,above] {\(\triangle_2=\{M({\sft_1}),M({\sft_2}),M({\sft_3})\}\)} ++(-15:4) node (E2) {} node[right] {\(M({\sft_3})\)}
;
\end{tikzpicture}
\end{tikzcd}
\end{adjustbox}
\caption{\(k=3\)}
\end{subfigure}
\caption{Deformation retraction \(r_\triangle\)}
\label{figure:deformation retract on simplex}
\end{figure}

\begin{corollary}\label{corollary:developable complex-of-groups}
For each \(k\ge2\), \(\cI_k(UP_2(\graf))\) admits a developable complex of groups structure with fundamental group \(\pi_1(UP_2(\graf))\) and universal development \(\cI_k(\bar{UP_2(\graf)})\).
\end{corollary}
\begin{proof}
The deck action restricts to \(\cI_k(\bar{UP_2(\graf)})\), with quotient \(\cI_k(UP_2(\graf))\).
By \Cref{theorem:structureofI,lemma:deformation retract}, the upstairs complex is contractible. The associated complex of groups therefore has the asserted fundamental group and universal development.
\end{proof}

By \Cref{theorem:structureofI,lemma:deformation retract}, the underlying complex \(\lvert\cI_2(\bar{UP_2(\graf)})\rvert\) is a tree. By contrast, \(\lvert\cI_2(UP_2(\graf))\rvert\) is a tree-like union of complete graphs \(\sfK_{\val_\sfT(\sfv)}\) and may contain nontrivial loops; see \Cref{figure:Ile2UP2}.
The following proposition characterizes when such loops are absent.


\begin{proposition}\label{SM-types}
For \(\graf=(\sfT,\loops)\in\grapegraph_{\norm}\), the following are equivalent:
\begin{enumerate}
\item The stem \(\sfT\) is a path graph.
\item The intersection complex \(\cI(UP_2(\graf))\) is a simplex.
\item The intersection complex \(\cI(UP_2(\graf))\) is simply connected.
\end{enumerate}
\end{proposition}
\begin{proof}
It is immediate that (1) implies (2), and (2) implies (3).

To show that (3) implies (1), suppose that \(\cI(UP_2(\graf))\) is simply connected. By \Cref{lemma:deformation retract}, the subcomplex \(\cI_{2}(UP_2(\graf))\) is then also simply connected. As observed above and illustrated in \Cref{figure:Ile2UP2}, this subcomplex is precisely the graph encoding adjacency between twigs.
If \(\sfT\) has a vertex of valency \(m\ge 3\), then \(\cI_{2}(UP_2(\graf))\) contains an embedded complete graph \(\sfK_m\), which is not simply connected. This contradicts the simple connectedness of \(\cI(UP_2(\graf))\). Hence, \(\sfT\) cannot have an essential vertex, and therefore \(\sfT\) is a path graph as claimed.
\end{proof}

Although the filtration was defined using paths in the stem, the following lemma shows that it is intrinsic to the labeled intersection complex. In particular, isomorphisms preserve the subcomplex recording twig adjacency, and the deformation retractions can be chosen naturally with respect to these isomorphisms.

\begin{lemma}\label{Lem:IntrinsicFiltration}
Let \(\graf,\graf'\in\grapegraph_\norm\). Every isomorphism \(\psi:\cI(UP_2(\graf))\to\cI(UP_2(\graf'))\) restricts to isomorphisms
\[\psi_k:\cI_k(UP_2(\graf))\longrightarrow\cI_k(UP_2(\graf'))\qquad (k\ge1).\]
With the symmetric deformation retractions chosen in the proof of \Cref{lemma:deformation retract}, and the usual affine realizations of simplicial isomorphisms, these restrictions satisfy
\[\psi_k\circ r_k=r'_k\circ\psi_{k+1}\qquad (k\ge2).\]
The analogous statements hold for isomorphisms between \(\cI(\bar{UP_2(\graf)})\) and \(\cI(\bar{UP_2(\graf')})\).
\end{lemma}

\begin{proof}
Every simplex \(\triangle\) has a unique largest coface \(\widehat\triangle\) with the same label. Downstairs, it is \(\triangle_{\sfP_\triangle}\): adding intermediate twigs does not change the associated intersection, while extending \(\sfP_\triangle\) changes it, by \Cref{Lem:RIconnected}. 
The corresponding assertion upstairs follows from \Cref{lem:Intofmaximal}: for each additional maximal product subcomplex downstairs, adjoin the unique connected component of its inverse image containing the given core label.
In either case,
\[\len(\sfP_\triangle)=\dim\widehat\triangle+1.\]
An isomorphism preserves inclusions and equality of labels by \Cref{Def:Morphism}, so it sends \(\widehat\triangle\) to \(\widehat{\psi(\triangle)}\).
It therefore preserves all levels of the filtration.

For a \(k\)-simplex \(\triangle\subseteq\cI_{k+1}\), \(k\ge2\), the edge joining its two terminal twigs is the unique smallest face with the same label as \(\triangle\).
Hence an isomorphism preserves this unordered pair of vertices. The symmetric barycentric construction in the proof of \Cref{lemma:deformation retract} is invariant under these permutations, giving the asserted commutativity. The same argument applies upstairs.
\end{proof}

\section{Applications to \texorpdfstring{\(2\)}{2}-braid groups on bunches of grapes}\label{Section:Application to bunches of grapes}
In this section, we first investigate a specific geometric feature of the intersection complexes \(\cI(UP_2(\graf))\) and \(\cI(\overline{UP_2(\graf)})\) for a normal bunch of grapes \(\graf=(\sfT,\loops)\), arising from the structure of its stem \(\sfT\). 
Using this description, we then provide a partial answer to \Cref{Question:QItoRAAG} concerning which graph braid groups are quasi-isometric to RAAGs. 
Finally, we use the \(\m Y\)-hierarchy to construct graph \(2\)-braid groups that are hyperbolic relative to thick proper subgroups.

\subsection{Sequences of leaves and maximal simplices}\label{subsection:sequences of leaves}
Throughout this subsection, let \(\graf=(\sfT,\loops)\in\grapegraph_\norm\). 
Since \(\cI(UP_2(\graf))\) is simply connected when \(\sfT\) is a path by \Cref{SM-types}, we assume that \(\sfT\) is not a path. We describe loops in \(\cI(UP_2(\graf))\) using \emph{sequences of leaves} in \(\sfT\), and use this description to rule out certain closed sequences of maximal simplices in \(\cI(\bar{UP_2(\graf)})\).

Let \(\sfw_\bullet=(\sfw_0,\sfw_1,\dots,\sfw_n)\) be a sequence of leaves in \(\sfT\) such that \(\sfw_i\neq\sfw_{i+1}\) for \(0\le i<n\). For each \(i\), let \(\sfP_i(\sfw_\bullet)\) be the unique embedded path in \(\sfT\) from \(\sfw_i\) to \(\sfw_{i+1}\), with this orientation.
Each such path has length at least two, and consecutive paths share the twig incident to their common leaf. Their concatenation is denoted by \(\sfP(\sfw_\bullet)\); this path need not be embedded. In particular, \(\sfP(\sfw_i,\sfw_{i+1})=\sfP_i(\sfw_\bullet)\).

An embedded path in \(\sfT\) consisting of \(k\ge1\) twigs determines a path of length \(k-1\) in \(\cI_2(UP_2(\graf))\) by \Cref{Lem:RIconnected,lemma:deformation retract}.
Let \(\cP_i(\sfw_\bullet)\) be the (oriented) path in \(\cI_{2}(UP_2(\graf))\) corresponding to \(\sfP_i(\sfw_\bullet)\).
Each \(\cP_i(\sfw_\bullet)\) has length at least one, and \(\cP_i(\sfw_\bullet)\cap\cP_{i+1}(\sfw_\bullet)\neq\varnothing\).
We define \(\cP(\sfw_\bullet)\) to be the concatenation of these paths; in particular, if \(\sfw_0=\sfw_n\), then \(\cP(\sfw_\bullet)\) is a loop in \(\cI(UP_2(\graf))\).

Furthermore, we define a sequence of maximal simplices in \(\cI(UP_2(\graf))\) by
\[\triangle_\bullet(\sfw_\bullet)=(\triangle_0(\sfw_\bullet),\dots,\triangle_{n-1}(\sfw_\bullet)),\qquad\triangle_i(\sfw_\bullet)=\triangle_{\sfP_i(\sfw_\bullet)}.\]
It follows that \(\triangle_i(\sfw_\bullet)\cap\triangle_{i+1}(\sfw_\bullet)=\triangle_{\sfP_i(\sfw_\bullet)\cap\sfP_{i+1}(\sfw_\bullet)}\neq\varnothing\).
Moreover, each path \(\cP_i(\sfw_\bullet)\) is a subcomplex of \(\triangle_i(\sfw_\bullet)\), namely \(r_{2}\circ\dots\circ r_{D-1}(\triangle_i(\sfw_\bullet))\) in \(\cI_{2}(UP_2(\graf))\subseteq\cI(UP_2(\graf))\), where \(D=\diam(\sfT)\).

\begin{proposition}\label{proposition:leaf sequence}
Let \(\sfw_\bullet=(\sfw_0,\dots,\sfw_n=\sfw_0)\) (\(n\ge3\)) be a closed sequence of leaves with \(\sfw_i\neq\sfw_{i+1}\) for every \(i\).
If, with indices taken modulo \(n\),
\[\triangle_i(\sfw_\bullet)\cap\triangle_{i+1}(\sfw_\bullet)\cap\triangle_{i+2}(\sfw_\bullet)=\varnothing\qquad\text{for every }i,\]
then \(\cP(\sfw_\bullet)\) is not null-homotopic in \(\cI(UP_2(\graf))\).
\end{proposition}
\begin{proof}
Abbreviate \(\cP_i(\sfw_\bullet)\) to \(\cP_i\), with indices modulo \(n\).
Let \(\cQ_i\) be the maximal common initial subpath of \(\cP_i^{-1}\) and \(\cP_{i+1}\). Its image is \(\cP_i\cap\cP_{i+1}\). The triple-intersection hypothesis implies that \(\cQ_{i-1}\) and \(\cQ_i\) are disjoint subpaths of \(\cP_i\).
Consequently,
\[\cP_i=\cQ_{i-1}\cP_i'\cQ_i^{-1},\]
where \(\cP_i'\) has positive length. By maximality of \(\cQ_i\), the concatenation \(\cP_i'\cP_{i+1}'\) has no backtracking. Thus \(\cP(\sfw_\bullet)\) is freely homotopic to the nonconstant cyclically reduced loop \(\cP_0'\cdots\cP_{n-1}'\) in \(\cI_2(UP_2(\graf))\).
Such a loop is not null-homotopic in this graph, and hence not in \(\cI(UP_2(\graf))\), by \Cref{lemma:deformation retract}.
\end{proof}

\begin{figure}[ht]
\centering
\begin{subfigure}{.48\textwidth}
\[
\tilde\sfD_n=
\begin{tikzpicture}[baseline=-.5ex, scale=0.9, transform shape]
\draw[fill] (-1,0) +(120:1) circle (2pt) -- +(0,0) circle (2pt) -- +(240:1) circle (2pt) +(0,0) -- ++(1,0) circle (2pt) ++(1,0) circle (2pt) --++(1,0) +(60:1) circle (2pt) -- +(0,0) circle (2pt) -- +(-60:1) circle (2pt);
\draw[dashed] (0,0) -- (1,0);
\draw (0.5,0) node[below] {\(\underbrace{\hspace{3cm}}_{n-4}\)};
\end{tikzpicture}
\]
\vspace{45pt}
\caption{The affine Dynkin diagram \(\tilde\sfD_n\)}
\label{figure:affine dynkin diagram}
\end{subfigure}
\begin{subfigure}{.48\textwidth}
\[
\sfT_{a,b,c}=
\begin{tikzpicture}[baseline=-.5ex, scale=0.9, transform shape]
\draw[fill] (0,0) circle (2pt) --++(60:1) circle (2pt) ++(60:1) circle (2pt);
\draw[fill] (0,0) -- ++(180:1) circle (2pt) + (180:1) circle (2pt) +(0,0) -- node[midway,below] {\(\underbrace{\hphantom{\hspace{4cm}}}_{c}\)} ++(180:2) circle (2pt) ++(180:1) circle (2pt);
\draw[fill] (0,0) -- ++(-60:1) circle (2pt) -- node[midway,sloped,above] {\(\overbrace{\hphantom{\hspace{3cm}}}^{b}\)} ++(-60:1) circle (2pt) ++(-60:1) circle (2pt);
\draw[dashed] (0,0) --node[midway,sloped,above] {\(\overbrace{\hphantom{\hspace{2cm}}}^{a}\)}  ++(60:2) circle (2pt) (60:1) -- (60:2)
(-60:2) -- (-60:3) (180:3) -- (180:4);
\end{tikzpicture}
\]
\caption{The tripod \(\sfT_{a,b,c}\) of type \((a,b,c)\)}
\label{figure:tripod}
\end{subfigure}
\caption{The affine Dynkin diagram \(\tilde\sfD_n\) and the tripod \(\sfT_{a,b,c}\)}
\label{figure:obstruction graphs}
\end{figure}

We say that a graph \(\grafl'\) \emph{leaf-respectingly embeds} into a graph \(\grafl\)  if there is an embedding \(\grafl'\hookrightarrow\grafl\) that maps leaves of \(\grafl'\) to leaves of \(\grafl\).
We are specifically interested in leaf-respecting embeddings of the following stems into the stem \(\sfT\):
\begin{definition}\label{Definition:Dynkin and tripod}
\begin{itemize}
\item An \emph{affine Dynkin diagram} \(\tilde\sfD_n\) (\(n\ge 5\)) is a tree with exactly two trivalent vertices and four leaves. The remaining \(n-5\) vertices are bivalent and lie on the path of length \(n-4\) joining the two trivalent vertices. See \Cref{figure:affine dynkin diagram}.
\item A \emph{tripod} of type \((a,b,c)\) (\(a,b,c\ge 1\)) is a tree with exactly one trivalent vertex and three leaves such that the lengths of the paths from the trivalent vertex to the three leaves are \(a,\, b\), and \(c\), respectively. See \Cref{figure:tripod}.
\end{itemize}
\end{definition}

It is straightforward to verify that any leaf-respecting embedded substem isomorphic to \(\tilde\sfD_n\) (for \(n\ge 5\)) or \(\sfT_{a,b,c}\) (for \(1\le a<b<c\)) yields a sequence of leaves satisfying the condition of \Cref{proposition:leaf sequence}; specifically, the corresponding sequences have the form
\[\sfw^1_{\bullet}=(\sfw_0,\dots,\sfw_3,\,\sfw_4=\sfw_0)\quad\text{ or }\quad\sfw^2_{\bullet}=(\sfw'_0,\dots,\sfw'_3=\sfw'_0),\] 
where the \(\sfw_i\) are leaves of \(\tilde\sfD_n\) and the \(\sfw'_j\) are leaves of \(\sfT_{a,b,c}\).

We now introduce terminology for a sequence of maximal simplices \(\triangle_\bullet=(\triangle_0,\dots,\triangle_{N})\) in \(\cI({UP_2(\graf)})\)  (or \(\cI(\bar{UP_2(\graf)})\)) satisfying \(\triangle_i\cap\triangle_{i+1}\neq\varnothing\).
We say that the sequence \(\triangle_\bullet\) is
\begin{enumerate}
\item
of \emph{edge-simplex type} if \(\triangle_i\cap\triangle_{i+1}\cap\triangle_{i+2}=\varnothing\) and \(\dim\triangle_i\) is \(1\) if \(i\) is even, and \(\ge 2\) if \(i\) is odd.
\item
of \emph{type \((a,b,c)\)} for some \(1\le a<b<c\) if every vertex of \(\triangle_{i+1}\) lies in \(\triangle_i\cup\triangle_{i+2}\) and the following conditions hold:
\[
\dim\triangle_i=
\begin{cases}
a+b-1 & \text{if } i\equiv 0 \pmod{3},\\[2pt]
b+c-1 & \text{if } i\equiv 1 \pmod{3},\\[2pt]
c+a-1 & \text{if } i\equiv 2 \pmod{3},
\end{cases}
\quad\text{and}\quad
\dim(\triangle_i\cap\triangle_{i+1})\ge
\begin{cases}
b-1 & \text{if } i\equiv 0 \pmod{3},\\[2pt]
c-1 & \text{if } i\equiv 1 \pmod{3},\\[2pt]
a-1 & \text{if } i\equiv 2 \pmod{3}.
\end{cases}
\]
\end{enumerate}

\begin{remark}\label{Rmk:TypeofSequence}
The above notions apply to sequences of maximal simplices in any intersection complex and are preserved by isomorphisms. 
For maps between simplicial intersection complexes, they are also preserved by combinatorial maps (see \Cref{Def:Morphism}(1)) that are injective on each simplex and send maximal simplices to maximal simplices.
For type \((a,b,c)\), this follows from the vertex-coverage condition and the lower bounds on intersection dimensions. For edge-simplex type, maximality ensures that consecutive images intersect at a single vertex, while injectivity on the middle simplex keeps the two intersection vertices distinct. For a closed sequence, the defining conditions are understood cyclically.
\end{remark}

Since \(\cI(\bar{UP_2(\graf)})\) is the development of \(\cI(UP_2(\graf))\), the nontrivial loop \(\cP(\sfw^1_{\bullet})\) lifts to a bi-infinite path in \(\cI(\bar{UP_2(\graf)})\). Consequently, the sequence of maximal simplices \(\triangle_{\bullet}(\sfw^1_{\bullet})=(\triangle_0,\dots,\triangle_3)\) in \(\cI(UP_2(\graf))\) induced by \(\sfw^1_\bullet\) lifts to a bi-infinite sequence of maximal simplices \(\bar\triangle_\bullet=(\dots,\bar\triangle_{-1},\bar\triangle_0,\bar\triangle_1,\dots)\) in \(\cI(\bar{UP_2(\graf)})\) satisfying \(\rho(\bar\triangle_{4n+i})=\triangle_i\) for all \(n\in\bbZ\), \(0\le i\le 3\). 
Here, at each step, we choose the next lifted simplex to contain the entire face lifting the intersection of the two consecutive simplices downstairs. Such a choice is possible by the coface-lifting property in \Cref{prop:EPPquotient}.
The consecutive intersections then project isomorphically onto their downstairs counterparts, so the intersection dimensions and vertex-coverage conditions are preserved. In particular, \(\bar\triangle_\bullet\) is of edge-simplex type.
Applying the same construction to \(\triangle_\bullet(\sfw^2_\bullet)\) yields a bi-infinite sequence of maximal simplices in \(\cI(\bar{UP_2(\graf)})\) of type \((a,b,c)\).

The next two lemmas give conditions under which a finite sequence of maximal simplices cannot close upstairs when its projection closes downstairs.

\begin{lemma}\label{proposition:edge-simplex type}
Let \(\bar\triangle_\bullet=(\bar\triangle_0,\dots,\bar\triangle_{N})\) be a sequence of maximal simplices in \(\cI(\bar{UP_2(\graf)})\) of edge-simplex type.
If \((\rho(\bar\triangle_0),\dots,\rho(\bar\triangle_{N}))\) forms a closed sequence of maximal simplices of edge-simplex type, with the defining conditions understood cyclically, then \(\bar\triangle_0\neq\bar\triangle_{N}\).
\end{lemma}
\begin{proof}
Let \(\triangle_i=\rho(\bar\triangle_i)\) and \(\sfP_i=\sfP_{\triangle_i}\), with indices taken modulo \(N\). Each consecutive intersection \(\triangle_i\cap\triangle_{i+1}\) is a single vertex, since one of the two simplices is a maximal edge. This vertex corresponds to a terminal twig of the associated path of length two, and therefore to a twig incident to a leaf of \(\sfT\).
The two intersection vertices on either side of \(\triangle_i\) are distinct, by the cyclic triple-intersection condition. Thus the corresponding twigs contain the two distinct leaves of \(\sfP_i\).
It follows that there is a closed sequence of leaves \(\sfw_\bullet=(\sfw_0,\dots,\sfw_N=\sfw_0)\) with \(\sfP_i=\sfP(\sfw_i,\sfw_{i+1})\).

By \Cref{proposition:leaf sequence}, \(\cP(\sfw_\bullet)\) is not null-homotopic in \(\cI(UP_2(\graf))\). If \(\bar\triangle_0=\bar\triangle_N\), the sequence \(\bar\triangle_\bullet\) would carry a loop whose projection is homotopic to \(\cP(\sfw_\bullet)\), contradicting the simple connectedness of \(\cI(\bar{UP_2(\graf)})\).
\end{proof}

In general, \(\dim(\bar\triangle\cap\bar\triangle')\) may be strictly smaller than \(\dim(\rho(\bar\triangle)\cap\rho(\bar\triangle'))\). The preceding proof relies on the presence of a \(1\)-dimensional maximal simplex and the fact that \(\rho\) preserves maximal simplices to circumvent this issue.
To address this subtlety for sequences of type \((a,b,c)\), we define the set of available tripod types for a tree \(\sfT\):
\[\tripod(\sfT)=\{(a,b,c):1\le a<b<c,\ \sfT_{a,b,c}\text{ leaf-respectingly embeds into }\sfT\}.\]
We order \(\tripod(\sfT)\) lexicographically by
\[(a,b,c)\longmapsto (-(b+c),\,a+b,\,a).\]
Thus a minimal type has maximal diameter and, subject to this, minimal shortest leaf-to-leaf distance.

\begin{lemma}\label{proposition:type (a_b_c)}
Suppose that \(\tripod(\sfT)\) is nonempty and \((a,b,c)\) is its minimal element. Let \(\bar\triangle_\bullet=(\bar\triangle_0,\dots,\bar\triangle_N)\) be a sequence of maximal simplices in \(\cI(\bar{UP_2(\graf)})\) of type \((a,b,c)\), where \(N\ge3\). If \((\rho(\bar\triangle_0),\dots,\rho(\bar\triangle_N))\) is a closed sequence of type \((a,b,c)\), with the defining conditions understood cyclically, then \(\bar\triangle_0\neq\bar\triangle_N\).
\end{lemma}
\begin{proof}
Let \(\triangle_i=\rho(\bar\triangle_i)\) and \(\sfP_i=\sfP_{\triangle_i}\), with indices taken modulo \(N\). By assumption, \(N=3m\) for some \(m\ge1\). The vertex-coverage condition implies that the two terminal twigs of each \(\sfP_i\) lie in its two neighboring paths.
They cannot both lie in the same neighboring path, by maximality. Hence there is a closed sequence of leaves \(\sfw_\bullet=(\sfw_0,\dots,\sfw_N=\sfw_0)\) such that \(\sfP_i=\sfP(\sfw_i,\sfw_{i+1})\).

Write
\[\len(\sfP_{3k}\cap\sfP_{3k+1})=b+b_k,\quad
\len(\sfP_{3k+1}\cap\sfP_{3k+2})=c+c_k,\quad
\len(\sfP_{3k+2}\cap\sfP_{3k+3})=a+a_k,\]
where \(a_k,b_k,c_k\ge0\). Since consecutive paths are distinct paths joining leaves, \(b_k,c_k<a\) and \(a_k<b\). We show that all these quantities vanish.

Suppose first that \(c_k>0\). The leaves \(\sfw_{3k+1},\sfw_{3k+2},\sfw_{3k+3}\) span a tripod of type
\[(a-c_k,\ b-c_k,\ c+c_k);\]
see \Cref{figure:triple intersection 1}. It has diameter \(b+c\), but its shortest leaf-to-leaf distance is \(a+b-2c_k<a+b\), contradicting the choice of \((a,b,c)\).
Thus \(c_k=0\) for every \(k\).

Next, suppose that \(a_k>0\). As in \Cref{figure:triple intersection 2}, the branch points of the tripods spanned by \((\sfw_{3k+2},\sfw_{3k+3},\sfw_{3k+4})\) and \((\sfw_{3k+3},\sfw_{3k+4},\sfw_{3k+5})\) lie at a distance \(a_k+b_{k+1}\) apart along \(\sfP_{3k+3}\).
Consequently, the leaves \(\sfw_{3k+2},\sfw_{3k+4},\sfw_{3k+5}\) span a tripod of type 
\[(b-a_k,\ c-a_k,\ c+a_k).\]
Its diameter is \(2c>b+c\), again contradicting the choice of \((a,b,c)\). Hence \(a_k=0\) for every \(k\).

Finally, suppose that \(b_k>0\). Since \(a_{k-1}=0\), the branch points of the tripods spanned by \((\sfw_{3k-1},\sfw_{3k},\sfw_{3k+1})\) and \((\sfw_{3k},\sfw_{3k+1},\sfw_{3k+2})\) lie at a distance \(b_k\) apart along \(\sfP_{3k}\). 
Thus the leaves \(\sfw_{3k-1},\sfw_{3k},\sfw_{3k+2}\) span a tripod of type
\[(a-b_k,\ c-b_k,\ c+b_k);\]
see \Cref{figure:triple intersection 0}. Its diameter is again \(2c>b+c\), a contradiction.
Therefore \(b_k=0\) for every \(k\).

The intersections of each \(\sfP_i\) with its two neighboring paths now meet only at a vertex of \(\sfT\). Hence \(\triangle_{i-1}\cap\triangle_i\cap\triangle_{i+1}=\varnothing\) for every \(i\), and \Cref{proposition:leaf sequence} shows that \(\cP(\sfw_\bullet)\) is not null-homotopic in \(\cI(UP_2(\graf))\).
If \(\bar\triangle_0=\bar\triangle_N\), the sequence \(\bar\triangle_\bullet\) would carry a loop whose image under \(\rho\) is homotopic to \(\cP(\sfw_\bullet)\). This contradicts the simple connectedness of \(\cI(\bar{UP_2(\graf)})\), proving the claim.
\end{proof}

\begin{figure}[ht]
\centering
\begin{subfigure}{.49\textwidth}
\[
\begin{tikzpicture}[baseline=-.5ex, scale=0.7, transform shape]
\draw[fill] (-4,0) circle (2pt) node[left] {\(\sfw_{3k+1}\)} -- node[midway,below] {\(\underbrace{\hphantom{\hspace{3cm}}}_{b-c_k}\)} (-1,0) circle (2pt) -- node[midway,below] {\(\underbrace{\hphantom{\hspace{1cm}}}_{c_k}\)} (0,0) circle (2pt) -- node[midway,below] {\(\underbrace{\hphantom{\hspace{1cm}}}_{a_{k}}\)} (1,0) circle (2pt) -- node[midway,below] {\(\underbrace{\hphantom{\hspace{4cm}}}_{c-a_{k}}\)} (5,0) circle (2pt) node[right] {\(\sfw_{3k+2}\)};
\draw[fill] (-1,0) -- node[midway, above, sloped] {\(\overbrace{\hphantom{\hspace{2cm}}}^{a-c_k}\)} ++(150:2) circle (2pt) node[left] {\(\sfw_{3k+3}\)};
\draw[fill] (1,0) --  node[midway,above, sloped] {\(\overbrace{\hphantom{\hspace{3cm}}}^{b-a_k}\)} ++(30:3) circle (2pt) node[right] {\(\sfw_{3k+4}\)};
\end{tikzpicture}
\]
\caption{When \(c_k>0\).}
\label{figure:triple intersection 1}
\end{subfigure}
\begin{subfigure}{.49\textwidth}
\[
\begin{tikzpicture}[baseline=-.5ex, scale=0.7, transform shape]
\draw[fill] (-5,0) circle (2pt) node[left] {\(\sfw_{3k+2}\)} -- node[midway,below] {\(\underbrace{\hphantom{\hspace{4cm}}}_{c-a_k}\)} (-1,0) circle (2pt) -- node[midway,below] {\(\underbrace{\hphantom{\hspace{1cm}}}_{a_k}\)} (0,0) circle (2pt) -- node[midway,below] {\(\underbrace{\hphantom{\hspace{1cm}}}_{b_{k+1}}\)} (1,0) circle (2pt) -- node[midway,below] {\(\underbrace{\hphantom{\hspace{2cm}}}_{a-b_{k+1}}\)} (3,0) circle (2pt) node[right] {\(\sfw_{3k+3}\)};
\draw[fill] (-1,0) -- node[midway, above, sloped] {\(\overbrace{\hphantom{\hspace{3cm}}}^{b-a_k}\)} ++(150:3) circle (2pt) node[left] {\(\sfw_{3k+4}\)};
\draw[fill] (1,0) --  node[midway,above, sloped] {\(\overbrace{\hphantom{\hspace{4cm}}}^{c-b_{k+1}}\)} ++(30:4) circle (2pt) node[right] {\(\sfw_{3k+5}\)};
\end{tikzpicture}
\]
\caption{When \(a_k>0\).}
\label{figure:triple intersection 2}
\end{subfigure}
\begin{subfigure}{.49\textwidth}
\centering
\[
\begin{tikzpicture}[baseline=-.5ex, scale=0.7, transform shape]
\draw[fill]
(-3,0) circle (2pt) node[left] {\(\sfw_{3k}\)} -- node[midway,below] {\(a-b_k\)} (-1,0) circle (2pt) -- node[midway,below] {\(b_k\)} (1,0) circle (2pt) -- node[midway,below] {\(b\)} (4,0) circle (2pt) node[right] {\(\sfw_{3k+1}\)};
\draw[fill] (-1,0) -- node[midway,above,sloped] {\(c-b_k\)} ++(150:2.5) circle (2pt) node[left] {\(\sfw_{3k+2}\)};
\draw[fill] (1,0) -- node[midway,above,sloped] {\(c\)} ++(30:2.5) circle (2pt) node[right] {\(\sfw_{3k-1}\)};
\end{tikzpicture}
\]
\caption{When \(b_k>0\), after \(a_{k-1}=0\) has been established.}
\label{figure:triple intersection 0}
\end{subfigure}
\caption{Path configurations used in the proof of \Cref{proposition:type (a_b_c)}. Labels indicate path lengths; segments of length zero are understood to be collapsed.}
\label{figure:triple intersection}
\end{figure}

\subsection{Quasi-isometry to RAAGs}
Following \Cref{Def:Salvetti complex}, we denote by \(S_\grafl\) the Salvetti complex associated to a simple graph \(\grafl\) and by \(X_\grafl\) its universal cover. If \(\grafl=\grafl_1\sqcup\dots\sqcup\grafl_n\sqcup\{\text{isolated vertices}\}\), then \(S_\grafl\) is the wedge sum of the subcomplexes \(S_{\grafl_i}\) and one circle for each isolated vertex of \(\grafl\), each locally convex in \(S_\grafl\); consequently, when \(\grafl\) is \(3\)-cycle-free, \(\cI(S_\grafl)\) is the disjoint union of the \(\cI(S_{\grafl_i})\). 

In this subsection, we study when \(\bbB_2(\graf)\), for \(\graf\in\grapegraph\), is quasi-isometric to a RAAG \(\bbA_\grafl\), equivalently, when \(\bar{UD_2(\graf)}\) is quasi-isometric to \(X_\grafl\). In \cite{Oh22}, two families of bunches of grapes were constructed: one whose \(2\)-braid groups are quasi-isometric to RAAGs and the other whose \(2\)-braid groups are not. 

\begin{proposition}[{\cite[Proposition~5.16, Corollary~5.18]{Oh22}}]
The following statements hold:
\begin{enumerate}
\item Let \(\graf_1=(\sfS_n,\loops_1)\) be a bunch of grapes over a star graph \(\sfS_n\) for some \(n\ge 3\) such that \(\loops_1(\sfv)=1\) for every leaf \(\sfv\) and \(\loops_1\) vanishes at the central vertex (see \Cref{fig:qitoRAAG}).
Then \(\bbB_2(\graf_1)\) is quasi-isometric to a RAAG \(\bbA_\grafl\ast\bbZ\) for some tree \(\grafl\) of diameter at least \(3\).
\item Let \(\graf_2=(\tilde\sfD_n,\loops_2)\) be a bunch of grapes over an affine Dynkin diagram \(\tilde\sfD_n\) for some \(n\ge 5\) such that \(\loops_2(\sfv)=1\) for every leaf and bivalent vertex \(\sfv\) and \(\loops_2\) vanishes at the two trivalent vertices (see \Cref{fig:notqitoRAAG}).
Then \(\bbB_2(\graf_2)\) is not quasi-isometric to any RAAG.
\end{enumerate}
\end{proposition}

\begin{figure}[ht]
\centering
\begin{subfigure}{.45\textwidth}
\[
\begin{tikzpicture}[baseline=-.5ex, scale=0.9,transform shape]
\draw[fill] (0,0) circle (2pt);
\foreach \i in {0,1,2,3,4,5} {
\draw[fill] (0,0) -- ({\i*60}:1) circle (2pt) node (A) {};
\grape[{\i*60}]{A};
}
\end{tikzpicture}
\]
\caption{A bunch of grapes over a star graph \(\sfS_6\)}
\label{fig:qitoRAAG}
\end{subfigure}
\begin{subfigure}{.5\textwidth}
\[
\begin{tikzpicture}[baseline=-.5ex,scale=0.9,transform shape]
\draw[fill] (0,0) circle (2pt) -- (120:1) circle (2pt) node (A) {};
\draw[fill] (0,0) -- (240:1) circle (2pt) node (B) {};
\draw[fill] (0,0) -- (1,0) circle (2pt) node (C) {} -- (2,0) circle (2pt) node (D) {} -- (3,0) circle (2pt) node (E) {} -- (4,0) circle (2pt);
\draw[fill] (4,0) -- +(60:1) circle (2pt) node (F) {} +(0,0) -- +(-60:1) circle (2pt) node (G) {};
\grape[120]{A};
\grape[240]{B};
\grape[90]{C};
\grape[90]{D};
\grape[90]{E};
\grape[60]{F};
\grape[-60]{G};
\end{tikzpicture}
\]
\caption{A bunch of grapes over an affine Dynkin diagram \(\tilde{\mathsf{D}}_8\)}
\label{fig:notqitoRAAG}
\end{subfigure}
\caption{Two bunches of grapes: one whose \(2\)-braid group is quasi-isometric to a RAAG, one not.}
\end{figure}

We next provide one sufficient condition and two necessary conditions for \(\bbB_2(\graf)\) to be quasi-isometric to a RAAG. We begin by reducing the problem to the quasi-isometry type of \(\pi_1(UP_2(\graf))\).

\begin{proposition}\label{Prop:StartingPoint}
Let \(\graf=(\sfT,\loops)\in\grapegraph\). Then \(\bbB_2(\graf)\) is quasi-isometric to a RAAG if and only if either \(UP_2(\graf)=\varnothing\) or \(\pi_1(UP_2(\graf))\) is quasi-isometric to a RAAG.
Moreover, if \(\pi_1(UP_2(\graf))\) is quasi-isometric to a RAAG \(\bbA_\grafl\), then \(\grafl\) must be a connected, \(3\)-cycle-free simple graph with at least two vertices.
\end{proposition}
\begin{proof}
If \(\graf\) is not large, or equivalently \(UP_2(\graf)=\varnothing\), then \(\bbB_2(\graf)\) is quasi-isometric to a free group (and hence to a RAAG) by \Cref{Lem:WhyNormal}. We may therefore assume that \(\graf\) is large, in which case \Cref{corollary:free factor,Prop:free factor} yield the isomorphisms 
\begin{equation}\label{Eq:isomorphisms}
\bbB_2(\graf)\,\cong\,\bbB_2(\graf')\ast\bbF\,\cong\,\pi_1(UP_2(\graf'))\ast\bbF',    
\end{equation}
where \(\graf'\) is the minimal simplicial model of the maximal leafless subgraph of \(\graf\), and \(\bbF,\bbF'\) are free groups (with \(\bbF\) possibly trivial only when \(\graf=\graf'\)). Since removing leaves does not change \(UP_2\), \Cref{theorem:subdivision and UP2} gives \(\pi_1(UP_2(\graf))\cong\pi_1(UP_2(\graf'))\). As \(\graf'\) is normal, \(\pi_1(UP_2(\graf))\) is one-ended by \Cref{theorem:structureofI}.

If \(\pi_1(UP_2(\graf))\) is quasi-isometric to a RAAG, then \Cref{PW} implies that \(\bbB_2(\graf)\) is quasi-isometric to a RAAG. 
Since \(\pi_1(UP_2(\graf))\) is one-ended, the defining graph of any RAAG \(\bbA_\grafl\) quasi-isometric to \(\pi_1(UP_2(\graf))\) is connected and has at least two vertices by \Cref{Prop:BasiconRAAG}. 
By the dimension bound for \(\CAT(0)\) cube complexes \cite{Wri12}, the universal cover of \(UP_2(\graf)\) has asymptotic dimension at most \(2\). Thus \(\grafl\) must be \(3\)-cycle-free: otherwise the universal cover of \(S_\grafl\) would contain an isometrically embedded \(3\)-flat.

Conversely, suppose that \(\bbB_2(\graf)\) is quasi-isometric to a RAAG \(\bbA_\grafl\). By \Cref{Eq:isomorphisms}, its Grushko decomposition has a unique one-ended factor, namely \(\pi_1(UP_2(\graf))\). 
Decomposing \(\bbA_\grafl\) according to the connected components of \(\grafl\) and applying \cite[Theorem~0.4]{PW02}, we conclude that \(\pi_1(UP_2(\graf))\) is quasi-isometric to \(\bbA_{\grafl_0}\) for some connected component \(\grafl_0\) with at least two vertices.
\end{proof}

\begin{theorem}[Sufficient condition]\label{theorem:isomorphictoRAAG}
Let \(\graf=(\sfT,\loops)\) be a bunch of grapes such that there exists a path substem \(\sfP\subseteq\sfT\) containing all vertices \(\sfv\) with \(\loops(\sfv)\ge 1\). Then \(\bbB_2(\graf)\) is isomorphic to a RAAG.
\end{theorem}
\begin{proof}
If \(\graf\) is small, by \Cref{scrg}, \(\bbB_2(\graf)\) is a free group and therefore a RAAG.
Otherwise, by \Cref{Eq:isomorphisms}, it suffices to consider the case \(\sfT=\sfP=[\sfv_0,\dots,\sfv_n]\) with \(n\ge 1\) and \(\loops(\sfv_i)\ge1\) for all \(i\).

For each \(0\le i\le j\le n\), we denote the path substem \([\sfv_i,\dots,\sfv_j]\) simply by \([i,j]\). By \Cref{Corollary:TwigCorrespondence}, there are \(n\) maximal product subcomplexes of \(UP_2(\graf)\): \(M_i=M(\sft_i)=\graf_{[0,i]}\itimes \graf_{[i+1,n]}\) for \(\sft_i=[\sfv_i,\sfv_{i+1}]\) and \(0\le i<n\).
Let
\[A_i:=\{a_{p,q}:0\le p\le i,\, 1\le q\le \loops(\sfv_p)\}\quad\text{ and }\quad
B_i:=\{b_{p',q'}:i+1\le p'\le n,\, 1\le q'\le \loops(\sfv_{p'})\}.\]
Then \(A_0\subsetneq A_1\subsetneq\cdots\subsetneq A_{n-1}\), \(B_0\supsetneq B_1\supsetneq\cdots\supsetneq B_{n-1}\), and 
\[\pi_1(M_i)=\langle A_i\cup B_i:ab=ba\quad\forall a\in A_i,\, b\in B_i\rangle\isom \bbA_{\grafl_i},\]
where $\grafl_i$ is the complete bipartite graph \(\sfK_{\loops([0,i]),\loops([i+1,n])}\) with vertex set $A_i\cup B_i$.
We regard $\grafl_i$ as a subgraph of \(\grafl=\sfK_{\loops([0,n-1]),\loops([1,n])}\) whose vertex set is $A_{n-1}\cup B_0$.

Now set \(U_i=M_0\cup\cdots\cup M_i\). For \(1\le i<n\), we have
\[U_{i-1}\cap M_i=M_{i-1}\cap M_i=\graf_{[0,i-1]}\itimes\graf_{[i+1,n]}.\]
Its fundamental group is generated by \(A_{i-1}\cup B_i\), and the inclusions identify the corresponding generators. Since \(U_0=M_0\), van Kampen's theorem gives, inductively,
\[\pi_1(U_i)=\left\langle A_i\cup B_0 : ab=ba\quad \forall a\in A_j\subseteq A_i,\, b\in B_j\subseteq B_0,\, 0\le j\le i \right\rangle \isom \bbA_{\grafl_0\cup\dots\cup\grafl_i}\]
where \(\grafl_0\cup\dots\cup\grafl_i\) is a subgraph of $\grafl$.
Since $U_{n-1}=UP_2(\graf)$, the proof is complete.
\end{proof}

We now turn to obstructions to quasi-isometry with RAAGs. Let \(\grafl\) be a finite connected \(3\)-cycle-free simple graph with at least two vertices.
The following local copies of \(\cI(S_\grafl)\) inside \(\cI(X_\grafl)\) allow closed sequences downstairs to be realized upstairs.
For the configurations considered below, this will contradict \Cref{proposition:edge-simplex type,proposition:type (a_b_c)}.

For each vertex \(\bar x\in X_\grafl\), let \(\cI_{\bar x}(X_\grafl)\) denote the subcomplex of \(\cI(X_\grafl)\) spanned by vertices corresponding to maximal product subcomplexes of \(X_\grafl\) containing \(\bar x\); we call it the \emph{local intersection complex} at \(\bar x\). 

\begin{lemma}\label{lem:CopiesofR}
The local intersection complex \(\cI_{\bar x}(X_\grafl)\) is isomorphic to \(\cI(S_\grafl)\), and \(\cI(X_\grafl)\) is the union of these subcomplexes as \(\bar x\) ranges over \(V(X_\grafl)\).
\end{lemma}
\begin{proof}
The local intersection complexes \(\cI_{\bar x}(X_\grafl)\), with \(\bar x\in V(X_\grafl)\), cover \(\cI(X_\grafl)\) by definition.
Since \(S_\grafl\) has one vertex, the restriction of the canonical quotient map \(\rho_{S_\grafl}:\cI(X_\grafl)\to\cI(S_\grafl)\) to \(\cI_{\bar x}(X_\grafl)\) is an isomorphism.
\end{proof}

\begin{theorem}[Necessary condition I]\label{theorem:NotQItoRAAG}
Let \(\graf=(\sfT,\loops)\in\grapegraph_\norm\). If an affine Dynkin diagram \(\tilde\sfD_n\) (\(n\ge5\)) leaf-respectingly embeds into \(\sfT\), then \(\bbB_2(\graf)\) is not quasi-isometric to any RAAG.
\end{theorem}
\begin{proof}
Assume, for contradiction, that \(\bbB_2(\graf)\) is quasi-isometric to a RAAG. By \Cref{Prop:StartingPoint}, there exists a quasi-isometry \(\phi:\bar{UP_2(\graf)}\to X_\grafl\) for a connected, \(3\)-cycle-free simple graph \(\grafl\) with at least two vertices. By \Cref{theorem:IsobetInt}, there is an induced isomorphism \(\cI(\phi):\cI(\bar{UP_2(\graf)})\to\cI(X_\grafl)\).

Regard \(\tilde\sfD_n\) as a leaf-respecting embedded subgraph of \(\sfT\), and label its four leaves by \(\sfw_0,\sfw_1,\sfw_2,\sfw_3\) such that the paths \(\sfP(\sfw_0,\sfw_1)\) and \(\sfP(\sfw_2,\sfw_3)\) each have length \(2\) (recall that \(\sfP(\sfw_i,\sfw_j)\) is defined as the shortest path between \(\sfw_i\) and \(\sfw_j\)).
For each \(i\), set \(\triangle_i=\triangle_{\sfP(\sfw_i,\sfw_{i+1})}\subseteq\cI(UP_2(\graf))\) with indices taken modulo \(4\).
Consider the leaf sequence \(\sfw_\bullet=(\sfw_0,\sfw_1,\sfw_2,\sfw_3,\sfw_4=\sfw_0)\). By \Cref{proposition:leaf sequence}, the associated loop \(\cP(\sfw_\bullet)\) is not null-homotopic in \(\cI(UP_2(\graf))\).
As explained after \Cref{Rmk:TypeofSequence}, this sequence lifts to a bi-infinite sequence of maximal simplices in \(\cI(\bar{UP_2(\graf)})\)
\[\bar\triangle_\bullet=(\dots,\bar\triangle_{-1},\,\bar\triangle_0,\,\bar\triangle_1,\dots)\] such that \(\rho(\bar\triangle_{4n+i})=\triangle_i\) for all \(n\in\bbZ\), \(0\le i\le 3\), and \(\bar\triangle_\bullet\) is of edge-simplex type.

Applying \(\cI(\phi)\) yields a bi-infinite sequence of maximal simplices in \(\cI(X_\grafl)\)
\[\bar\triangle_\bullet'=(\dots,\bar\triangle_{-1}',\,\bar\triangle_0',\,\bar\triangle_1',\dots)=\cI(\phi)(\bar\triangle_\bullet),\]
and composing with \(\rho_{S_\grafl}:\cI(X_\grafl)\to\cI(S_\grafl)\) gives a bi-infinite sequence of maximal simplices in \(\cI(S_\grafl)\) 
\[\triangle_\bullet'=(\dots,\triangle_{-1}',\,\triangle_0',\,\triangle_1',\dots).\]
Since \(\cI(\phi)\) is an isomorphism and \(\rho_{S_{\grafl}}\) is a combinatorial map, both sequences \(\bar\triangle'_\bullet\) and \(\triangle'_\bullet\) are of edge-simplex type by \Cref{Rmk:TypeofSequence}.
Since \(\cI(S_\grafl)\) is finite by \Cref{Thm:TPBCM}, there exist integers \(n<m\) such that \((\triangle'_{4n},\triangle'_{4n+1})=(\triangle'_{4m},\triangle'_{4m+1})\).
Thus \((\triangle'_{4n},\dots,\triangle'_{4m}=\triangle'_{4n})\) is a closed sequence of edge-simplex type, with the defining conditions understood cyclically.
Fix a vertex \(\bar x\in X_\grafl\). By \Cref{lem:CopiesofR}, this sequence corresponds to a closed sequence \((\bar\triangle''_{4n},\dots,\bar\triangle''_{4m})\) of maximal simplices in \(\cI_{\bar x}(X_\grafl)\), still of edge-simplex type.
In particular, \(\bar\triangle''_{4n}=\bar\triangle''_{4m}\).


Applying \(\cI(\phi)^{-1}\) again yields a closed sequence of maximal simplices in \(\cI(\bar{UP_2(\graf)})\)
\[\cI(\phi)^{-1}(\bar\triangle_{4n}'',\dots,\bar\triangle_{4m}'')=\left(\cI(\phi)^{-1}(\bar\triangle_{4n}''),\dots,\cI(\phi)^{-1}(\bar\triangle_{4m}'')\right),\]
still of edge-simplex type. However, this is impossible by \Cref{proposition:edge-simplex type}. Therefore, no such quasi-isometry can exist, and \(\bbB_2(\graf)\) is not quasi-isometric to any RAAG.
\end{proof}

\begin{theorem}[Necessary condition II]\label{theorem:necessary condition 2}
Let \(\graf=(\sfT,\loops)\in\grapegraph_\norm\) be such that a tripod \(\sfT_{a,b,c}\) for some \(1\le a<b<c\) leaf-respectingly embeds into \(\sfT\).
Then \(\bbB_2(\graf)\) is not quasi-isometric to any RAAG.
\end{theorem}
\begin{proof}
Assume, for contradiction, that \(\bbB_2(\graf)\) is quasi-isometric to a RAAG. As in the beginning of the proof of \Cref{theorem:NotQItoRAAG}, there is a quasi-isometry \(\phi:\bar{UP_2(\graf)}\to X_\grafl\), where \(\grafl\) is connected and \(3\)-cycle-free and has at least two vertices, and \(\phi\) induces an isomorphism \(\cI(\phi):\cI(\bar{UP_2(\graf)})\to\cI(X_\grafl)\).

Choose \((a,b,c)\) minimal in \(\tripod(\sfT)\) and let \(\sfw_0,\sfw_1,\sfw_2\) be the leaves of a corresponding leaf-respecting embedded tripod. 
After relabeling, the maximal simplices \(\triangle_i=\triangle_{\sfP(\sfw_i,\sfw_{i+1})}\), with indices modulo \(3\), have dimensions \(a+b-1,b+c-1,c+a-1\), respectively. As explained after \Cref{Rmk:TypeofSequence}, the leaf sequence \((\sfw_0,\sfw_1,\sfw_2,\sfw_0)\) yields a bi-infinite sequence \(\bar\triangle_\bullet\) of type \((a,b,c)\) with \(\rho(\bar\triangle_{3j+i})=\triangle_i\) for all \(j\in\bbZ\) and \(0\le i\le2\).

Set \(\bar\triangle'_i=\cI(\phi)(\bar\triangle_i)\) and \(\triangle'_i=\rho_{S_\grafl}(\bar\triangle'_i)\). These sequences are of type \((a,b,c)\) by \Cref{Rmk:TypeofSequence}.
Since \(\cI(S_\grafl)\) is finite, there exist integers \(n<m\) such that
\[(\triangle'_{3n},\triangle'_{3n+1})=(\triangle'_{3m},\triangle'_{3m+1}).\]
Thus \((\triangle'_{3n},\dots,\triangle'_{3m}=\triangle'_{3n})\) is a closed sequence of type \((a,b,c)\), with the defining conditions understood cyclically.

Fix a vertex \(\bar x\in X_\grafl\). By \Cref{lem:CopiesofR}, this sequence has a corresponding closed sequence \((\bar\triangle''_{3n},\dots,\bar\triangle''_{3m})\) in \(\cI_{\bar x}(X_\grafl)\), still of type \((a,b,c)\).
Applying \(\cI(\phi)^{-1}\) gives a closed sequence of maximal simplices in \(\cI(\bar{UP_2(\graf)})\) of type \((a,b,c)\). Its projection under \(\rho\) is also a closed sequence of this type, by \Cref{Rmk:TypeofSequence}.
This contradicts \Cref{proposition:type (a_b_c)}.
\end{proof}

\begin{remark}
The conclusions of \Cref{theorem:NotQItoRAAG,theorem:necessary condition 2} also hold for a possibly non-normal bunch of grapes whenever every vertex of the indicated embedded subtree carries a grape. Indeed, that subtree survives unchanged upon passing to the minimal simplicial model of the maximal leafless subgraph. The claim then follows from \Cref{Eq:isomorphisms,PW}.
\end{remark}

\begin{question}
Let $\graf$ be a bunch of grapes. Is the property that $\bbB_2(\graf)$ is quasi-isometric to a RAAG decidable from the combinatorics of $\graf$?
\end{question}

In a forthcoming paper, we will propose generalizations of the sufficient and necessary conditions established above.

\subsection{Relative hyperbolicity and braid groups of subgraphs}
Genevois characterized hyperbolicity and toral relative hyperbolicity for graph braid groups \cite{Gen21GBG}. Beyond this toral setting, however, the general problem of determining when a graph braid group is relatively hyperbolic remains largely open. Aside from a specific class of graphs whose \(2\)-braid groups are hyperbolic relative to \(2\)-braid groups on certain subgraphs (see \cite[Theorem~4.30]{Gen21GBG} and the examples following it), little is currently known.

A different phenomenon was later exhibited by Berlyne \cite[Theorem~4.4]{Ber23}, who constructed a graph \(\graf\) such that \(\bbB_2(\graf)\) is hyperbolic relative to a thick proper subgroup which is not contained in any graph braid group of the form \(\bbB_k(\grafl)\) with \(k\le 2\) and \(\grafl\subsetneq\graf\). 
His example is obtained from \Cref{figure:non quasi-isometry} by removing the grape attached to \(\sfv_3\) and the edge \(\sfe_3\). As observed in \Cref{example:not quasiisometry}, this graph belongs to \(\graphs_{(2)}\).

We begin with the following immediate observation.

\begin{proposition}\label{Prop:Relativehyerbolicity}
Let \(\graf\) be a graph in \(\graphs_{(2)}\). Then the graph \(2\)-braid group \(\bbB_2(\graf)\) is hyperbolic relative to \(\pi_1(UP_2(\graf))\).

If, in addition, \(UD_2(\graf)\) satisfies the standard intersection property, then \(\bbB_2(\graf)\) is hyperbolic relative to the thick subgroup \(\pi_1(UP_2(\graf))\).
\end{proposition}
\begin{proof}
The first statement directly follows from the definition of \(\graphs_{(2)}\) together with the definition of relative hyperbolicity. The second statement follows from \Cref{Lem:UP_2One-ended}.
\end{proof}

Two possible degeneracies of the preceding statement are worth keeping in mind. If \(UP_2(\graf)=UD_2(\graf)\), equivalently \(\graf\in\graphs_{(5)}\), the relative hyperbolicity described above is trivial.
On the other hand, even when \(\bbB_2(\graf)\cong\pi_1(UP_2(\graf))\ast\bbF_N\) with \(N>0\), the peripheral group \(\pi_1(UP_2(\graf))\) may itself admit a nontrivial free product decomposition. In the latter situation, it is not known in general whether it can be isomorphic to \(\bbB_2(\grafl)\) for another graph \(\grafl\).

For large bunches of grapes, the distinguished peripheral subgroup is proper and one-ended.
Together with the fact that connected subgraphs of bunches of grapes are again bunches of grapes, this yields the following result.

\begin{theorem}\label{Corollary:Thickness}
Let \(\graf\in\grapegraph_{\larg}\). Then the graph \(2\)-braid group \(\bbB_2(\graf)\) is hyperbolic relative to a thick, proper subgroup which is not isomorphic to any graph braid group of the form \(\bbB_k(\grafl)\), where \(k\le 2\) and \(\grafl\subseteq\graf\).
\end{theorem}
\begin{proof}
By \Cref{Eq:isomorphisms,theorem:structureofI}, it suffices to show that \(\pi_1(UP_2(\graf))\) is not isomorphic to any graph braid group of the form \(\bbB_k(\grafl)\) with \(k\le 2\) and \(\grafl\subseteq\graf\).
Since \(\grafl\) is a subgraph of a bunch of grapes, it is itself a bunch of grapes. 
By \Cref{theorem:structureofI}, \(\pi_1(UP_2(\graf))\) is one-ended, so \(k\neq 1\). Moreover, \Cref{Lem:WhyNormal} implies that \(\grafl\) must be large; hence, by \Cref{Eq:isomorphisms,Prop:free factor}, \(\bbB_2(\grafl)\) is not one-ended. Therefore no such subgraph \(\grafl\) exists, and the conclusion follows.
\end{proof}

Berlyne's example \(\graf'\) is not a bunch of grapes, but \(UD_2(\graf')\) satisfies the standard intersection property. This suggests that the phenomenon above may hold more generally and leads to the following conjecture.

\begin{conjecture}
Let \(\graf\in\graphs_{(2)}\), and suppose that \(UD_2(\graf)\) satisfies the standard intersection property. Then \(\bbB_2(\graf)\) is hyperbolic relative to a thick, proper subgroup which is not isomorphic to \(\bbB_k(\grafl)\) for any \(k\ge1\) and any subgraph \(\grafl\subseteq\graf\).
\end{conjecture}

\bibliographystyle{alpha} 
\bibliography{Ref.bib}

@article {Lovasz,
    AUTHOR = {Lov\'{a}sz, L\'{a}szl\'{o}},
     TITLE = {{On graphs not containing independent circuits}},
   JOURNAL = {Mat. Lapok},
  FJOURNAL = {Matematikai Lapok. Bolyai J\'{a}nos Matematikai T\'{a}rsulat},
    VOLUME = {16},
      YEAR = {1965},
     PAGES = {289--299},
      ISSN = {0025-519X},
   MRCLASS = {05.40},
  MRNUMBER = {211902},
MRREVIEWER = {P.\ Ungar},
}

@phdthesis{Abrams00,
  title  = {{Configuration spaces and braid groups of graphs}},
  author = {Abrams, Aaron D.},
  school = {University of California, Berkeley},
  year   = {2000}
}

@misc{AOP26,
      title={{Hierarchical geometry and right-angled Artin groups in graph braid groups}}, 
      author={An, Byung Hee and Oh, Sangrok and Park, Jihoon},
      year={2026},
      eprint={2609.06589},
      archivePrefix={arXiv},
      primaryClass={math.GT},
      note={Preprint, arXiv:2609.06589},
      url={https://arxiv.org/abs/2609.06589}, 
}

@article {AK2022,
    AUTHOR = {An, Byung Hee and Knudsen, Ben},
     TITLE = {On the second homology of planar graph braid groups},
   JOURNAL = {J. Topol.},
  FJOURNAL = {Journal of Topology},
    VOLUME = {15},
      YEAR = {2022},
    NUMBER = {2},
     PAGES = {666--691},
      ISSN = {1753-8416},
   MRCLASS = {20F36 (55R80)},
  MRNUMBER = {4441601},
       DOI = {10.1112/topo.12228},
       URL = {https://doi.org/10.1112/topo.12228},
}

@article{LSUTeam25,
author = {Appiah, Benjamin and Dani, Pallavi and Ge, Wayne and Hudson, Christopher and Jain, Saumya and Lemoine, Matthew and Murphy, Jake and Murray, Justin and Pandikkadan, Adithyan and Schreve, Kevin and Vo, Huong},
title = {The algebraic structure of hyperbolic graph braid groups},
journal = {International Journal of Algebra and Computation},
volume = {35},
number = {03},
pages = {329-342},
year = {2025},
doi = {10.1142/S0218196724500528}
}

@article{BDM,
author  = {Behrstock, Jason and Dru{\c{t}}u, Cornelia and Mosher, Lee},
title   = {Thick metric spaces, relative hyperbolicity, and quasi-isometric rigidity},
journal = {Mathematische Annalen},
volume  = {344},
number  = {3},
pages   = {543--595},
year    = {2009},
doi     = {10.1007/s00208-008-0317-1}
}

@article{Ber23,
  author = {Berlyne, Daniel},
  title = {Graph of groups decompositions of graph braid groups},
  journal = {International Journal of Algebra and Computation},
  fjournal = {},
  volume = {33},
  number = {8},
  year = {2023},
  pages = {1531--1569},
  doi = {10.1142/S0218196723500583},
  note = {}
}

@article{BKRR23,
  author  = {Bauer, Ulrich and Kerber, Michael and
             Roll, Fabian and Rolle, Alexander},
  title   = {A unified view on the functorial nerve theorem
             and its variations},
  journal = {Expositiones Mathematicae},
  volume  = {41},
  number  = {4},
  pages   = {125503},
  year    = {2023},
  doi     = {10.1016/j.exmath.2023.04.005}
}

@article{CD14,
author = {Connolly, Francis and Doig, Margaret},
year = {2014},
title = {{On braid groups and right-angled Artin groups}},
volume = {172},
pages = {179-190},
journal = {Geometriae Dedicata},
doi = {10.1007/s10711-013-9914-6}
}

@article{BD14,
author = {Jason Behrstock and Cornelia Druţu},
title = {{Divergence, thick groups, and short conjugators}},
volume = {58},
journal = {Illinois Journal of Mathematics},
number = {4},
publisher = {Duke University Press},
pages = {939 -- 980},
year = {2014},
doi = {10.1215/ijm/1446819294},
URL = {https://doi.org/10.1215/ijm/1446819294}
}

@book{BH,
author = {Bridson, Martin and Haefliger, Andr\'e},
year = {1999},
title = {Metric Spaces of Non-Positive Curvature},
isbn = {978-3-642-08399-0},
publisher = {Grundlehren der mathematischen Wissenschaften 319, Springer},
doi = {10.1007/978-3-662-12494-9}
}

@article{BS08,
author = {Buyalo, Sergei and Schroeder, Viktor},
journal = {St. Petersburg Math. J.},
volume = {19},
number ={1},
pages = {67--76},
title = {{Hyperbolic dimension of metric spaces}},
year = {2008}
}

@article{BS84,
author = {Baumslag, Gilbert and Roseblade, James E.},
title = {Subgroups of Direct Products of Free Groups},
journal = {Journal of the London Mathematical Society},
volume = {s2-30},
number = {1},
pages = {44--52},
doi = {https://doi.org/10.1112/jlms/s2-30.1.44},
url = {https://londmathsoc.onlinelibrary.wiley.com/doi/abs/10.1112/jlms/s2-30.1.44},
eprint = {https://londmathsoc.onlinelibrary.wiley.com/doi/pdf/10.1112/jlms/s2-30.1.44},
year = {1984}
}

@article{CH,
author = {Charney, Ruth},
title = {{An introduction to right-angled Artin groups}},
volume = {125},
journal = {Geometriae Dedicata},
number = {1},
pages = {141 - 158},
year = {2007},
doi = {10.1007/s10711-007-9148-6},
URL = {https://doi.org/10.1007/s10711-007-9148-6}
}

@article{CW,
author = {Crisp, John and Wiest, Bert},
year = {2004},
pages = {439 - 472},
title = {{Embeddings of graph braid and surface groups in right-angled Artin groups and braid groups}},
volume = {4},
journal = {Algebraic \& Geometric Topology},
doi = {10.2140/agt.2004.4.439}
}

@book{DK18,
  author       = {Cornelia Dru{\c{t}}u and Michael Kapovich},
  title        = {Geometric Group Theory},
  series       = {American Mathematical Society Colloquium Publications},
  volume       = {63},
  publisher    = {American Mathematical Society},
  address      = {Providence, RI},
  year         = {2018},
  note         = {With an appendix by Bogdan Nica},
}

@article{FS05,
author = {Farley, Daniel and Sabalka, Lucas},
year = {2005},
pages = {1075-1109},
title = {{Discrete Morse theory and graph braid groups}},
volume = {5},
number = {3},
journal = {Algebraic \& Geometric Topology},
doi = {10.2140/agt.2005.5.1075}
}

@article{FS08,
author = {Farley, Daniel and Sabalka, Lucas},
year = {2008},
pages = {53-71},
title = {On the cohomology rings of tree braid groups},
volume = {212},
number = {1},
journal = {Journal of Pure and Applied Algebra},
doi = {10.1016/j.jpaa.2007.04.011}
}

@article{FS12,
url = {https://doi.org/10.1515/form.2011.086},
title = {Presentations of graph braid groups},
author = {Daniel Farley and Lucas Sabalka},
pages = {827--859},
volume = {24},
number = {4},
journal = {Forum Mathematicum},
doi = {doi:10.1515/form.2011.086},
year = {2012},
lastchecked = {2022-10-11}
}

@article{Gen21,
author = {Genevois, Anthony},
title = {Algebraic characterisation of relatively hyperbolic special groups},
journal = {Israel Journal of Mathematics},
volume = {241},
number = {1},
pages = {301-341},
year = {2021},
doi = {10.1007/s11856-021-2097-1}
}

@article{Gen21GBG,
author = {Genevois, Anthony},
title = {Negative curvature in graph braid groups},
journal = {International Journal of Algebra and Computation},
volume = {31},
number = {01},
pages = {81-116},
year = {2021},
doi = {10.1142/S0218196721500053},
URL = {https://doi.org/10.1142/S0218196721500053},
eprint = {https://doi.org/10.1142/S0218196721500053}
}

@incollection{Ghrist01,
  author    = {Ghrist, Robert},
  title     = {Configuration spaces and braid groups on graphs in robotics},
  booktitle = {Knots, braids, and mapping class groups---papers dedicated to Joan S. Birman},
  series    = {AMS/IP Studies in Advanced Mathematics},
  volume    = {24},
  pages     = {29--40},
  publisher = {American Mathematical Society},
  address   = {Providence, RI},
  year      = {2001},
  note      = {MR1873106}
}

@incollection{Grom,
   author={Misha Gromov},
   editor = {Gersten S.M.},
   year={1987},
   title={Hyperbolic Groups},
   booktitle={Essays in Group Theory},
   isbn= {978-1-4613-9586-7},
   publisher={Mathematical Sciences Research Institute Publications, Springer},
   doi={10.1007/978-1-4613-9586-7_3},
   url={https://doi.org/10.1007/978-1-4613-9586-7_3}
}

@article{Hua17QuasiFlat,
author = {Huang, Jingyin},
year = {2017},
pages = {2281-2352},
title = {Top dimensional quasiflats in {CAT(0)} cube complexes},
volume = {21},
number = {4},
journal = {Geometry \& Topology},
doi = {10.2140/gt.2017.21.2281}
}

@article{HW08,
author = {Fr\'ed\'eric Haglund and Daniel T. Wise},
year = {2008},
pages = {1551-1620},
title = {Special Cube Complexes},
volume = {17},
number = {5},
journal = {Geometric and Functional Analysis},
doi = {10.1007/s00039-007-0629-4}
}

@article{KKP12,
author = {Kim, Jee Hyoun and Ko, Ki Hyoung and Park, Hyo Won},
year = {2012},
pages = {309-360},
title = {Graph braid groups and right-angled {A}rtin groups},
volume = {364},
number = {1},
journal = {Transactions of the American Mathematical Society},
doi = {10.1090/S0002-9947-2011-05399-7}
}

@article{KLP16,
author = {Ko, Ki Hyoung and La, Joon Hyun and Park, Hyo Won},
year = {2016},
pages = {133-153},
title = {{Graph 4-braid groups and Massey products}},
volume = {197},
number ={1},
journal = {Topology and its Applications},
doi = {10.1016/j.topol.2015.10.010}
}

@article{KP12,
author = {Ko, Ki Hyoung and Park, Hyo Won},
year = {2012},
pages = {915-963},
title = {Characteristics of Graph Braid Groups},
volume = {48},
number = {4},
journal = {Discrete and Computational Geometry},
doi = {10.1007/s00454-012-9459-8}
}

@article{Lea,
author = {Leary, Ian J.},
title = {{A metric Kan–Thurston theorem}},
journal = {Journal of Topology},
volume = {6},
number = {1},
pages = {251-284},
doi = {https://doi.org/10.1112/jtopol/jts035},
url = {https://londmathsoc.onlinelibrary.wiley.com/doi/abs/10.1112/jtopol/jts035},
eprint = {https://londmathsoc.onlinelibrary.wiley.com/doi/pdf/10.1112/jtopol/jts035},
year = {2013}
}

@article{Oh22,
title = {Quasi-isometry invariants of weakly special square complexes},
journal = {Topology and its Applications},
volume = {307},
pages = {107945},
year = {2022},
issn = {0166-8641},
doi = {https://doi.org/10.1016/j.topol.2021.107945},
url = {https://www.sciencedirect.com/science/article/pii/S016686412100362X},
author = {Sangrok Oh},
}

@unpublished{OhProducts,
  author = {Oh, Sangrok},
  title  = {Product subcomplexes and intersection complexes
            of weakly special square complexes},
  year   = {2026},
  note   = {\href{https://arxiv.org/abs/2609.25140}{arXiv:2609.25140}}
}

@article{PW02,
author = {Papasoglu, Panos and Whyte, Kevin},
year = {2002},
pages = {133--144},
volume= {77},
number={1},
title = {Quasi-isometries between groups with infinitely many ends},
journal={Commentarii Mathematici Helvetici},
doi = {10.1007/s00014-002-8334-2}
}

@article{PS12,
title = {Abrams's stable equivalence for graph braid groups},
journal = {Topology and its Applications},
volume = {178},
pages = {136-145},
year = {2014},
issn = {0166-8641},
doi = {https://doi.org/10.1016/j.topol.2014.09.009},
url = {https://www.sciencedirect.com/science/article/pii/S0166864114003836},
author = {Paul Prue and Travis Scrimshaw},
}

@article{Sab07,
title = {Embedding right-angled {A}rtin groups into graph braid groups},
author = {Sabalka, Lucas},
year = {2007},
pages = {191 - 198},
volume = {124},
number = {1},
journal = {Geometriae Dedicata},
doi = {10.1007/s10711-006-9101-0}
}

@book{Wis12,
  author    = {Daniel T. Wise},
  title     = {From Riches to RAAGs: 3-Manifolds, Right-Angled Artin Groups, and Cubical Geometry},
  series    = {CBMS Regional Conference Series in Mathematics},
  volume    = {117},
  publisher = {American Mathematical Society},
  address   = {Providence, RI},
  year      = {2012},
  isbn      = {978-0-8218-5367-1}
}

@article{Wri12,
  author  = {Wright, Nick},
  title   = {Finite asymptotic dimension for {CAT(0)} cube complexes},
  journal = {Geometry \& Topology},
  volume  = {16},
  number  = {1},
  pages   = {527--554},
  year    = {2012},
  doi     = {10.2140/gt.2012.16.527},
  eprint  = {1004.4172},
  archivePrefix = {arXiv},
  primaryClass  = {math.MG}
}
\end{document}